\documentclass[11pt,letterpaper]{amsart}

\usepackage{amsmath,amssymb,latexsym, amsbsy, amsfonts, amscd, amsthm, mathrsfs, xy,comment, xcolor}
\usepackage{multirow}
\usepackage{hyperref}
\usepackage{enumitem, empheq, url}
\usepackage{tikz-cd}
\usepackage{rotating}
\usepackage[bibtex-style]{amsrefs}
\usepackage{ltablex}
\usepackage{array}
\usepackage{mathtools}

\DeclareMathAlphabet\euscr{U}{eus}{m}{n}
\xyoption{all}

\newcommand{\longhookrightarrow}{\lhook\joinrel\longrightarrow}
\newcommand{\longtwoheadrightarrow}{\relbar\joinrel\twoheadrightarrow}

\theoremstyle{plain}

\newtheorem{theorem}{Theorem}[section]
\newtheorem{conjecture}[theorem]{Conjecture}
\newtheorem{prop}[theorem]{Proposition}
\newtheorem{corollary}[theorem]{Corollary}

\newtheorem{lemma}[theorem]{Lemma}

\theoremstyle{definition}

\newtheorem{remark}[theorem]{Remark}
\newtheorem{example}[theorem]{Example}

\long\def\symbolfootnote[#1]#2{\begingroup
\def\thefootnote{\fnsymbol{footnote}}\footnote[#1]{#2}\endgroup} 

\DeclareMathOperator{\WD}{W}
\DeclareMathOperator{\Ann}{Ann}
\DeclareMathOperator{\KS}{SKu}

\DeclareMathOperator{\cont}{cont}
\def\bl{{\Sigma_{*}}}
\def\alg{{\mathrm{alg}}}
\def\an{{\mathrm{an}}}

\def\GL{{\bf GL}}

\def\A{\mathbf{A}}
\def\I{\mathbf{I}}
\def\cA{\mathcal{A}}
\def\eA{\euscr{A}}

\def\sD{\mathscr{D}}

\def\sgn{\mathrm{sgn}}
\def\N{\mathrm{N}}
\def\1{\mf{1}}

\DeclareMathOperator{\Fitt}{Fitt}
\DeclareMathOperator{\Aug}{Aug}

\DeclareMathOperator{\Ind}{Ind}

\DeclareMathOperator{\nr}{nr}

\DeclareMathOperator{\Frac}{Frac}

\DeclareMathOperator{\cyc}{cyc}

\DeclareMathOperator{\cond}{cond}

 \DeclareMathOperator{\Norm}{Norm}
\DeclareMathOperator{\rec}{rec} \DeclareMathOperator{\Div}{Div}
\DeclareMathOperator{\Hom}{Hom} \DeclareMathOperator{\End}{End}

\DeclareMathOperator{\sign}{sign}
\DeclareMathOperator{\coker}{coker} 
\DeclareMathOperator{\ord}{ord}

\DeclareMathOperator{\cusps}{cusps}
 \DeclareMathOperator{\real}{Re}
 
 \DeclareMathOperator{\Cl}{Cl}
\DeclareMathOperator{\lcm}{lcm} 
 
 \DeclareMathOperator{\Frob}{Frob}

\DeclareMathOperator{\tr}{tr}

\DeclareMathOperator{\Ext}{Ext}

\newcommand{\bpsi}{\pmb{\psi}}

\newcommand{\mat}[4]{\begin{pmatrix}{#1} & {#2} \\ {#3} & {#4} \end{pmatrix}}
\newcommand{\mint}{\times\!\!\!\!\!\!\!\int}

\newcommand{\mf}{\mathfrak }

\newcommand{\mscr}{\mathscr }

\def\fa{\mathfrak{a}}
\def\fn{\mathfrak{n}}
\def\fp{\mathfrak{p}}
\def\fq{\mathfrak{q}}

\def\fg{\mathfrak{g}}
\def\fm{\mathfrak{m}}
\def\ft{\mathfrak{t}}

\def\fO{\mathfrak{O}}

\def\fK{\mathfrak{K}}
\def\fl{\mathfrak{l}}
\def\fP{\mathfrak{P}}

\def\fh{\mathfrak{h}}
\def\fm{\mathfrak{m}}

\def\fd{\mathfrak{d}}

\def\T{\mathbf{T}}
\def\Z{\mathbf{Z}}
\def\F{\mathbf{F}}
\def\Q{\mathbf{Q}}
\def\C{\mathbf{C}}
\def\G{\mathbf{G}}

\def\B{\mathbf{B}}
\def\bO{\mathbf{O}}

\def\bdf{\begin{defn}}
\def\edf{\end{defn}}

\def\cH{\mathcal{H}}
\def\cD{\mathcal{D}}
\def\cO{\mathcal{O}}

\def\cC{\mathcal{C}}

\def\fb{\mathfrak{b}}
\def\fc{\mathfrak{c}}

\def\Gal{{\rm Gal}}

\def\ab{{\rm ab}}

\def\cF{{\mathcal F}}

\def\ram{\text{ram}}
\def\ab{\text{ab}}

\def\sL{{\mscr L}}

\begin{document}
\baselineskip 14.25 pt

\title{Brumer--Stark Units and Explicit Class Field Theory}
\author{Samit Dasgupta \\ Mahesh Kakde}

\maketitle

\begin{abstract}
Let $F$ be a totally real field of degree $n$ and $p$ an odd prime.  We prove the $p$-part of the integral Gross--Stark conjecture for the Brumer--Stark $p$-units living in CM abelian extensions of $F$.  In previous work, the first author showed that such a result implies an exact $p$-adic analytic formula for these Brumer--Stark units up to a bounded root of unity error, including a ``real multiplication'' analogue of Shimura's celebrated reciprocity law from the theory of Complex Multiplication.  In this paper we show that the Brumer--Stark units, along with $n-1$ other easily described elements (these are simply square roots of certain elements of $F$) generate the maximal abelian extension of $F$.  We therefore obtain an unconditional construction of the maximal abelian extension of any totally real field, albeit one that involves $p$-adic integration for infinitely many primes $p$.

Our method of proof of the integral Gross--Stark conjecture is a generalization of our previous work on the Brumer--Stark conjecture.  We apply Ribet's method in the context of group ring valued Hilbert modular forms.  A key new construction here is the definition of a Galois module $\nabla_{\!\sL}$ that incorporates an integral version of the Greenberg--Stevens $\sL$-invariant into the theory of Ritter--Weiss modules.  This allows for the reinterpretation of Gross's conjecture as the vanishing of the Fitting ideal of $\nabla_{\!\sL}$.  This vanishing is obtained by constructing a quotient of $\nabla_{\!\sL}$ whose Fitting ideal vanishes using the Galois representations associated to cuspidal Hilbert modular forms.
\end{abstract}

\tableofcontents

\section{Introduction}

Our motivation in this paper is explicit class field theory, i.e.\ the explicit analytic construction of the maximal abelian extension of a number field $F$.  Let $F$ be a totally real number field.  Up to a bounded root of unity, we prove an explicit $p$-adic analytic formula for certain elements (Brumer--Stark $p$-units) that we show generate, along with other easily described elements, the maximal abelian extension of $F$ as we range over all primes $p$ and all conductors $\fn \subset \cO_F$.  To demonstrate the simplest possible novel case of these formulas, in \S\ref{s:compute}  we present example computations of narrow Hilbert class fields of real quadratic fields generated by our elements; complete tables of hundreds of such calculations are given in \cite{dasweb}.

The $p$-adic formula for Brumer--Stark units that we prove was conjectured by the first author, collaborators, and others over a series of previous papers (\cite{thesis}, \cite{dd}, \cite{chapthesis}, \cite{chapcomp}, \cite{das}, \cite{ds}).  
It was proven in \cite{das} that under a mild assumption denoted $(*)$ below, 
 our formula for Brumer--Stark $p$-units is implied by
 the $p$-part of a conjecture of Gross on the relationship between Brumer--Stark $p$-units and the special values of $L$-functions in towers of number fields \cite{gross2}*{Conjecture 7.6}.  Note that the assumption $(*)$ excludes only finitely many primes $p$ for a given totally real field $F$ (a subset of those dividing the discriminant of $F$).
The conjecture of Gross is often referred to as the ``integral Gross--Stark conjecture" or ``Gross's tower of fields conjecture,"
and the proof of the $p$-part of this conjecture takes up the bulk of the paper.

We prove the $p$-part of the integral Gross--Stark conjecture by applying Ribet's method, which was first established in his groundbreaking paper \cite{ribet}.
We apply Ribet's method in the context of group ring valued families of Hilbert modular forms as employed in \cite{wiles} and developed in our previous work \cite{dk}. 

The main new feature in the present paper that goes beyond our work in \cite{dk} is to incorporate an integral group-ring version of the Greenberg-Stevens $\sL$-invariant.  In this way, we generalize from considering just the ``value'' of the $L$-function to the ``derivative'' of the $L$-function (in Gross's integral group ring sense).  After defining an appropriate generalized group ring $R_\sL$ in which this integral Greenberg--Stevens $\sL$-invariant lives, we  construct a Ritter--Weiss module $\nabla_\sL$ associated to the $\sL$-invariant.  We calculate the Fitting ideal of $\nabla_\sL$ using the Galois representations associated to group ring valued modular forms.  The connection between $\sL$-invariants, families of modular forms, and Galois representations was pioneered by the work of Greenberg and Stevens on the Mazur--Tate--Teitelbaum conjecture \cite{gs}.  The application of these ideas toward the rational Gross--Stark conjecture (\cite{gross}*{Conjecture 2.12}) was introduced in \cite{ddp} and developed in \cite{dkv}.  The current construction is a strong integral refinement of those prior works.  The methods of this paper are similar to those used by Atsuta and Kataoka to give a near complete proof of the Equivariant Tamagawa Number Conjecture for the minus part of the Tate motive associated to CM abelian extensions of totally real fields \cite{ak}.

We now describe our results in greater detail.

\subsection{The Brumer--Stark conjecture}

Let $F$ be a totally real field of degree $n$ over $\Q$.  Let $H$ be a finite abelian extension of $F$ that is a CM field.  Write $G = \Gal(H/F)$.   Let $S$ and $T$ denote finite nonempty disjoint sets of places of $F$ such that $S$ contains the set $S_{\infty}$ of real places and the set $S_{\ram}$ of finite primes ramifying in $H$. 
Associated to any character $\chi \colon G \longrightarrow \C^*$  one has the Artin $L$-function
\begin{equation} \label{e:ls} 
L_S(\chi, s) = \prod_{\fp \not \in S} \frac{1}{1 - \chi(\fp) \N\fp^{-s}}, \qquad \real(s) > 1, 
\end{equation}
and its ``$T$-smoothed" version
\begin{equation} \label{e:lst}
L_{S,T}(\chi, s) = L_S(\chi, s) \prod_{\fp \in T} (1 - \chi(\fp)\N\fp^{1-s}). 
\end{equation}
The function $L_{S,T}(\chi, s)$ can be analytically continued to a holomorphic function on the complex plane. These $L$-functions can be packaged together into Stickelberger elements 
\[ 
\Theta_S^{H/F}(s), \quad \Theta_{S,T}^{H/F}(s) \in \C[G] 
\] 
defined by (we drop the superscript $H/F$ when unambiguous)
\[ 
\chi(\Theta_S(s)) = L_S(\chi^{-1}, s), \qquad \chi(\Theta_{S,T}) = L_{S,T}(\chi^{-1}, s) \qquad \text{ for all } \chi \in \hat{G}. 
\]

A classical  theorem of Siegel \cite{siegel}, Klingen \cite{klingen} and Shintani \cite{shintani} states that  $\Theta_S := \Theta_S(0)$ lies in $\Q[G]$.  This was refined by Deligne--Ribet \cite{dr} and Cassou-Nogu\`es \cite{cn}, who proved that under a certain mild technical condition on $T$ (which is discussed in (\ref{e:drcond}) below and which we assume holds for the remainder of the paper), we have $\Theta_{S,T}:=\Theta_{S,T}(0) \in \Z[G].$  

The following conjecture  stated by Tate is known as the Brumer--Stark conjecture. Let $\fp \not\in S \cup T$ be a prime of $F$ that splits completely in $H$.  Let $U_{\fp}^- \subset H^*$ denote the group of elements $u$ satisfying  $|u|_v = 1$ for all places $v$ of $H$ not lying above $\fp$, including the complex places.  Let $U_{\fp, T}^- \subset U_\fp^-$ denote the subgroup of elements such that $u \equiv 1 \pmod{\fq \cO_H}$ for all $\fq \in T$. 

\begin{conjecture}[Tate--Brumer--Stark, \cite{tate}] \label{c:bs}  Fix a prime $\fP$ of $H$ above $\fp$.  There exists an element  $u_\fp \in U_{\fp, T}^-$ such that
\begin{equation} \label{e:bs}
 \ord_G(u_\fp) := \sum_{\sigma \in G} \ord_{\fP}(\sigma(u_\fp)) \sigma^{-1} = \Theta_{S,T} 
 \end{equation}
in $\Z[G]$.
\end{conjecture}

In previous work \cite{dk}, we proved this conjecture away from $2$, i.e.\ over $\Z[1/2]$.  In forthcoming work \cite{dksw}, we will prove the conjecture at $2$ and thereby complete the proof.  (Even without the result at $p=2$, one could carry around an unspecified power of $2$ in various results in this paper, and our applications to explicit class field theory would not change.)

\begin{theorem}   \label{t:bs}
The Brumer--Stark Conjecture holds.
\end{theorem}

\subsection{The Integral Gross--Stark conjecture} \label{s:igs}

Let $\fp$ be as above and write $S_\fp = S \cup \{\fp\}$.
Let $L$ denote a finite abelian CM extension of $F$ containing $H$ that is ramified over $F$ only at the places in $S_\fp$.  Write $\fg = \Gal(L/F)$ and $\Gamma = \Gal(L/H)$, so $\fg/\Gamma \cong G$.  Let $I$ denote the relative augmentation ideal associated to $\fg$ and $G$, i.e.\ the kernel of the canonical projection
\[ 
 \Aug_G^{\fg} \colon \Z[ \fg] \longtwoheadrightarrow \Z[G]. 
\]
Then $ \Theta_{S_\fp,T}^{L/F}   $ lies in $I$, since its image under $\Aug_G^{\fg}$ is 
\begin{equation} \label{e:euler}
 \Theta_{S_\fp,T}^{H/F}= \Theta_{S,T}^{H/F} (1 - \Frob(H/F, \fp)) = 0, \end{equation}
as $\fp$ splits completely in $H$.  Intuitively, if we view $ \Theta_{S_\fp,T}^{L/F}$ as a function on the ideals of $\Z[\fg]$, equation (\ref{e:euler}) states that this function ``has a zero" at the ideal $I$; the value of the ``derivative" of this function at $I$ is simply the image of $ \Theta_{S_\fp,T}^{L/F} $ in $I/I^2$.  Gross provided a conjectural algebraic interpretation of this derivative as follows.
Denote by \begin{equation} \label{e:recp}
 \rec_\fP \colon H_\fP^* \longrightarrow \Gamma 
 \end{equation} the composition of the inclusion $H_\fP^* \longhookrightarrow \A_H^*$ with the global Artin reciprocity map \[ \A_H^* \longtwoheadrightarrow \Gamma. \]
Throughout this article we adopt Serre's convention \cite{serre} for the reciprocity map. Therefore $\rec(\varpi^{-1})$ is a lifting to $G_{\fp}^{\ab}$ of the Frobenius element on the maximal unramified extension of $F_{\fp}$ if $\varpi \in F_{\fp}^*$ is a uniformizer. 

\begin{conjecture}[Gross, \cite{gross2}*{Conjecture 7.6}] \label{c:grosstower} Define
\begin{equation} \label{e:recg}
\rec_G(u_\fp) = \sum_{\sigma \in G} (\rec_\fP\sigma(u_\fp)-1)  \tilde{\sigma}^{-1} \in I/I^2, \end{equation}
where $\tilde{\sigma} \in \fg$ is any lift of $\sigma \in G$.  Then
\[ \rec_G(u_\fp) \equiv  \Theta_{S_\fp,T}^{L/F}  \] 
in $I/I^2$.
\end{conjecture}

Let $p$ denote the rational prime below $\fp$, and assume that $p \neq 2$.
Our first main result is  the $p$-part of Gross's conjecture.

\begin{theorem} \label{t:maingross} Let $p$ be an odd prime and suppose that $\fp$ lies above $p$. 
Gross's Conjecture~\ref{c:grosstower} holds in $(I/I^2) \otimes \Z_p$.
 \end{theorem}

\subsection{An Exact Formula for Brumer--Stark units}

Building off the $p$-adic Gross--Stark conjecture and applying the methods introduced by Darmon in \cite{darmon}, the 
first author proposed an exact formula for Brumer--Stark units in his Ph.D.\ thesis \cite{thesis}, published  jointly with Darmon in \cite{dd}.
The setting for this conjecture was that of a real quadratic ground field $F$, a prime $\fp$ of the form  $\fp = p \cO_F$ for a rational prime $p$, and  a ring class field extension $H/F$.
Afterward, a sequence of works generalized and refined this conjecture to the case of arbitrary totally real fields $F$ and finite primes $\fp$ that split completely in a CM abelian extension $H$ (\cite{chapthesis}, \cite{chapcomp}, \cite{das}, \cite{ds}).  Further details on this history are given in \S\ref{s:hilbert}.  See also the analogous works \cite{rs}, \cite{cd} in the archimedean context.  

For expositional purposes in this introduction, let us describe the shape of these conjectures in a special case mentioned above: we assume that the rational prime $p$ is inert in $F$ and that $\fp = p\cO_F$.  Then for each integral ideal $\fa$ of $F$ relatively prime to the primes in $S$ and $T$,
 one may define a $\Z$-valued measure $\nu_{\fa, S, T}$ on $\cO_\fp^*$ in terms of special values of Shintani zeta-functions (or, alternatively, in terms of periods of Eisenstein series).  The following is a special case of \cite{das}*{Conjecture 3.21} or \cite{ds}*{Conjecture 6.1}.
 
\begin{conjecture}\label{c:thesis} Let $\sigma = \Frob(H/F, \fa)$.  We have the following exact analytic formula for the associated conjugate of the Brumer--Stark unit $u_{\fp}$:
\begin{equation} \label{e:thesis}
\sigma(u_\fp) = p^{\zeta_{S,T}(\sigma, 0)} \mint_{\cO_\fp^*} x \ d\nu_{\fa, S, T}(x) \quad \text{ in } \quad F_\fp^*. \end{equation}
\end{conjecture}

Here we view $\sigma(u_\fp)$ as an element of $F_\fp^*$ via $H \subset H_\fP \cong F_\fp$, where $\fP$ is the prime above $\fp$ appearing in Conjecture~\ref{c:bs}.  Conjecture~\ref{c:thesis} implies not only the algebraicity of the $p$-adic integrals in (\ref{e:thesis}), but a ``Shimura reciprocity law" in which the geometric action of a generalized class group on equivalence classes of ideals $\fa$ is identified with the Galois action of $\Gal(H/F)$ on the units $u_\fp$ (see \cite{das}*{Conjecture 3.21}).  In this way, Conjecture~\ref{c:thesis} can be viewed as a part of a theory of ``real multiplication" in parallel with the classical theory of complex multiplication that is governed by Shimura's celebrated reciprocity law.

Our second main result, which applies in the general case (i.e.~without assuming $\fp$ is inert over $\Q$), is the following.

\begin{theorem} \label{t:thesis}  Let $p$ denote the rational prime below $\fp$.
Suppose that
\begin{quote} $(*)$ \ \ \ \  $p$ is odd and $H \cap F(\mu_{p^\infty}) \subset H^+$, the maximal totally real subfield of $H$.
\end{quote}
Then equation (\ref{e:thesis}), or more precisely its generalization (\ref{e:genform}) to the general setting,  holds up to multiplication by a root of unity in $F_\fp^*$.
\end{theorem}

Since $\fp$ splits completely in $H$ while $p$ is totally ramified in $\Q(\mu_{p^n})$ for all $n$, condition $(*)$ can fail for an odd prime $p$ only if $p$ is ramified in $F$.  The condition therefore eliminates only finitely many $p$, a subset of those dividing the discriminant of $F$.

An analogue of Conjecture~\ref{c:thesis} where $F$ is replaced by the function field of a smooth projective algebraic curve over $\F_q$---a far simpler setting because of the explicit class field theory afforded by the theory of Drinfeld modules---was proven by the first author and Miller in \cite{dm}.

We conclude this discussion by bringing attention to the beautiful concurrent work of Darmon, Pozzi, and Vonk, who prove a version of Conjecture~\ref{c:thesis} in the setting that $F$ is a real quadratic field and the rational prime $p$ is inert in $F$ \cite{dpv}.  While their work also employs the deformations of $p$-adic modular forms and their associated Galois representations, they work with deformations in ``vertical" $p$-adic towers as opposed to the ``horizontal" tame deformations applied in this paper.

\subsection{Explicit Class Field Theory} \label{s:iecft}

A celebrated theorem of Kronecker and Weber states that the maximal abelian extension of the field $\Q$ of rational numbers is obtained by adjoining all roots of unity.

\begin{theorem}[Kronecker--Weber] \label{t:kw} 
We have $\Q^{\ab} = \bigcup_{n \ge 1} \Q(e^{2 \pi i/n}).$
\end{theorem}

The roots of unity can be viewed analytically as the special values of the analytic function $e^{2 \pi i x}$ at rational arguments, or algebraically as the set of torsion points of the group scheme $\G_m$.  The theory of complex multiplication provides a similar description of $F^{\ab}$ when $F$ is a quadratic imaginary field.  

\begin{theorem} \label{t:cm}  Let $F$ be a quadratic imaginary field.  Let $E$ denote an elliptic curve with complex multiplication by the ring of integers $\cO_F$ and let $w$ denote the Weber function.  We have \[ F^{\ab} = \bigcup_{n \ge 1} F(j(E), w(E[n])).  \]
\end{theorem}

See the elegant exposition~\cite{ghate} for the definition of the Weber function $w$ and a proof of Theorem~\ref{t:cm}.  From the analytic perspective, the modular functions $j$ and $w$ take on the role of the exponential function $e^{2 \pi i x}$ in the case $F=\Q$; from the algebraic perspective,  the abelian variety $E$ takes on the role of the group scheme $\G_m$.

As we now describe, Theorem~\ref{t:thesis} can be viewed as an explicit class field theory for totally real fields $F$ in the spirit of
  of Theorems~\ref{t:kw} and~\ref{t:cm}.  Our approach here is inspired by Stark's discussion of the application of his conjectures to Hilbert's 12th problem \cite{stark3}.
For each nonzero ideal $\fn \subset \cO_F$, pick a prime ideal $\fp(\fn) \subset \cO_F$ whose image in the narrow ray class group of $F$ of conductor $\fn$ is trivial.  Choose $\fp(\fn)$ such that the rational prime $p$ below it satisfies $(*)$.
Let $u_{\fp(\fn)}$ denote the Brumer--Stark unit for the narrow ray class field $F(\fn)$ of conductor $\fn$.
Let 
\[ S_\fn = \left\{ \sigma(u_{\fp(\fn)}) \colon \sigma \in \Gal(F(\fn)/F) \right\}. \]
Finally, let $\{\alpha_1, \dotsc, \alpha_{n-1}\}$ denote any elements of $F^*$ whose signs in \[ \{\pm 1\}^{n}/ (-1, \dotsc, -1) \] under the real embeddings of $F$ form a basis for this $\Z/2\Z$-vector space. In \S\ref{s:hilbert}, we prove the following.

\begin{theorem}  \label{t:h12}
The maximal abelian extension of $F$ is generated by \[ \sqrt{\alpha_1}, \dotsc, \sqrt{\alpha_{n-1}} \]together with the elements of $S_\fn$ as $\fn$ ranges over all nonzero  ideals $\fn \subset \cO_F$:
\[ F^{\ab}  = {\dot{\bigcup}}_{\fn} \ F(S_{\fn}) \ \ \dot\cup \ \ F(\sqrt{\alpha_1}, \dotsc, \sqrt{\alpha_{n-1}}), \]
where $\dot\cup$ denotes compositum of fields.
\end{theorem}

Since Theorem~\ref{t:thesis} gives an exact formula for the elements in $S_\fn$, we obtain via Theorem~\ref{t:h12} an effective method of generating the maximal abelian extension of any totally real field.     See Remark~\ref{r:rou} for a discussion regarding the root of unity ambiguity in Theorem~\ref{t:thesis}.

The integrals in (\ref{e:thesis}) are explicitly computable and yield a practical method of generating class fields. In Section~\ref{s:compute}  we provide examples of narrow Hilbert class fields of real quadratic fields generated by this analytic formula.

\bigskip

Any discussion of explicit class field theory would be incomplete without mentioning Hilbert's 12th problem.  In his famed address at the ICM in Paris in 1900, Hilbert wrote \cite{hilbert}: 
\begin{quote}
``The theorem that every abelian number field arises from the realm of rational numbers by the composition of fields of roots of unity is due to Kronecker..."

``Since the realm of the imaginary quadratic number fields is the simplest after the realm of rational numbers, the problem arises, to extend Kronecker's theorem to this case..."

``Finally, the extension of Kronecker's theorem to the case that, in the place of the realm of rational numbers or of the imaginary quadratic field, any algebraic field whatever is laid down as the realm of rationality, seems to me of the greatest importance.  I regard this problem as one of the most profound and far-reaching in the theory of numbers and of functions."
\end{quote}

At the time of Hilbert's lecture, Theorem~\ref{t:cm} was not  fully proved.  Over the previous decades, the explicit construction of class fields of imaginary quadratic fields using special values of modular functions was the topic of great study, particularly by Kronecker (1823--1891), who called this program his {\em Jugendtraum} (``dream of youth").  Nevertheless, it was already clear by 1900 that analytically constructing class fields of ground fields {\em other than} $\Q$ or quadratic imaginary fields represented a substantially more difficult problem. 
Hilbert was somewhat specific in the type of explicit class field theory he envisioned: he asked for the definition of certain complex analytic functions whose special values or transformation properties yield the maximal abelian extension of $F$.  Certainly, as  $p$-adic numbers had only recently been invented at the time of Hilbert's lecture, the constructions of this paper do not fit neatly into his framework.
We refer the reader to Schappacher's delightful exposition on Hilbert's 12th problem for further background and historical details \cite{schapp}.

\bigskip

As a final note on explicit class field theory in the introduction, we recall that if $E/F$ is a quadratic CM extension, then ``most" of the field $E^{\ab}$ is obtained by taking the compositum of fields of moduli of appropriate CM-motives with $F^{\ab}$.  More precisely, the following result, which combines Corollary 1.5.2, Theorem 2.1, and Corollary 2.3 of \cite{wei}, is known.

\begin{theorem} Let $E$ be a CM field with maximal totally real subfield $F$.  Let $M_E$ be the field generated over $E$ by the fields of moduli of all CM-motives with Hodge cycle structure whose reflex fields are contained in $E$.  Equivalently, $M_E$ is the field obtained by adjoining to $E$ the fields of moduli of all polarised abelian varieties of CM-type, whose reflex fields are contained in $E$, and their torsion points.  Then the compositum $M_E F^{\ab}$ is a subfield of $E^{\ab}$ such that $\Gal(E^{\ab}/M_E F^{\ab})$ has exponent dividing 2 (it is an infinite product of $\Z/2\Z$'s unless $F=\Q$, in which case it is trivial).
\end{theorem}

 As a result, we find that the construction of $F^{\ab}$ given in Theorem~\ref{t:h12} together with the field $M_E$ yields a description of most of the maximal abelian extension of $E$.

\subsection{Summary of Proof}

We conclude the introduction by describing the proof of Theorem~\ref{t:maingross}, the $p$-part of the integral Gross--Stark conjecture where $p$ is the prime below $\fp$.  We always assume that $p$ is odd. Recall that we are given a tower of fields $L/H/F$ with $\fg = \Gal(L/F)$ and $G = \Gal(H/F)$.  
Let 
\[ R = \Z_p[\fg]^- = \Z_p[\fg] / (\sigma + 1), \qquad \overline{R} = \Z_p[G]^-, \]
where $\sigma$ is the complex conjugation of $\fg$.
 Let $I = \ker(R \longtwoheadrightarrow \overline{R})$.

As a first step, we alter the smoothing set $S$ and depletion set $T$ as follows.  Define
\begin{align*}
\Sigma &= \{v \in S \colon v \mid p \infty \}, \\
\Sigma_\fp &= \Sigma \cup \{\fp\}, \\
\Sigma' &= \{v \in S: v  \nmid p\infty \} \cup T.
\end{align*}
There exists an associated modified Stickelberger element $ \Theta_{\Sigma,\Sigma'}^H \in \Z_p[G]$ and  modified Brumer--Stark unit $u_{\fp}^{\Sigma, \Sigma'} \in 
U_{\fp,T}^-$ such that $\ord_G(u_{\fp}^{\Sigma, \Sigma'}) = \Theta_{\Sigma,\Sigma'}^H.$  We show in Lemma~\ref{l:equiv} that the modified Gross--Stark congruence
\begin{equation} \label{e:mod}
 \rec_G(u_{\fp}^{\Sigma, \Sigma'}) \equiv \Theta_{\Sigma_\fp, \Sigma'}^L \pmod{I^2} 
 \end{equation}
implies the original one.

\bigskip

In order to prove (\ref{e:mod}), we recall the Ritter--Weiss modules $\nabla_{\Sigma}^{\Sigma'}(H)$ and $\nabla_{\Sigma_\fp}^{\Sigma'}(L)$ and the relationship between these modules.  Versions of these modules were originally defined by Ritter and Weiss in the foundational work \cite{rw}.  An alternate approach was studied in \cite{bks} and \cite{burns}.  Here we apply our previous work \cite{dk}, which builds upon the original definition of Ritter--Weiss.

In order to make a connection with the Greenberg--Stevens theory, we introduce an $R$-algebra $R_\sL$ that is generated over $R$ by an element $\sL$ that plays the role of the analytic 
$\sL$-invariant, i.e. the ``ratio" between $\Theta_L = \Theta_{\Sigma_\fp, \Sigma'}^L$ and $\Theta_H = \Theta_{\Sigma, \Sigma'}^H$.  In some sense, $R_\sL$ is the canonical $R$-algebra in which such a ratio can be considered:
\[ R_\sL = R[\sL] / (\Theta_H \sL - \Theta_L, \sL I, \sL^2, I^2). \]
See \S\ref{s:rldef} for an expanded discussion motivating this definition. An important feature of the ring $R_\sL$ that we prove in \S\ref{s:rldef} is that the canonical $R$-algebra map $R/I^2 \longrightarrow R_\sL$ is injective.

\bigskip

 We next define a generalized Ritter--Weiss module $\nabla_{\!\sL}$ over the ring $R_{\sL}$.
   This module can be viewed as a gluing of the modules $\nabla_{\Sigma}^{\Sigma'}(H)$ and $\nabla_{\Sigma_\fp}^{\Sigma'}(L)$ over $R_\sL$.  By its defining properties, the module
 $\nabla_{\Sigma}^{\Sigma'}(H)$ ``sees" the modified Brumer-Stark unit  $u_{\fp}^{\Sigma, \Sigma'}$, while the module 
 $\nabla_{\Sigma_\fp}^{\Sigma'}(L)$ sees the  image of $u_{\fp}^{\Sigma, \Sigma'}$ under $\rec_G$.  By fiat, the ring $R_\sL$ sees the Stickelberger elements 
$ \Theta_H$ and $\Theta_L$ (or more precisely, a stand-in $\sL$ for their ``ratio").  The upshot is that $\nabla_{\!\sL}$ will be large if (\ref{e:mod}) holds, and will be small if (\ref{e:mod}) fails.  This notion of size is made precise via the theory of Fitting ideals.

We will show in \S\ref{s:gcfi} that the Fitting ideal $\Fitt_{R_\sL}(\nabla_{\!\sL})$ is generated by the element 
\[ \rec_G(u_{\fp}^{\Sigma, \Sigma'}) -  \Theta_{\Sigma_\fp, \Sigma'}^L \in I/I^2.  \]
In view of the injectivity of $R/I^2 \longrightarrow R_\sL$, in order to prove that (\ref{e:mod}) holds it suffices to prove that
\begin{equation} \label{e:ifit}
 \Fitt_{R_\sL}(\nabla_{\!\sL}) = 0. 
  \end{equation}

We prove (\ref{e:ifit}) and thereby conclude the proof of Theorem~\ref{t:maingross} as follows. 
We first give a characterization of $\nabla_{\!\sL}$  via Galois cohomology as in \cite{dk}*{Lemma A.8}.
Using this characterization, we show that for an $R_\sL$-module $M$, an $R_\sL$-module surjection \begin{equation} \label{e:nablam} \nabla_{\!\sL} \longtwoheadrightarrow M \end{equation} is equivalent to a Galois cohomology class $\kappa \in H^1(G_F, M)$ satisfying certain local conditions.  For each positive integer $m$, we construct such a Galois cohomology class in an $R_\sL$-module $M$ satisfying $ \Fitt_{R_\sL}(M) \subset (p^m)$.  The surjection (\ref{e:nablam}) then implies 
\[  \Fitt_{R_\sL}(\nabla_{\!\sL}) \subset  \Fitt_{R_\sL}(M) \subset (p^m). \]  Since this is true for all $m$, we obtain $(\ref{e:ifit})$ as desired.

\bigskip

The $R_\sL$-module $M$ and cohomology class $\kappa$ are constructed using group ring valued families of Hilbert modular forms. 
Write $\fn$ for the conductor of $L/F$.  For simplicity in this introduction we assume that all primes of $F$ above $p$ divide $\fn$ and that there is more than one such prime.
Let $\Lambda$ be a finite free $\Z_p$-module.
For a positive integer $k$,
let $M_k(\fn, \Lambda)$ (respectively $S_k(\fn, \Lambda)$) denote the space of Hilbert modular forms (respectively, cusp forms) over $F$ of level $\Gamma_1(\fn)$ over $\Lambda$. The group $S_k(\fn, \Lambda)$ is endowed with the action of a Hecke algebra $\tilde{\T}$ that is generated over $\Z_p$ by the Hecke operators $T_\fq$ for $\fq \nmid \fn$, $U_\fq$ for $\fq \mid p$, and the diamond operators $S(\fm)$ for each integral ideal 
$\fm$ relatively prime to $\fn$.  

The main result of \S\ref{s:cusp} is the definition of an $R$-module $\Lambda$, a submodule $J \subset \Lambda$, and the construction of a  cusp form $f \in S_k(\fn, \Lambda)$ satisfying  the following properties.
\begin{itemize}
\item The weight $k$ is congruent to 1 modulo $(p-1)p^m$.
\item The quotient $\Lambda/J$ has the structure of a faithful $R/(I^2, p^m)$-module.
\item  Let $G_\fn^+$ denote the narrow ray class group of conductor $\fn$, and let \[ \bpsi \colon G_\fn^+ \longtwoheadrightarrow \fg \subset R^* \] denote the canonical character.  The form $f$ has nebentypus $\bpsi$, i.e.\  $S(\fm) f = \bpsi(\fm)f$ for $(\fm, \fn) = 1$.
\item The form $f$ is Eisenstein modulo $J$ in the sense that $T_\fq(f) \equiv (1 + \bpsi(\fq)) f  \pmod{J}$ for  $\fq \nmid \fn$, and 
$U_\fq(f) \equiv f  \pmod{J}$ for all $\fq \mid p, \fq \neq \fp$.

\item Modulo $J$, the operator $(1 - U_\fp)$ acting on the Hecke span of the form $f$ satisfies the relations governing the element $\sL \in R_\sL$, e.g. \[ (1 - U_\fp) \Theta_H f \equiv \Theta_L f \pmod{J}. \]
\end{itemize}

The form $f$ allows for the definition of an $R_{\sL}/p^m$-algebra $W$ and an $R$-algebra homorphism \[ \varphi\colon \tilde{\T} \longrightarrow W\] such that 
\begin{itemize}
\item $\varphi(T_\fq) = 1 + \bpsi(\fq)$ for  $\fq \nmid \fn$ 
\item  $\varphi(U_\fq) = 1$ for all $\fq \mid p, \fq \neq \fp$, and 
\item $\varphi(1 - U_\fp) = \sL$.
\end{itemize}
  Furthermore the algebra $W$ is large enough that the structure map \begin{equation} \label{e:rslw}
  R_{\sL}/p^m \longrightarrow W\end{equation} is an injection.

There is a Galois representation $\rho \colon G_F \longrightarrow \GL_2(\Frac(\tilde{\T}))$ that satisfies the usual conditions; in particular $\rho$ is unramified outside $\fn$ and $\tr(\rho(\sigma_\fq)) = T_\fq$ for $\fq \nmid \fn$, where $\sigma_\fq$ denotes a Frobenius at $\fq$.  We choose a basis for $\rho$ such that $\rho(\tau)$ is diagonal for a certain well-chosen $\tau \in G_F$ that restricts to the complex conjugation in $\fg$.  Writing \[ \rho(\sigma) = \begin{pmatrix} a(\sigma) & b(\sigma) \\ c(\sigma) & d(\sigma)\end{pmatrix} \] in this basis, 
we let $\tilde{\B}$ denote the $\tilde{\T}$-module generated by $b(\sigma)$ for $\sigma \in G_F$, together with certain other elements $x_\fq \in \Frac(\tilde{\T})$ for $\fq \mid p$.  We then define $\tilde{B} = \tilde{\B}/(p^m, \ker \varphi)\tilde{\B}$.

Standard methods in the theory of pseudorepresentations then allow for the definition of a cohomology class $\kappa \in H^1(G_F, \tilde{B}(\bpsi^{-1}))$.  We show that the class $\kappa$ satisfies the necessary local conditions to yield a surjection
$ \nabla_{\!\sL} \longtwoheadrightarrow \tilde{B}(\bpsi^{-1})$ as in (\ref{e:nablam}).  These local calculations are based on the fact that the form $f$ is ordinary at primes $\fq \mid p$, and $p$-ordinary forms have Galois representations with prescribed shapes when restricted to decomposition groups at primes dividing $p$.  The elements $x_\fq$ mentioned above arise from this local calculation.

Finally, we use the methods of \cite{dk}*{Theorem 9.10} along with the crucial injection (\ref{e:rslw}) to prove the desired result 
\[  \Fitt_{R_\sL}(\tilde{B}(\bpsi^{-1})) \subset (p^m). \]
One modification of the description above in the text is that we break $R$ into a product of components $A$ and work over $A_\sL = A \otimes_R R_\sL$.  This reduction is useful in our constructions with modular forms.

Also, the case where the prime $\fp$ is the only prime of $F$ above $p$ requires special consideration and is handled in \S\ref{s:onep}.  Let us try to motivate why this case is unique.  In this setting, when $k \equiv 1 \pmod{(p-1)p^m}$, the Eisenstein series $E_k(1, \chi)$ and $E_k(\chi, 1)$ are congruent modulo $p^m$ for any odd character of $G$.  The intersection of the associated Hida families in weight 1 causes a singularity in the spectrum of the Hida Hecke algebra at that point.  The deformations of modular forms at the weight 1 point corresponding to $E_1(1, \chi_\fp)$  are more complicated for this reason.  
  We are not able to construct a group ring family of modular forms defined over a module endowed with an action of the ring $R_\sL$ in this case. 
  
   Instead, we provide a different strategy for proving the congruence (\ref{e:mod}) when there is only one prime of $F$ above $p$.  In a sense, the argument is easier in this case, though it relies on previous significant results, some of which were themselves proved using Hilbert modular forms.  
   We are able to deform up the cyclotomic tower of $F$ without altering the depletion set $\Sigma_\fp = S_\infty \cup \{\fp\}$ since this tower is ramified only at the prime $\fp$.  Let $F_m$ denote the $m$th layer of the tower, let $\fg_m = \Gal(LF_m/F)$, and let $R_m = \Z_p[\fg_m]^-$.  
   We introduce a ring $R_{X,m}$ analogous to $R_\sL$, except that $X$ now represents the ratio between $\Theta_L$ and the derivative of the $p$-adic $L$-function of $H/F$ at $s=0$, denoted $\Theta_H'$.  We then define a homomorphism $R_m \longrightarrow R_{X,m}$ and show that in order to prove (\ref{e:mod}), it suffices to prove that
 \begin{equation} \label{e:nxi}
  \Fitt_{R_m} (\nabla_{\Sigma_\fp}^{\Sigma'}(LF_m)_{R_m} \otimes_{R_m} R_{X,m}) = 0. \end{equation}
This reduction is dependent on the proof of the rational rank 1 Gross--Stark conjecture by the first author with Darmon and Pollack in \cite{ddp} and Ventullo in \cite{v}.  Another ingredient of the reduction proof is the nonvanishing of the first derivatives of $p$-adic $L$-functions deduced by combining the  rational rank 1 Gross--Stark conjecture with the transcendence result of Brumer--Baker on the nonvanishing of algebraic linear combinations of $p$-adic logarithms of algebraic numbers (see Theorem~\ref{t:nv}). 

   We prove that (\ref{e:nxi}) holds and thereby conclude the proof by applying the formula
 \begin{equation} \label{e:dkni}
  \Fitt_{R_m} (\nabla_{\Sigma_\fp}^{\Sigma'}(LF_m)_{R_m}) = (\Theta_{\Sigma_\fp, \Sigma'}^{LF_m/F}(0))  \end{equation}
    proved in our previous work \cite{dk}.  We deduce (\ref{e:nxi}) from (\ref{e:dkni})
 using the explicit construction of the Deligne--Ribet $p$-adic $L$-function as an integral.

This concludes our summary discussion of the proof of Theorem~\ref{t:maingross}.

\subsection{Acknowledgements}

We would first like to thank Takamichi Sano for very helpful discussions while this research was performed.  Sano explained to us in detail his work with Burns and Kurihara on the link between the Leading Term Conjecture and the integral Gross--Stark conjecture \cite{bks}.  While we do not apply their results directly, the philosophy presented in {\em loc.\ cit.} was very influential in our definition of the module $\nabla_{\!\sL}$ and the proof of Theorem~\ref{t:mfig} in \S\ref{s:gcfi}.

We would also like to thank Brian Conrad, Henri Darmon, Jesse Silliman, and Jiuya Wang for helpful discussions.  We thank James Milne for pointing out the application to CM fields mentioned in \S\ref{s:iecft} and in particular the article \cite{wei}.

We are extremely grateful to Max Fleischer and Yijia Liu, two undergraduate students of the first author at Duke University, who wrote the code to produce the computations in \S\ref{s:compute}.  Their code and complete tables are given in \cite{dasweb}.

Both authors would like to thank the hospitality of the Mathematical Research Institute of Oberwolfach, where certain aspects of this research were completed.  
The first author was supported by grants DMS-1600943 and DMS-1901939 from the National Science Foundation during the completion of this project. The second author was support by SUPRA grant SPR/2019/000422 from the Science \& Engineering Research Board during the completion of this project.

\section{Explicit Class Field Theory}

\subsection{An Exact Formula for Brumer--Stark units}

In \cite{das}, we proposed a conjectural exact $p$-adic analytic formula for the Brumer--Stark units $u_\fp$.
We briefly recall the shape of this formula. Let $\fn \subset \cO_F$ denote a nonzero ideal and let $F(\fn)$ denote the narrow ray class field  of $F$ associated to the conductor $\fn$.  Let $H$ be the maximal CM subfield of $F(\fn)$ in which the prime $\fp$ splits completely.  Let $f$ denote the order of $\fp$ in the narrow ray class group of conductor $\fn$, and write $\fp^f = (\pi)$ for a totally positive element $\pi \in 1 + \fn$.  We also assume that the set $T$ contains a prime whose norm is a rational prime in $\Z$.  

Next we let $\sD$ denote a Shintani domain, as defined in \cite{das}*{Proposition 3.7}.  Let $\cO_\fp$ denote the completion of $\cO_F$ at $\fp$, and let $\bO = \cO_\fp - \pi \cO_\fp$.  Let $\fb \subset \cO_F$ denote an integral ideal that is relatively prime to $\fn$.  Associated with all this data, we have:
\begin{itemize}
\item A totally positive unit $\epsilon(\fb, \sD, \pi) \in \cO_F^*$
 congruent to 1 modulo $\fn$, 
defined in \cite{das}*{Definition 3.17}.
\item A $\Z$-valued measure $\nu(\fb, \sD)$ on $\cO_\fp$, defined in \cite{das}*{eqn. (21)} using Shintani's theory of simplicial zeta functions.
\end{itemize}
We then proposed:
\begin{conjecture}[\cite{das}*{Conjecture 3.21}]  \label{c:shin} Let $\sigma_\fb \in \Gal(H/F)$ denote the Frobenius element associated with $\fb$.  Let $\fP$ and $u_\fp$ be as in Conjecture~\ref{c:bs}, and consider $H$ as a subfield of $F_\fp$ via $H \subset H_\fP \cong F_\fp$.
We then have
\begin{equation} \label{e:genform}
 \sigma_{\fb}(u_\fp) = \epsilon(\fb, \cD, \pi) \cdot \pi^{\zeta_{S,T}(F(\fn)/F, \fb, 0)} \mint_{\bO} x \ d\nu(\fb, \sD, x)) \in F_\fp^*. 
 \end{equation}
\end{conjecture}

We stress that the exponent $\zeta_{S,T}(F(\fn)/F, \fb, 0)$ and the measure $\nu(\fb, \sD)$ may be computed explicitly using 
Shintani's formulas.  If we may take $\pi$ to be a rational integer (e.g.\ if $p$ is inert in $F$ and $\pi = p^f$), then the unit $\epsilon(\fb, \sD, \pi)$ is equal to 1 and the formula simplifies as described in the introduction (\ref{e:thesis}).
See \cite{comp}*{Proposition 3.2} for an explicit formula for the measure $\nu(\fb, \sD)$ in this setting when $F$ is real quadratic.
One of the main theorems of \cite{das} is the following:

\begin{theorem}[\cite{das}*{Theorem 5.18}]  Suppose that condition $(*)$ holds. Then the $p$-part of Conjecture~\ref{c:grosstower} for all $L$ implies  Conjecture~\ref{c:shin}
 up to multiplication by a root of unity in $F_\fp^*$. \label{t:maindas}
\end{theorem}

In other words, Theorem~\ref{t:maingross} implies Theorem~\ref{t:thesis}.

\bigskip

We conclude this section by discussing the history of Conjecture~\ref{c:shin} and its various manifestations.  The first conjecture of this form appeared in the 2004 Ph.D.\ thesis of the first author \cite{thesis}, which was published in the form \cite{dd}.  The setting for this conjecture was that of a real quadratic field $F$ and a prime $p$ that is inert in $F$.  Furthermore $H$ was taken to be a CM ring class field extension of $F$.  In that paper, the measure $\nu$ was defined as a specialization of the {\em Eisenstein cocycle} obtained by integrating Eisenstein series on the complex upper half plane.  The resulting interplay between complex and $p$-adic integration was inspired by Darmon's theory of integration on $\cH_p \times \cH$ \cite{darmon}.
 
 This first construction was generalized by Chapdelaine in his 2007 Ph.D. thesis \cite{chapthesis} to the context where $F$ is still a real quadratic field and $p$ is an inert prime, but $H$ is an arbitrary CM abelian extension of $F$ in which $\fp = p \cO_F$ splits completely.  Chapedelaine's construction applied integration of more general Eisenstein series than considered in \cite{dd}.  It was published in the form \cite{chapcomp}.  
 
 Next, as described above, the first author applied Shintani's theory of simplicial zeta functions in order to expand the earlier constructions to cover the general case: 
 $F$ is any totally real field, $H$ is a CM abelian extension, and $\fp$ is a finite prime of $F$ that splits completely in $H$  \cite{das}.  
 The equivalence between the constructions of \cite{dd} and \cite{das} in the setting of the former article was established in \cite{das}*{\S8}, though we note an error in this argument that was observed and corrected by Chapdelaine in \cite{chapquebec}.  In  Chapdelaine's paper the equivalence between the constructions was generalized to include that of \cite{chapcomp}.
 
 Notably absent from the article \cite{das}, however, is the cohomological perspective of the earlier works.  More recently, several articles have appeared that have reestablished the cohomological underpinnings of the $p$-adic formula (\ref{e:genform}) in the general case.

In \cite{pcsd}, Charollois and the first author reconsidered the construction of the Eisenstein cocycle by Sczech using conditionally convergent sums, and proved an integrality result yielding another construction of the measure $\nu$ in the general case.  The cohomological approach was applied to Shintani's method independently by Spiess in \cite{spiess} and by the first author in joint work with Charollois and M.~Greenberg in \cite{cdg}.  The application of these constructions to an exact $p$-adic analytic formula for $u_\fp$ was given in joint work of the first author with Spiess \cite{ds}*{\S6}.  The equivalence of the cohomological approach of {\em loc.\ cit.} with the formula (\ref{e:genform}) above is the subject of current work by Honnor, building on the Ph.D.\ thesis of Tsosie \cite{tsosie}.
We hope to prove directly that Theorem~\ref{t:maingross} implies the conjecture of \cite{ds}  (and in fact generalize to the higher rank case) in future joint work with Spiess.

\subsection{The Maximal Abelian Extension of $F$} \label{s:hilbert}

The goal of this section is to prove Theorem~\ref{t:h12}. The following is \cite[Remarque 2.3 Chap. IV]{tatebook}.  

\begin{lemma}  \label{l:ugen}
Let $F$ be a totally real field and suppose that $H/F$ is a cyclic CM extension such that the finite prime $\fp \subset \cO_F$ splits completely in $H$.  Let $S = S_\infty \cup S_{\ram}(H/F)$ be minimal for the extension $H/F$.  Then $H = F(u_\fp)$.
\end{lemma}

\begin{proof}  Since $H/F$ is cyclic, there exists a {\em faithful} character of $G = \Gal(H/F)$. 
Such a character is necessarily odd, i.e. $\chi(\text{complex conjugation}) = -1$.
 By \cite{tatebook}*{Pg. 25}, we have
\[ \ord_{s=0} L_{S,T}(\chi, 0) = \# \{ v \in S\colon \chi(G_v) = 1 \} = 0 \]
since $\chi$ is faithful and $S$ is minimal (in particular $S$ contains no place $v$ that splits completely in $H$).  Hence $L_{S,T}(\chi,0) \neq 0$.  On the other hand, applying $\chi^{-1}$ to (\ref{e:bs}) yields
\[ L_{S,T}(\chi, 0)  = \sum_{\sigma \in G} \ord_{\fP}(\sigma(u_\fp)) \chi(\sigma).
\]
If $\tau \in G$ fixes $u_\fp$, then 
\begin{align*} \chi(\tau)  \cdot L_{S,T}(\chi, 0)  & = \sum_{\sigma \in G} \ord_{\fP}(\sigma(u_\fp)) \chi(\tau\sigma) \\
& = \sum_{\sigma \in G} \ord_{\fP}(\sigma\tau^{-1}(u_\fp)) \chi(\sigma) \\
&=  \sum_{\sigma \in G} \ord_{\fP}(\sigma(u_\fp)) \chi(\sigma) \\
&=  L_{S,T}(\chi, 0).
\end{align*}
Since  $L_{S,T}(\chi, 0)$ is nonzero, we conclude that $\chi(\tau) = 1$.  Since $\chi$ is faithful, we have $\tau=1$.  By Galois theory, we  have $H = F(u_{\fp})$.
\end{proof}

\begin{lemma} \label{l:cmcyclic} Let $F$ be a totally real field and suppose that $H/F$ is an abelian CM extension.  Then $H$ is the compositum of its CM subfields $H' \subset H$ containing $F$ such that $\Gal(H'/F)$ is cyclic.
\end{lemma}

\begin{proof} Let $c \in G = \Gal(H/F)$ denote the unique complex conjugation.
Let $G'$ be a subgroup of $G$. The fixed field $H^{G'}$ is CM if and only if $c \not \in G'$.  The desired result then follows from the following elementary fact in group theory:  if $G$ is a finite abelian group and $c \in G$ is any nontrivial element, then the intersection of all subgroups $G' \subset G$ such that $G/G'$ is cyclic and $c \not\in G'$ is trivial.  We leave the proof of this fact as an exercise.
\end{proof}

To pass from CM fields to arbitrary abelian extensions, we  consider the sign homomorphism
\[ \sgn\colon F^* \longrightarrow \{\pm 1\}^n, \qquad x \mapsto (\sign(\sigma_1(x)), \dotsc, \sign(\sigma_n(x)). \]
Here the $\sigma_i$ denote the $n$ real embeddings of $F$.
We say that elements $\alpha_1, \dotsc, \alpha_{n}$ are {\em sign spanning} if the images of these elements under $\sgn$
generate the abelian group $\{\pm 1\}^n.$

\begin{lemma}  \label{l:sqrts}
Let $F$ be a totally real field and let $K/F$ be a finite abelian extension. Suppose that $\alpha_1, \dotsc, \alpha_{n}$ are a sign spanning set of elements of $F^*$ such that $\sqrt{\alpha_i} \in K$ for all $i=1, \dotsc, n$.  Let $H$ denote the maximal CM  extension of $F$ contained in $K$.  Then $K = H(\sqrt{\alpha_1}, \dotsc, \sqrt{\alpha_{n}}).$
 \end{lemma}

\begin{proof}  The assumption of the existence of the $\alpha_i$ implies that $K$ contains a CM extension of $F$, e.g.\ $F(\sqrt{\alpha})$ for a product $\alpha$ of the $\alpha_i$ that is totally negative.  Note also that the notion of ``maximal CM extension" is well-defined, since the compositum of CM fields is again CM.

Let $\tau_i \in G = \Gal(K/F)$ denote the complex conjugation corresponding to any complex place of $K$ above the real place $\sigma_i$ of $F$.  Let $G' \subset G$ denote the exponent 2 subgroup generated by the elements $\tau_i/\tau_j$ for all $i, j = 1, \dotsc, n$.
The fixed field $H = K^{G'}$ has at most one complex conjugation, namely the image of any $\tau_i$ in $G/G'$.  The field $H$ will be CM if this image is nontrivial, and totally real if the image is trivial.  The latter situation happens if and only if there is an equality of the form
\begin{equation} \label{e:tauprod} \prod_{i \in E} \tau_i = 1 \text{ in } G \end{equation}
for some subset $E \subset \{1, \dotsc, n\}$ of odd size.  In fact the assumption of the lemma implies that no such relation can occur for any nonempty $E$.  Indeed, choose an appropriate product $\alpha$ of the $\alpha_i$ such that $\sigma_j(\alpha)< 0 $ for some $j \in E$ and $\sigma_{j'}(\alpha) > 0$ for all $j' \neq j$.  Then $\left(\prod_{i \in E} \tau_i\right)(\sqrt{\alpha}) = - \sqrt{\alpha}$, precluding the possibility that (\ref{e:tauprod}) holds.

It remains to show that $K = H(\sqrt{\alpha_1}, \dotsc, \sqrt{\alpha_{n}}).$  Any element $\tau \in G$ that fixes the subfield $H(\sqrt{\alpha_1}, \dotsc, \sqrt{\alpha_{n}})$ pointwise must in particular fix $H$ pointwise, and hence must lie in $G'$.  Such an element therefore has the form $\tau = \prod_{i \in E} \tau_i$ for some subset $E \subset \{1, \dotsc, n\}$.  Yet we just showed that no such nontrivial product can fix every $\sqrt{\alpha_i}$.  It follows that $\tau = 1$, and the desired result follows from Galois theory.
\end{proof}

We now prove Theorem~\ref{t:h12}, whose statement we recall.
For each nonzero ideal $\fn \subset \cO_F$, pick a prime ideal $\fp(\fn) \subset \cO_F$ whose image in the narrow ray class group of $F$ of conductor $\fn$ is trivial.  Choose $\fp(\fn)$ such that the rational prime $p$ below it satisfies $(*)$.
Let $u_{\fp(\fn)} \in U_{\fp(\fn),T}$ denote the Brumer--Stark element for the narrow ray class field $F(\fn)$ of conductor $\fn$ (i.e.~for the maximal CM extension of $F$ contained in $F(\fn)$).  Define
\begin{equation} \label{e:sfndef}
 S_\fn = \left\{ \sigma(u_{\fp(\fn)}) \colon \sigma \in \Gal(F(\fn)/F) \right\}. 
 \end{equation}
Finally, let $\alpha_1, \dotsc, \alpha_{n-1} \in F^*$ such that $-1, \alpha_1, \alpha_2, \dotsc, \alpha_{n-1}$ are sign spanning.

\begin{theorem} \label{t:h12main}
The maximal abelian extension of $F$ is generated by \[ \sqrt{\alpha_1}, \dotsc, \sqrt{\alpha_{n-1}} \] together with the elements of $S_\fn$ as $\fn$ ranges over all nonzero  ideals $\fn \subset \cO_F$:
\begin{equation} \label{e:h12}
 F^{\ab}  = {\dot{\bigcup}}_{\fn} \ F(S_{\fn}) \ \ \dot\cup \ \ F(\sqrt{\alpha_1}, \dotsc, \sqrt{\alpha_{n-1}}), 
 \end{equation}
where $\dot\cup$ denotes compositum of fields.
\end{theorem}

\begin{proof}  Let $L$ denote the field on the right side of (\ref{e:h12}).  It is clear that $L$ is an abelian extension of $F$.  We must show that if $K$ is any finite abelian extension of $F$, then $K \subset L$.  After replacing $K$ by $K(\sqrt{-1}, \sqrt{\alpha_1}, \dotsc, \sqrt{\alpha_{n-1}})$, we may assume that $K$ contains $\sqrt{-1}, \sqrt{\alpha_1}, \dotsc, \sqrt{\alpha_{n-1}}$.  If $H$ denotes the maximal CM extension of $F$ contained in $K$, then $\sqrt{-1} \in H$ and hence Lemma~\ref{l:sqrts} implies that $K = H(\sqrt{\alpha_1}, \dotsc, \sqrt{\alpha_{n-1}})$.  It therefore suffices to show that $H \subset L$.  By Lemma~\ref{l:cmcyclic} we may assume that $H$ is a cyclic CM extension of $F$.

Let $\fn \subset \cO_F$ denote the conductor of $H$.  The minimal set $S$ for the narrow ray class field $F(\fn)$ is the same as that for $H \subset F(\fn)$, namely the union of $S_\infty$ with the set of primes dividing $\fn$.  We write $\fp = \fp(\fn)$. By  \cite{tatebook}*{Prop IV.3.5, Pg. 92}, the Brumer--Stark units for the extensions $F(\fn)$ and $H$ are related by
\begin{equation} \label{e:normrel}
 u_{\fp}(H) = \N_{F(\fn)/H}(u_{\fp}(F(\fn))) = \prod_{\sigma \in \Gal(F(\fn)/H)} \sigma(u_{\fp}(F(\fn))). \end{equation}
It follows from (\ref{e:normrel}) and the definition of $S_\fn$ that $u_{\fp}(H)\in L$.  Lemma~\ref{l:ugen} implies that
\[ H = F(u_{\fp}(H)) \] and hence $H \subset L$.  The result follows.
\end{proof}

\begin{remark}  \label{r:rou} We should remark on the root of unity ambiguity in Theorem~\ref{t:maindas} with respect to providing an explicit formula for the elements in $S_{\fn}$ appearing on the right side of (\ref{e:h12}).  Since this root of unity necessarily lies in $F_{\fp(n)}^*$, its order divides $\N\fp(\fn) - 1$.  Therefore we may simply raise our elements $u_{\fp(\fn)}$ to the power $\N\fp(\fn) - 1$ and obtain an unconditional exact equality, with elements that still satisfy the necessary properties for Theorem~\ref{t:h12main} (it is easy to adapt Lemma~\ref{l:ugen} to replace $u_\fp$ by $u_\fp^m$ for any positive integer $m$).   Alternatively, since roots of unity always generate abelian extensions, we can simply adjoin all roots of unity to the right side of (\ref{e:h12}) and ignore any possible root of unity error in the formula (\ref{e:genform}) for the elements of $S_\fn$.
\end{remark}

\subsection{Computations} \label{s:compute}

In this section we present some computations of the units $u_\fp$ calculated using the formula  (\ref{e:genform}).  We consider the simplest possible  case beyond $F = \Q$: we let $F = \Q(\sqrt{D})$ be a real quadratic field with discriminant $D$, let $p$ be a rational prime that is inert in $F$ (so $\fp = p \cO_F$), and let $H$ be the narrow Hilbert class field of $F$.  
The set $S$ is taken to equal $S_\infty$, the minimal possible set in this setting.  The set $T$ is taken to contain a single prime $\fq$ such that $\N\fq = \ell$ is a prime not equal to $p$.

The code to perform these computations was written by Max Fleischer and Yijia Liu, two undergraduate students of the first author at Duke University.  Their algorithm  closely follows the paper \cite{comp}.  One important difference is as follows.  In {\em loc.\ cit.}, a formal divisor $\sum_{d \mid N} n_d[d] \in \Div(\Z)$ is used to smooth the zeta values. The conditions $\sum_{d \mid N} n_d = 0$ and $\sum_{d \mid N} n_d d = 0$ are imposed.  For this reason, computations are performed with the simplest possible nonzero such divisor, namely $2[1] - 3[2] + [4]$.
In the present paper, we use the set $T$ to smooth our zeta values, which corresponds to setting $N = \ell$ and using the divisor $[\ell]-\ell[1]$.
 Note that this divisor does not satsify the condition $\sum_{d \mid N} n_d = 0$.  It turns out that this condition is unnecessary to apply the general algorithm of {\em loc.\ cit.}; only minor modifications are necessary (see \cite{dasweb}).

In each case, formula (\ref{e:genform}) is used to compute the image of $u_\fp$ and all of its conjugates over $F$ in $F_p^*$ to 100 $p$-adic digits.  The elementary symmetric polynomials of these conjugates are then calculated and, after scaling by the appropriate power of $p$ to achieve integrality, recognized as elements of $\cO_F$ using a standard nearest lattice vector algorithm.  This allows for the computation of the minimal polynomial of $u_\fp$ over $F$.  The computed 100 $p$-adic digits were enough to recognize the minimal polynomials in each case.  

We stress that the field $H$ itself is never fed into the program; all computations take place within $F$ and its completion $F_p$.  After the fact, it was verified that in each case the minimal polynomials listed split over a CM abelian extension of $F$ unramified at all finite primes, and hence must be contained in the narrow Hilbert class field $H$ of $F$.  The splitting field of the minimal polynomial of $u_\fp$ is in fact precisely $H$, as implied by a suitable modification of Lemma~\ref{l:ugen} (using the fact that $L_{S,T}(\chi, 0) \neq 0$ for every odd character $\chi$ since here $S=S_\infty$ contains no finite primes).

The programs to generate these results were written in SageMath and executed on a Jupyter Notebook using the kernel SageMath 9.0.  The code is available at 
\cite{dasweb}.  
This web page also describes a typo (sign error) in one equation in \cite{comp} that was discovered by Fleischer and Liu.
Below we give three interesting examples of the computations. 

\begin{example} $D = 221, p=3, \ell = 5, G = \Gal(H/F) \cong \Cl^+(F) \cong \Z/4\Z$.  The values $\ord_p(\sigma(u_\fp))$ for $\sigma \in G$, i.e.\ the partial zeta values $\zeta_{S,T}(\sigma, 0)$, are $\pm 3, \pm 15$.  The minimal polynomial of $u_\fp$ over $F$ is computed to be:
\begin{align*}
 X^4 + &  \left (\frac{-423812}{3^{13}} + \frac{71680\sqrt{D}}{3^{15}}\right)X^3 + \left(\frac{76348630}{3^{18}} + \frac{-5218304\sqrt{D}}{3^{16}}\right)X^2  \\
 & + \left(\frac{-423812}{3^{13}} + \frac{71680\sqrt{D}}{3^{15}}\right)X + 1. 
 \end{align*}
\end{example}

\begin{example} $D = 321, p=7, \ell = 5, G \cong \Z/6\Z$.  The values $\ord_p(\sigma(u_\fp))$ for $\sigma \in G$ are $\pm 1, \pm 3, \pm 7$.  The computed minimal polynomial of $u_\fp$ over $F$ is:
\begin{align*}
X^6 + & \left(\frac{55935}{2 \cdot 7^{7 }} + \frac{-63891\sqrt{D}}{2 \cdot 7^{7 }}\right)X^5 + \left(\frac{1062148509}{2 \cdot 7^{10 }} + 
\frac{2960001\sqrt{D}}{2 \cdot 7^{10 }}\right)X^4 
\\
+ & \left(\frac{-49244921}{2 \cdot 7^{10 }} + \frac{-279429993\sqrt{D}}{2 \cdot 7^{11}}
\right)X^3 + \cdots + 1.
 \end{align*}
Note that the minimal polynomials of Brumer--Stark units are always palindromic, since $u_\fp^{-1}$ is the complex conjugate of $u_\fp$.  Hence the coefficient of $X^2$ above equals the coefficient of $X^4$ and the coefficient of $X$ equals that of $X^5$.
\end{example}

\begin{example} $D = 897, p=5, \ell=7, G \cong \Z/4\Z \times \Z/2\Z$. The values $\ord_p(\sigma(u_\fp))$ are $\pm 7, \pm 9, \pm 11, \pm 21$.  The computed minimal polynomial of $u_\fp$ over $F$ is: 
\begin{align*}
X^8  + & \left(\frac{2549757626558363}{2\cdot 5^{21}}+\frac{1416002374557\sqrt{D}}{2\cdot 5^{21}}\right)X^7 \\
+ & \left(\frac{51143699935554731498041}{5^{32}}+\frac{56709030111424864533\sqrt{D}}{5^{31}}\right)X^6 \\
+ & \left(-\frac{11738117897361345671334368371}{2\cdot 5^{41}}+\frac{4935116278645813872967514931\sqrt{D}}{2\cdot 5^{41}}\right)X^5 \\
+ & \left(-\frac{4489586764048071498962140328642159}{5^{48}}+\frac{49988908282076855221482\sqrt{D}}{5^{34}}\right)X^4 \\ 
+ & \cdots +1
\end{align*}
\end{example}

Complete tables for all fundamental discriminants $D < 1000$ whose associated narrow class field is CM, comprising hundreds of similar examples, are given in \cite{dasweb}.

We conclude this section by noting that in his 2007 undergraduate senior thesis at Harvard University, Kaloyan Slavov computed an example where the ground field $F$ is a totally real cubic field and the conductor $\fn$ is nontrivial.  We refer to \cite{slavov}*{\S8.6.1} for details.

\section{Algebraic preliminaries}

The rest of the paper is concerned with proving Theorem~\ref{t:maingross}.  Let us recall the setup.
We are given a totally real field $F$ and a finite abelian CM extension  $H$. 
Let $S$ be a finite set of places of $F$ containing $S_\ram(H/F) \cup S_\infty$. Let
$T$ be a finite set of primes  of $F$ disjoint from $S$ satisfying the following condition of Deligne--Ribet:
\begin{equation}
\label{e:drcond}
  \parbox{\dimexpr\linewidth-5em}{
    \strut
$T$ contains a prime of residue characteristic greater than $[F:\Q]+1$, or two primes of different residue characteristic.
    \strut
  }
  \end{equation}
  This condition is useful because it implies the following:
\[
  \parbox{\dimexpr\linewidth-5em}{
    \strut
Let $T_H$ denote the set of primes of $H$ above those in $T$.  The group of roots of unity $\zeta \in \mu(H)$ such that $\zeta \equiv 1 \!\!\!\!\!\pmod{\fq}$ for all $\fq \in T_H$ is trivial.
    \strut
  }
\]

We fix a prime $\fp \subset \cO_F$ not in $S \cup T$ that splits completely in $H$ and write $S_\fp = S \cup \{\fp\}$.  
We consider another finite abelian CM extension $L/F$ containing $H$ and unramified outside $S_\fp$. Write  \[ \fg = \Gal(L/F), \quad \Gamma = \Gal(L/H), \quad G = \Gal(H/F) \cong \fg/\Gamma . \]
Let $p$ denote the rational prime contained in $\fp$. We assume that $p$ is odd.
Let \[ R = \Z_p[\fg]^-, \qquad  \overline{R} = \Z_p[G]^-, \qquad I = \ker(R \longrightarrow \overline{R}). \]
By Theorem~\ref{t:bs}, there is a unique $u_{\fp} \in U_{\fp,T}^- \otimes \Z_p$ such that \[ \ord_G(u_{\fp}) =  \Theta_{S,T}^H. \]
Our goal is to prove the following congruence, the $p$-part of Gross's conjecture:
\[ \rec_G(u_{\fp}) \equiv \Theta_{S_\fp,T}^L \pmod{I^2}. \]

\subsection{Altering Depletion and Smoothing Sets}

Following \cite{dk}, define
\begin{align*}
\Sigma &= \{v \in S  \colon v \mid p \infty \}, \\
\Sigma_\fp &= \Sigma \cup \{\fp\}, \\
\Sigma' &= \{v \in S \colon v  \nmid p\infty \} \cup T.
\end{align*}
It will be very convenient for us to replace the sets $S, S_\fp$, and $T$ by the sets $\Sigma, \Sigma_\fp$, and $\Sigma'$, respectively.  In this section, we make this replacement precise and prove that it suffices to prove the result in this context.

There is a Stickelberger element $\Theta_{\Sigma, \Sigma'}^H \in \Q[G]^-$ defined by the property 
\[ \chi(\Theta_{\Sigma, \Sigma'}^H) = L_{\Sigma, \Sigma'}(\chi^{-1}, 0) \] for every odd character $\chi$ of $G$, where
$L_{\Sigma, \Sigma'}(\chi, s)$ is defined by (\ref{e:ls}) and (\ref{e:lst}) along with the convention that $\chi(\fq)=0$ if $\chi$ is ramified at $\fq$.  We show in \cite{dk}*{Remark 3.6} that $\Theta_{\Sigma, \Sigma'}^H \in \overline{R}$.

As we explain, the results of \cite{dk} imply that there exists a unique \[ u_{\fp}^{\Sigma, \Sigma'} \in U_{\fp, T}^- \otimes \Z[\frac{1}{2}] \] such that \[ \ord_G(u_{\fp}^{\Sigma, \Sigma'}) = \Theta_{\Sigma,\Sigma'}^H. \]  
Let $\#$ denote the involution on $\overline{R}$ induced by $g \mapsto g^{-1}$ for $g \in G$.
Define \[ \KS_p^T(H/F) = (\Theta_{\Sigma, T}^H)^\# \prod_{\mathclap{\substack{v \in S_{\ram}(H/F) \\v \nmid p}}} (\N I_v, 1 - \sigma_v \N I_v/\#I_v) \subset \overline{R}. \] 
Here $I_v$ is the inertia subgroup at $v$. 
Remark 3.6 of {\em loc.\,cit.} implies that 
\[ \Theta_{\Sigma,\Sigma'}^H \in \KS_p^T(H/F)^\#. \] 
Equation (35) of {\em loc.\,cit.} then implies that $\Theta_{\Sigma,\Sigma'}^H$ annihilates $\Cl^T(H)_p^-$.  The annihilation of the class represented by $\fp$ is equivalent to the existence of the desired $u_{\fp}^{\Sigma, \Sigma'}$.

\begin{lemma}  \label{l:equiv}
To prove Gross's conjecture, it suffices to show that in each setup as above, we have the ``modified Gross conjecture":
\begin{equation} \label{e:mgc}
 \rec_G(u_{\fp}^{\Sigma, \Sigma'}) \equiv \Theta_{\Sigma_\fp, \Sigma'}^L \pmod{I^2}. \end{equation}
\end{lemma}
\begin{proof}  More generally if $\Sigma \subset J \subset S$ and $J' = \Sigma' \setminus J$, $J_\fp = J \cup \{\fp\}$, then we will show that 
\begin{equation} \label{e:uj}
 \rec_G(u_{\fp}^{J, J'}) \equiv \Theta_{J_\fp, J'}^L \pmod{I^2}. 
 \end{equation}
 The desired result is the case $J=S$.
We proceed by induction on $\#(J \setminus \Sigma)$.  The base case $J = \Sigma$ is given. For the inductive step, fix $J \supset \Sigma$ and let $v \in S \setminus J$.  Then 
$v \nmid p\infty$.  We write \[ 
J_v = J \cup \{v\}, \qquad
J_{v, \fp} = J \cup \{v, \fp\}, \qquad J'_v = J' - \{v\}. \]
Then \begin{equation} \label{e:theta}
 \Theta_{J_{v, \fp}, J'_v}^L = \Theta_{J_\fp, J'}^L + \frac{\N v - 1}{\#I_v} \sigma_v^{-1} \N I_v \Theta_{J_\fp, J_v'}^L. \end{equation}
Note that $\# I_v \mid \N v - 1 $ in $\Z_p$, since $v \nmid p$.
Both terms on the right side of (\ref{e:theta}) lie in $\text{SKu}_p^T(L/F)^\#$ and hence lie in $\overline{R}$.  By the inductive hypothesis, we have (\ref{e:uj}).
Furthermore, if we write $\overline{I}_v$ for the image of $I_v$ in $G$ and \[ I_G(v) = \ker(\Z_p[\fg/I_v]^- \longrightarrow \Z_p[G/\overline{I}_v]^-), \] then by induction we have
\begin{equation} \label{e:uj2}
 \rec_{G/\overline{I}_v}(u_{\fp}^{J, J_v'}) \equiv \Theta_{J_\fp, J_v'}^{L^{I_v}} \pmod{I_G(v)^2}.  
 \end{equation}
Here $u_{\fp}^{J, J_v'} \in U_\fp(H^{\overline{I}_v})^-$ 
satisfies \[ \ord_{G/\overline{I}_v}(u_{\fp}^{J, J_v'}) = 
\Theta_{J, J_v'}^{H^{\overline{I}_v}}. \]
 Now $x \mapsto x \cdot \N I_v$ yields a map $\Z_p[\fg/I_v] \rightarrow \Z_p[\fg]$ sending $I_G(v)^2 \rightarrow I^2$, so from (\ref{e:uj2}) we deduce
 \begin{equation}  \rec_G((u_{\fp}^{J, J_v'})^{\N I_v}) \equiv  \N I_v \Theta_{J_\fp, J_v'}^{L} \pmod{I^2}.  \label{e:uj3}
 \end{equation}
The desired result
\[  \rec_G(u_{\fp}^{J_v, J_v'}) \equiv \Theta_{J_{v,\fp}, J_v'}^L \pmod{I^2}  \]
now follows by combining (\ref{e:uj}) and (\ref{e:uj3}), using (\ref{e:theta}).  We must simply note that
\begin{align}
\Theta_{J_v, J_v'}^H &= \Theta_{J,J'}^H + \frac{\N v - 1}{\#\overline{I}_v} \sigma_v^{-1} \N \overline{I}_v \Theta_{J, J_v'}^H \\
&= \Theta_{J,J'}^H + \frac{\N v - 1}{\#{I_v}} \sigma_v^{-1} \N {I_v} \Theta_{J, J_v'}^H,
  \end{align}
which implies that $u_{\fp}^{J_v, J_v'} = u_{\fp}^{J, J'} \cdot (u_{\fp}^{J, J_v'})^{\frac{\N v - 1}{\#{I_v}} \sigma_v^{-1} \N {I_v}}.$
\end{proof}

\subsection{Removing primes above $p$ from the smoothing set} \label{s:remove}

When working with modular forms in \S\ref{s:cusp}, it will be convenient if the set $\Sigma'$ does not contain any primes above $p$.  Note that any primes above $p$ in $\Sigma'$ necessarily lie in $T$ and hence are unramifed in $L/F$.  
We give now the elementary argument, in the spirit of \cite{dk}*{\S4.1}, that shows that we can safely remove these primes.
Let 
\[ T_p = \{\fq \in T \colon \fp \mid p\}, \qquad T_0 = T \setminus T_p, \qquad  \Sigma'_0 = \Sigma' \setminus T_p. \]

\begin{lemma} \label{l:tp}  There exists $u_\fp^{\Sigma, \Sigma'_0} \in U_{\fp, T_0}^{-} \otimes {\Z_p}$ such that $ \ord_G(u_{\fp}^{\Sigma, \Sigma'_0}) = \Theta_{\Sigma,\Sigma'_0}^H$.  Furthermore the congruence
\begin{equation} \label{e:rec0}
 \rec_G(u_{\fp}^{\Sigma, \Sigma'_0}) \equiv \Theta_{\Sigma_\fp, \Sigma'_0}^L \pmod{I^2}  \end{equation}
implies the congruence  \[ \rec_G(u_{\fp}^{\Sigma, \Sigma'}) \equiv \Theta_{\Sigma_\fp, \Sigma'}^L \pmod{I^2}. 
\]
\end{lemma}

\begin{proof} 
For each $\fq \in T_p$, let $\sigma_\fq$ denote the associated Frobenius in $\fg$, and $\overline{\sigma}_\fq$ its image in $G$.
We have \begin{equation} \label{e:tp}
\Theta_{\Sigma,\Sigma'}^H =  \Theta_{\Sigma,\Sigma'_0}^H \prod_{\fq \in T_p} (1 - \overline{\sigma}_\fq \N\fq). \end{equation}
Each term in the product in (\ref{e:tp}) is congruent to 1 modulo $p$ and hence is invertible in $\Z_p[G]$.  Hence we can simply define
\[ u_\fp^{\Sigma, \Sigma'_0} =  \prod_{\fp \in T_p} (1 - \overline{\sigma}_\fp \N\fp)^{-1} (u_\fp^{\Sigma, \Sigma'}). \]
The congruence (\ref{e:rec0}) then implies
\begin{align*}
 \rec_G(u_{\fp}^{\Sigma, \Sigma'}) & \equiv 
  \Theta_{\Sigma_\fp, \Sigma'_0}^L  \prod_{\fq \in T_p} (1 - \overline{\sigma}_\fq \N\fq) \pmod{I^2} \\
  & \equiv \Theta_{\Sigma_\fp, \Sigma'}^L  \pmod{I^2}
  \end{align*}
  as desired.
\end{proof}

Applying Lemma~\ref{l:tp}, we assume for the remainder of the paper that $T$ contains no primes above $p$.
We will continue to write $T, \Sigma'$ rather than $T_0, \Sigma'_0$.

\begin{remark}  After removing the primes above $p$ from $T$, condition (\ref{e:drcond}) might no longer be satisified.  That condition was used to ensure the integrality of $\Theta_{S,T}$, which was then used to deduce the $p$-integrality of $\Theta_{\Sigma, \Sigma'}$.  As the argument in this section shows, after removing the primes above $p$ from $T$, the $p$-integrality of $\Theta_{\Sigma, \Sigma'}$ still holds. 
\end{remark}

\subsection{The ring $R_\sL$} \label{s:rldef}
In \cite{ddp}, the rank one $p$-adic Gross--Stark conjecture was interpreted as the equality of an algebraic $L$-invariant $\sL_{\alg}$ and
an  analytic $L$-invariant $\sL_{\an}$.  The analytic $\sL$-invariant is the ratio of the leading term of the $p$-adic $L$-function at $s=0$ to its classical counterpart: \begin{equation} \label{e:lan}
 \sL_{\an} = -\frac{L_p'(\chi\omega, 0)}{L(\chi, 0)}. \end{equation}
   The algebraic $L$-invariant is the ratio of the $p$-adic logarithm and valuation of the $\chi^{-1}$-component of the Brumer--Stark unit: 
   \begin{equation} \label{e:lalg}
    \sL_{\alg} = \frac{\log_p \Norm_{H_{\fP}/\Q_p}(u_\fp^{\chi^{-1}}) }{ \ord_\fP(u_\fp^{\chi^{-1}})}.  
    \end{equation}

Here and throughout the paper, $\log_p$ denotes the $p$-adic logarithm following Iwasawa's convention $\log_p(p)=0$.
There is no difficulty in defining the ratios (\ref{e:lan}) and (\ref{e:lalg}), since the quantities live in a $p$-adic field and the denominators are non-zero.  The analog of this situation in our present context is more delicate.  Let $I$ denote the kernel of the projection 
\[ 
\begin{tikzcd}
R = \Z_p[\fg]^- \ar[r] &  \overline{R} = \Z_p[G]^-. 
\end{tikzcd} 
\]
The role of the $p$-adic $L$-function is played by the Stickelberger element 
$\Theta_{\Sigma_\fp, \Sigma'}^L \in R$, and the analogue of the derivative at 0 is played by the image of $\Theta_{\Sigma_\fp, \Sigma'}^L$ in $I/I^2$.  The role of the classical $L$-function is played by the element $\Theta_{\Sigma, \Sigma'}^H \in \overline{R}$.  It is therefore not clear how to take the ``ratio" of these quantities.  Similarly, the role of the $p$-adic logarithm is played by $\rec_G(u_{\fp}^{\Sigma,\Sigma'}) \in I/I^2$ and the role of the $p$-adic valuation is played by $\ord_G(u_{\fp}^{\Sigma, \Sigma'}) \in \overline{R}$.

For this reason, we introduce an $R$-algebra $R_\sL$ that is generated by an element $\sL$ that plays the role of the analytic 
$\sL$-invariant, i.e. the ``ratio" between 
\[ 
\Theta_L = \Theta_{\Sigma_\fp, \Sigma'}^L \qquad \text{and} \qquad \Theta_H = \Theta_{\Sigma, \Sigma'}^H.
\]
We define
\begin{equation} \label{e:rldef}
 R_\sL = R[\sL] / (\Theta_H \sL - \Theta_L, \sL I, \sL^2, I^2). 
 \end{equation}
Note that $\Theta_H\sL$ is well-defined in $R[{\sL}]/\sL I$, so this definition makes sense. A key point is that the ring $R_\sL$, in which we have adjoined a ratio $\sL$ between $\Theta_L$ and $\Theta_H$, is still large enough to see  $R/I^2$.  
 
 \begin{theorem} The kernel of the structure map $R \longrightarrow R_{\sL}$ is $I^2$. \label{t:ws}
\end{theorem}

Before proving the theorem we establish some intermediate results that are important in their own right.

\begin{theorem} \label{t:gorder} For each prime $v$, let \[ \begin{tikzcd}
I(v) = \ker(R \ar[r] & \Z_p[\fg/\fg_v]^-). \end{tikzcd} \]  We have \[ \Theta_L \in \prod_{v \in \Sigma_\fp} I(v). \]
\end{theorem}

The proof of Theorem~\ref{t:gorder} uses the Ritter--Weiss modules that will be recalled in the next section.  For this reason we postpone the proof until that point.

\begin{lemma}  Suppose that $r \in R$ satisfies $\overline{r} \Theta_H = 0$ in $\overline{R}$.  Then $r \Theta_L \in I^2$. \label{l:rtt}
\end{lemma}
\begin{proof}  Let \[ e_1 = \ \ \ \ \ \  \sum_{\mathclap{L_{\Sigma, \Sigma'}(H/F, \chi, 0) =0}} \ e_\chi \] denote the idempotent of $\Frac(\overline{R})$ corresponding to the set of odd characters of $G$ at which $\Theta_H$ has a trivial zero.  Let $e_2 = 1 - e_1$ the denote the idempotent corresponding to the set of other odd  characters of $G$.  
Let $I_\Gamma$ denote the (absolute) augmentation ideal of $\Z_p[\Gamma]$.
Our goal is to prove that the image of $r \Theta_L$ in \[ I/I^2 \cong \overline{R}  \otimes I_\Gamma/I_\Gamma^2 \] vanishes
(see \cite{pcmi}*{5.2.3(b)} for this isomorphism).  Now $\overline{r} \Theta_H = 0$ implies that $\overline{r}e_2 =0$.

 Let 
$\Upsilon$ denote the kernel of the projection $\overline{R} \longrightarrow \overline{R}e_1$.  We have a short exact sequence
\begin{equation} \label{e:e1s}
\begin{tikzcd}
\Upsilon \otimes I_\Gamma/I_\Gamma^2 \ar[r, "\varpi"] & \overline{R}  \otimes I_\Gamma/I_\Gamma^2 \ar[r,"e_{1}"] &  \overline{R} e_1  \otimes I_\Gamma/I_\Gamma^2 \ar[r] & 0. \end{tikzcd}
\end{equation}
We claim that the image of $\Theta_L$ under the map denoted $e_{1}$ in (\ref{e:e1s}) vanishes.  Granting this claim for now, let us finish the proof.  The claim implies that $\Theta_L$ lies in the image of $\Upsilon \otimes I_\Gamma / I_\Gamma^2$ under the map $\varpi$ in (\ref{e:e1s}).  But $\overline{r}e_2 = 0$ implies that $\overline{r}$ annihilates $\Upsilon$.  It follows that $\overline{r} \Theta_L$ vanishes in $I/I^2$ as desired.

We now prove the claim, which states that $ \Theta_Le_1$ vanishes in $\overline{R}e_1  \otimes I_\Gamma/I_\Gamma^2$.
By Theorem~\ref{t:gorder}, we have that $\Theta_L \in \prod_{v \in \Sigma_\fp} I(v)$.  Now $\fg_\fp \subset \Gamma$, so $I(\fp) \subset I$.   We may therefore write $\Theta_L$ as a sum of elements of the form $y z$ where $y \in \prod_{v \in \Sigma} I(v)$ and $z \in I$.  Now if $y \in \prod_{v \in \Sigma} I(v)$ and $\overline{y}$ denotes the image of $y$ in $\overline{R}$, then $e_1 \overline{y} = 0$.  Indeed, it suffices to check this character by character: for an odd $\chi \in \hat{G},$ we have  $\chi(e_1) =0$ if $\chi(G_v) \neq 1$ for all $v \in \Sigma$, whereas $\chi(\overline{y}) = 0$ if 
$\chi(G_v) = 1$ for some $v \in \Sigma$.  It follows that if $z \in I$, then $e_1(y z)$ vanishes in 
$\overline{R}e_1  \otimes I_\Gamma/I_\Gamma^2$.  The desired result for $\Theta_Le_1$ follows.
\end{proof}

We can now establish the injectivity of the map $R/I^2 \longrightarrow R_{\sL}$.

\begin{proof}[Proof of Theorem~\ref{t:ws}] If the image of $a \in R$ in the polynomial ring $R[x]$ belongs to the ideal generated by $\Theta_Hx - \Theta_L, x^2, xI, I^2$, then considering the constant term implies that $a + \Theta_L r \in I^2$ for some $r \in R$. Considering the linear term implies that $\overline{r} \Theta_H = 0$ in $\overline{R}$, where $\overline{r}$ denotes the image of $r$ in $\overline{R}$.  It then follows from Lemma~\ref{l:rtt} that $a \in I^2$, as desired.
\end{proof}

In view of Theorem \ref{t:ws},
the integral Gross--Stark conjecture can be reinterpreted as an equation between algebraic and analytic $\sL$-invariants in the ring $R_\sL$.  We will show 
\begin{equation} 
 \rec_G(u_{\fp}^{\Sigma, \Sigma'})  = \sL \ord_G(u_{\fp}^{\Sigma, \Sigma'}) \text{ in } R_{\sL}. \label{e:linv} \end{equation}
Since the modified Brumer--Stark unit satsifies $\ord_G(u_{\fp}^{\Sigma, \Sigma'}) = \Theta_H$, the right side of (\ref{e:linv}) equals $\Theta_L$. The equality $\rec_G(u_{\fp}^{\Sigma, \Sigma'})  = \Theta_L$ in $R_{\sL}$ then gives the desired congruence
$\rec_G(u_{\fp}^{\Sigma, \Sigma'})  \equiv \Theta_L \pmod{I^2}$ in $R$ by the injectivity of $R/I^2 \longrightarrow R_{\sL}$.

\section{Generalized Ritter--Weiss modules} \label{s:rw}

The Galois modules introduced by Ritter and Weiss to give generalized Tate sequences play a central role in this work. Before delving into the details, we give a road map for \S\ref{s:rwdef}--\S\ref{s:gcfi}.
In  \S\ref{s:rwdef}, we recall the definition of the Ritter--Weiss module $\nabla_{\Sigma}^{\Sigma'}$, following the construction of \cite{dk} that incorporates the smoothing set $\Sigma'$.  To maintain maximal generality, we work over $\Z[\fg]$ rather than over $R =\Z_p[\fg]^-$.  In \S\ref{s:gc} we give an interpretation of $\nabla_{\Sigma}^{\Sigma'}$ in terms of Galois cohomology.  In \S\ref{s:nablal} we define a module $\nabla_{\!\sL}$ over the ring $R_{\sL}$ that incorporates the analytic $\sL$-invariant. In \S\ref{s:gcfi} we interpret the $p$-part of Gross's conjecture as the statement that the Fitting ideal of $\nabla_{\!\sL}$ over $R_{\sL}$ vanishes.
\subsection{Definition} \label{s:rwdef}

We recall the definition of $\nabla_{\Sigma}^{\Sigma'}$ from \cite{dk}*{\S A}.
We begin by choosing an auxiliary finite set of primes $S'$ of $F$ that contains $\Sigma_\fp$ and is disjoint from $\Sigma'$. Note that the places in $S' - \Sigma_\fp$ are unramified in $L$.
We furthermore assume that $S'$ is large enough so that $\Cl_{S'}^{\Sigma'}(L) = 1$, $\Cl_{S'}^{\Sigma'}(H) = 1$, and such that the union of the decomposition groups 
$\fg_v \subset \fg$ for $v \in S'$ is all of $\fg$.  The construction of $\nabla$ is  independent of the chosen auxiliary set $S'$ (see \cite{dk}*{\S A.2}).

For each place $v$ of $F$, we fix a place $w$ of $L$ above $v$. Write $\Delta \fg_v$ for the augmentation ideal of $\Z[\fg_v]$.
  Ritter--Weiss \cite{rw} define $\Z[\fg]$-modules $V_w(L)$ and $W_w(L)$ sitting in  exact sequences:
\begin{equation} \label{e:rwvwdef}
\begin{tikzcd}[row sep=tiny]
 1 \ar[r] &   L_w^* \ar[r] &  V_w(L) \ar[r] & \Delta \fg_v \ar[r] & 1  \\
  1 \ar[r] &   \cO_{w}^* \ar[r] & V_w(L) \ar[r] &  W_w(L) \ar[r] &1.
 \end{tikzcd}
\end{equation}
 Here $\cO_{w}$ denotes the completion of $\cO_L$ at $w$.
Let $U_{w} \subset \cO_{w}^*$ denote the subgroup of 1-units.   

The modules $V_w(L)$ and $W_w(L)$ are defined as follows.   Let $L_w^{\ab} \supset L_w^{\nr}$ denote the maximal abelian and unramified extensions of $L_w$, respectively.
There  are canonical short exact sequences
\[
\begin{tikzcd}[row sep=tiny]
1  \ar[r] & \WD(L_w^{\ab}/L_w) \cong L_w^*  \ar[r] & \WD(L_w^{\ab}/F_v)  \ar[r,"\pi_V"] & \fg_v  \ar[r] & 1 \\
1  \ar[r] & \WD(L_w^{\nr}/L_w) \cong \Z  \ar[r] & \WD(L_w^{\nr}/F_v)  \ar[r,"\pi_W"] & \fg_v  \ar[r] & 1,
 \end{tikzcd}
\]
where $\WD$ denotes the Weil group. 
 Let $I_V$ denote the (absolute) augmentation ideal of $\WD(L_w^{\ab}/F_v)$, and let 
$I_{\pi_V}$ denote the relative augmentation ideal corresponding to $\pi_V$.   Define $I_W$ and $I_{\pi_W}$ similarly
from the corresponding terms in the second exact sequence above.
Then \begin{equation} \label{e:vandw}
 V_w(L) = I_V / I_V I_{\pi_V}, \qquad W_w(L) = I_W / I_W I_{\pi_W}. \end{equation}

Following \cite{greither}, given a collection of $\fg_v$-modules $M_w$, 
we write   \[ \prod_{v}^\sim M_w = \prod_{v} \Ind_{\fg_v}^{\fg} M_w. \]
Define \[ 
V = \prod_{v \in S'}^\sim V_w(L) \prod_{v \in \Sigma'}^{\sim} U_{w}
 \prod_{v \not\in S' \cup \Sigma'}^{\sim} \cO_{w}^*. \]
 Let \begin{align*}
 J_\Sigma &=  \prod_{v \not\in \Sigma  \cup \Sigma'}^{\sim} \cO_{w}^* \prod_{v \in \Sigma}^{\sim} L_w^* \prod_{v \in \Sigma'}^{\sim} U_{w},  
   & J_{\Sigma_\fp} & =  \prod_{v \not\in \Sigma_\fp  \cup \Sigma'}^{\sim} \cO_{w}^* \prod_{v \in \Sigma_\fp}^{\sim} L_w^* \prod_{v \in \Sigma'}^{\sim} U_{w}, 
 \\
  W_\Sigma &= \prod_{v \in S' - \Sigma}^{\sim}  W_{w}(L)  \prod_{v \in \Sigma}^{\sim} \Delta \fg_v, 
   &  W_{\Sigma_\fp} & =  \prod_{v \in S' - \Sigma_\fp}^{\sim}  W_{w}(L)  \prod_{v \in \Sigma_\fp}^{\sim} \Delta \fg_v.
 \end{align*}

Note that we do not adorn $V$ with a subscript because it does not depend on the choice of $\Sigma$ versus $\Sigma_\fp$.  The fact that the same module $V$ is used in both constructions will be of great importance.

For each set $\bl = \Sigma$ or $\Sigma_\fp$, we have a commutative diagram
\begin{equation}  \label{e:rwdiagram}
\begin{gathered}
\begin{tikzcd}
1 \ar[r] & J_\bl \ar[r]  \ar[d,"\theta_J"] & V  \ar[r] \ar[d,"\theta"] & W_\bl  \ar[r] \ar[d,"\theta_W"] & 1 \\
1 \ar[r] & C_L  \ar[r] & \fO  \ar[r] & \Delta \fg  \ar[r] & 1.
\end{tikzcd}
\end{gathered}
\end{equation}
Here $C_L = \A_L/L^*$ denotes the id\`ele class group of $L$, and $\fO$ denotes the extension of $\Delta \fg$ by $C_L$ associated to the global fundamental class in $H^2(\fg, C_L)$ (see \cite{rw}).
By \cite{dk}*{Lemma A.1}, the map $\theta$ is surjective.
We therefore get a short exact sequence
\begin{equation} \label{e:snake}
\begin{tikzcd}
 0 \ar[r] & \cO_{L, \bl, \Sigma'}^* \ar[r] & V^\theta \ar[r] & W_\bl^\theta \ar[r] & \Cl_{\bl}^{\Sigma'}(L) \ar[r] & 0,
 \end{tikzcd}
 \end{equation}
where $V^\theta$ denotes the kernel of $\theta$ and $W_{\bl}^{\theta}$ denotes the kernel of $\theta_W$. Further, $\cO_{L, \bl, \Sigma'}^*$ denotes the group of $\bl$-units of $L$ congruent to 1 modulo $\Sigma'$, i.e.\ the set of elements in $L^*$ whose image in $\tilde{\prod}_{v} L_w$ is contained in $J_{\Sigma_*}$.    
Next we define 
\[ B_\Sigma =  (\Z[\fg] \oplus \Z[\fg/\fg_\fp]) \oplus \prod_{v \in S' - \{\fp\}} \Z[\fg], \qquad \qquad B_{\Sigma_\fp} =  \prod_{v \in S'} \Z[\fg].  \]
In $B_\Sigma$, the first term $\Z[\fg] \oplus \Z[\fg/\fg_\fp]$  will be referred to as the component at $\fp$.  

There are injective maps $j_\bl \colon W_\bl \longrightarrow B_\bl $ that we now describe on each component.
\begin{itemize}
\item Each $v \in S' - \Sigma_\fp$ is unramified in $L$, so we have $W_w(L) \cong \Z[\fg_v]$ (see \cite{rw}*{Lemma 5}).
If $\sigma_w \in \WD(L_w^{\nr}/F_v)$ is the Frobenius element, this isomorphism sends the image of $\sigma_v - 1 \in I_W$ to 1.
We then have  $\Ind_{\fg_v}^{\fg} W_w(L) \cong \Z[\fg]$.  The component of $j_\bl$ at $v \in S' - \Sigma_\fp$ is this isomorphism.

\item For $v \in \bl$, the inclusion $\Delta \fg_v \subset \Z[\fg_v]$ induces $\Ind_{\fg_v}^{\fg} \Delta {\fg_v} \subset \Z[\fg]$.
The component of $j_\bl$ at $v \in \bl$  is  this inclusion.

\item To conclude we define, for $w$ the chosen place of $L$ above $\fp$, a map 
\begin{equation} \label{e:wpl}
\xymatrix{ j_\fp \colon \Ind_{\fg_\fp}^{\fg} W_w(L) \ar[r] & \Z[\fg] \oplus \Z[\fg/\fg_\fp] } \end{equation}
giving the component of $j_{\Sigma}$ at $\fp$.
Consider the composition \begin{equation} \label{e:ordp}
\begin{tikzcd}
 \WD(\overline{F}_\fp/F_\fp) \ar[r] & \WD(F^{\ab}_\fp/F_\fp) \cong F_\fp^* \ar[r,"\ord_\fp"] & \Z,
 \end{tikzcd}
 \end{equation}
which we simply denote $\ord_\fp$.  Clearly $\ord_\fp$ factors through $\WD(L_w^{\nr}/F_\fp)$.  We then define
\[ \begin{tikzcd}
 W_w(L) \ar[r] &   \Z[\fg_\fp] \oplus \Z
 \end{tikzcd} \]
by 
\begin{equation} \label{e:jpdef}
  \sigma -1 \mapsto (  \sigma|_{L_w} -1, \ord_{\fp}(\sigma)), \qquad \sigma \in\WD(L_w^{\nr}/F_\fp). 
  \end{equation}
Inducing this map from $\fg_\fp$ to $\fg$ yields the desired map $j_\fp$.  
\end{itemize}

The fact that $j_\fp$ is an injection follows from \cite{rw}*{Lemma 5(b)}.  
Let $Y_\fp$ denote the cokernel of $j_\fp$.  
Let \[ Y_{\Sigma} =  Y_\fp \prod_{v \in \Sigma} \Ind_{\fg_v}^{\fg} \Z, \qquad Y_{\Sigma_\fp} = \prod_{v \in \Sigma_\fp} \Ind_{\fg_v}^{\fg} \Z. \]
For each $\bl = \Sigma$ or $\Sigma_\fp$ we have a commutative diagram with   exact rows:
\begin{equation}  \label{e:wbdiagram}
\begin{gathered}
\begin{tikzcd}
0 \ar[r] & W_\bl \ar[r]  \ar[d,"\theta_W"] & B_\bl \ar[r] \ar[d,"\theta_B"] & Y_\bl \ar[r] \ar[d,"\theta_Y"] & 0 \\
0 \ar[r] & \Delta \fg  \ar[r] & \Z[\fg] \ar[r] & \Z \ar[r] & 0.
\end{tikzcd}
\end{gathered}
\end{equation}
The map $\theta_B$ is defined as follows:
\begin{itemize}
\item $\theta_B$ is the identity in the factors corresponding to $v \in \Sigma_*$.
\item $\theta_B$ is multiplication by $\sigma_v - 1$ (where $\sigma_v \in \fg_v \subset \fg$ denotes Frobenius at $v$) in the factors corresponding to $v \in S' - \Sigma$ (see \cite{rw}*{Lemma 5}).
\item For $v = \fp$ and $\bl = \Sigma$, the map $\theta_B$ on the component at $\fp$ is  projection 
onto the first factor.
\end{itemize}
Since $\theta_W$ is surjective, taking kernels in (\ref{e:wbdiagram}) yields a short exact sequence
\begin{equation} \label{e:wbx}
\begin{tikzcd}
 0 \ar[r] & W^\theta_\bl  \ar[r] & B^\theta_\bl  \ar[r] & Y^\theta_\bl  \ar[r] & 0. 
 \end{tikzcd}
 \end{equation}
Note that $B^\theta_\Sigma \cong \Z[\fg/\fg_\fp] \oplus \Z[\fg]^{\#S'-1} $ and $B_{\Sigma_\fp}^\theta \cong \Z[\fg]^{\#S' - 1}$.  We define 
\begin{equation} \begin{tikzcd}
\tilde{\nabla}_{\bl}^{\Sigma'}(L) = \coker(f_{\bl} \colon V^\theta \ar[r] &  W^\theta_{\bl} \ar[r] & B^\theta_\bl). 
\end{tikzcd} \label{e:fsigmadef}
\end{equation}
When $\bl = \Sigma_\fp$, the module $\tilde{\nabla}_{\Sigma_\fp}^{\Sigma'}(L)$ is the same as the module 
$\nabla_{\Sigma_\fp}^{\Sigma'}(L)$  defined in \cite{dk}, and when we speak of it individually we will use the latter notation.  However for $\bl = \Sigma$, 
the module $\tilde{\nabla}_{\Sigma}^{\Sigma'}(L)$ is a subquotient of the module  ${\nabla}_{\Sigma}^{\Sigma'}(L)$ defined in {\em loc.\,cit.}.  We have introduced the $\tilde{\nabla}$ notation so as to not conflict with the notation of {\em loc.\,cit.}, and so that we may speak of ${\nabla}_{\Sigma_\fp}^{\Sigma'}(L)$ and
$\tilde{\nabla}_{\Sigma}^{\Sigma'}(L)$ simultaneously when convenient, using the symbol $\tilde{\nabla}_{\bl}^{\Sigma'}(L)$.

In view of (\ref{e:snake}) and (\ref{e:wbx}), we have two exact sequences:
\begin{equation} \label{e:tate}
\begin{tikzcd}
  0   \ar[r] & \cO_{L, \bl, \Sigma'}^*  \ar[r] & V^\theta  \ar[r] & B^\theta_\bl  \ar[r] & \tilde{\nabla}_{\bl}^{\Sigma'}(L)  \ar[r] & 0, 
  \end{tikzcd}
 \end{equation}
\begin{equation} \label{e:seltr}
\begin{tikzcd}
0  \ar[r] & \Cl_\bl^{\Sigma'}(L)  \ar[r] & \tilde{\nabla}_{\bl}^{\Sigma'}(L)  \ar[r] & Y^\theta_\bl  \ar[r] & 0. 
\end{tikzcd}
\end{equation}

In \cite{dk} we proved the following results.  Write 
\[ (V^\theta)_R = (V^\theta) \otimes_{\Z[\fg]} R, \]
and similarly with $V^\theta$ replaced by any other $\Z[\fg]$-module.

\begin{lemma}[\cite{dk}*{Lemma A.4}]  \label{l:qp}
The module $(V^\theta)_R$ is free over $R$ of rank equal to the rank of the free module $B_{\Sigma_\fp}^\theta$, namely $\#S' - 1$.  Hence the module 
${\nabla}_{\Sigma_{\fp}}^{\Sigma'}(L)_R$ is quadratically presented over $R$.
\end{lemma}

\begin{theorem}[\cite{dk}*{Theorem 3.3}]   \label{t:dkresult} We have \begin{equation} \label{e:ftheta}
   \Fitt_{R}(\nabla_{\Sigma_\fp}^{\Sigma'}(L)_R) = (\Theta^L_{\Sigma_\fp, \Sigma'}). \end{equation}
\end{theorem}

We can now give the proof of Theorem~\ref{t:gorder}.
\begin{proof}[Proof of Theorem~\ref{t:gorder}]    

If we denote by $\{e_v\colon v \in S' \}$ the standard $R$-basis for \[ (B_{\Sigma_\fp})_R = \prod_{v \in S'} R, \] 
then the module $(B_{\Sigma_\fp}^\theta)_R$ has a basis \[ \{b_v = e_v - \theta_B(e_v)e_\infty : v \in S' - \infty\} \] where $\infty \in S_\infty \subset S'$ is any fixed infinite place of $F$.
By Lemma~\ref{l:qp}, the $R$-module $V^\theta_R$ is free of rank $\#S'-1$, and hence 
$  \Fitt_{R}(\nabla_{\Sigma_\fp}^{\Sigma'}(L)_R)$ is the ideal generated by the determinant of the square matrix $A$
representing the map $f\colon V^\theta_R \longrightarrow (B_{\Sigma_\fp}^\theta)_R$ with respect to any bases.

We choose any basis of $V^\theta_R$ and the basis $\{b_v \colon v \in S', v \neq \infty\}$ for 
$(B_{\Sigma_\fp}^\theta)_R$.
  The columns of the associated matrix $A$ are indexed by the basis vectors $b_v$.  For $v \in \Sigma_\fp$, the elements of the corresponding column vector lie, by definition, in the image of $\Ind_{\fg_v}^{\fg} \Delta \fg_v \longrightarrow R$.
  This is exactly the ideal \[ I(v) = \ker(R \longrightarrow \Z_p[\fg/\fg_v]^-). \]  The determinant of $A$ therefore lies in $\prod_{v \in \Sigma_\fp} I(v)$.  In view of (\ref{e:ftheta}), the result follows.
\end{proof}

\subsection{Interpretation via Galois Cohomology} \label{s:gc}

In \cite{dk}*{\S A.3}, we gave a description of the projection to the minus side of the class in $\Ext_{\Z[G]}^1( Y_{\Sigma_\fp}^\theta, \Cl_{\Sigma_\fp}^{\Sigma'}(L) )$ determined by ${\nabla}_{\Sigma_\fp}^{\Sigma'}(L)$ via the exact sequence (\ref{e:seltr}), using  Galois cohomology.  We recall this now and give the generalization that allows for the application to $\tilde{\nabla}_{\Sigma}^{\Sigma'}(L)$.

For each $v \in S'$, let
 \[ G_{F, v} \cong \Gal(\overline{F}_v/F_v) \subset G_F\] denote the decomposition group at $v$ associated to a place of $\overline{F}$ above $v$ restricting to the chosen place $w$ of $L$.  Let $M$ be a $\fg$-module.
Suppose that we are given a 1-cocycle \[ \kappa \in Z^1(G_{F}, M), \] 
where $M$ is endowed with a $G_{F}$-action via the canonical map $G_F \longtwoheadrightarrow \fg.$ Suppose  that the restriction of $\kappa$ to $\Gal(\overline{F}_v/L_w^{\nr})$ is trivial. For a $\Z[\fg]$-module $N$ we write \[ N^- = (N \otimes \Z[1/2])/(c+1), \] where $c \in \fg$ denotes complex conjugation.
We may then define a $\fg$-module homomorphism
\[ \begin{tikzcd}
 \varphi_{\kappa}^v \colon W_v^- = (\Ind_{\fg_v}^{\fg} W_w(L))^- \ar[r] & M^-
 \end{tikzcd} \]
by the rule \begin{equation} \label{e:phidef}
 \varphi_{\kappa}^v(g \otimes (\sigma - 1)) = g \cdot \kappa(\sigma), \qquad g \in G, \ \ \sigma \in \WD(L_w^{\nr}/F_v). 
 \end{equation}
The condition on $\kappa$ ensures that $\kappa(\sigma)$ is well-defined  for $\sigma \in \WD(L_w^{\nr}/F_v)$.  It is elementary to check that $\varphi_{\kappa}^v$ is well-defined, as follows.
\begin{itemize}
\item If $g \in \fg_v$, then 
\[ \varphi_{\kappa}^v(1 \otimes(g\sigma - g)) = \kappa(g \sigma) - \kappa(g) = g \cdot \kappa(\sigma) = \varphi_{\kappa}^v(g \otimes (\sigma - 1)). \]
\item If $\sigma \in \WD(L_w^{\nr}/F_v)$ and $\tau \in \WD(L_w^{\nr}/L_w)$, then
\[ \varphi_{\kappa}^v(1 \otimes (\tau-1)(\sigma-1)) = \kappa( \tau\sigma) - \kappa(\tau) -\kappa(\sigma)  = 0 \]
since $\tau$ acts trivially on $M$.
\end{itemize}

Write $Z_v = \Ind_{\fg_v}^{\fg} \Delta \fg_v$.  Since $\Delta \fg_v$ is canonically a $\fg_v$-module quotient of $W_w(L)$ (see \ref{e:rwvwdef}), $Z_v^-$ is canonically a $\fg$-module quotient of $W_v^-$.
If  the restriction of $[\kappa]$ to $\Gal(\overline{F}_v/L_w)$ is trivial, then $\varphi_\kappa^v$ descends to a homomorphism
\begin{equation} \label{e:phiz}
\begin{tikzcd}
\varphi_{\kappa}^v \colon Z_v^-  \ar[r] & M^-.
\end{tikzcd}
 \end{equation}

 Now let $M_\bl = \Cl_{\bl}^{\Sigma'}(L)^-$ for $\bl = \Sigma$ or $\bl = \Sigma_\fp$.  By class field theory, we may view $M_\bl$ as the Galois group of a field extension $\tilde{L}_\bl/L$.  The field $\tilde{L}_\bl$ is the maximal abelian extension of $L$ of odd order that is unramified outside $\Sigma'$, tamely ramified at $\Sigma'$, such that the primes in $\bl$ split completely, and such that the conjugation action of the complex conjugation in $\fg$ is inversion.
Let \[ \xymatrix{ \tilde{\lambda}_\bl \colon G_L \ar[r] & \Gal(\tilde{L}_{\bl}/L) \cong M_\bl}  \] denote the canonical homomorphism given by the reciprocity map of class field theory.
By \cite{dk}*{Lemma 6.3}, the class $\tilde{\lambda}_\bl \in H^1(G_L, M_\bl)$ is the restriction of a unique class
 \[ [\lambda_\bl] \in H^1(G_F, M_\bl).\]  An explicit 1-coycle representing this class is given as follows.  Let $\tilde{c}$ denote a lift of the complex conjugation $c \in \fg$ to an element of $\tilde{\fg} = \Gal(\tilde{L}_\bl/F)$.  Then
\begin{equation} \label{e:lambdadef}
 \lambda_\bl(\tilde{\sigma}) = \tilde{\lambda}_\bl(\tilde{\sigma}\tilde{c} \tilde{\sigma}^{-1} \tilde{c}^{-1})^{1/2}, \qquad \tilde{\sigma} \in G_F. 
 \end{equation}

Note that the cocycle $\lambda_{\bl} \in Z^1(G_{F}, M_\bl)$ satisfies the condition described above for $\kappa,$ namely that the restriction to
$\Gal(\overline{F}_v/L_w^{\nr})$ is trivial for each $v \in S'$.  This follows since $\tilde{L}_\bl \subset L_w^{\nr}$.  
Furthermore if $v \in \bl$, then the restriction of  $\lambda_{\bl} $ to $\Gal(\overline{F}_v/L_w)$ is trivial.
Therefore the construction in (\ref{e:phidef}) and (\ref{e:phiz}) yields  elements:
\begin{align*}
 \varphi_{\lambda_{\bl}}^{v} & \in \Hom_{\Z[\fg]^-}(W_v^-, M_{\Sigma_*}), \qquad v \in S' - \bl, \\
 \varphi_{\lambda_{\bl}}^v & \in \Hom_{\Z[\fg]^-}(Z_v^-, M_\bl), \qquad v \in \bl.
\end{align*}

\begin{lemma}  \label{l:gtd}  The ``snake map" $\varphi_\bl \colon (W_\bl^\theta)^- \longrightarrow M_\bl$ given by the minus part of the last nontrivial arrow in $(\ref{e:snake})$ can be described explicitly by the formula 
\[ \varphi_\bl((a_v)_{v \in S'}) = \sum_{v \in S'} \varphi_{\lambda_\bl}^v(a_v). \]
\end{lemma}

\begin{proof}  The proof is nearly identical to \cite{dk}*{Lemma A.9}, but we give it for completeness and because of notational differences.
We use the description of the snake map given by Ritter and Weiss in \cite{rw}*{Theorem 5}.
Write $\tilde{\fg} = \Gal(\tilde{L}/F)$, where as above $\tilde{L}$ is the extension of $L$ corresponding to $M_{\Sigma_*}$ via class field theory.
Write $\Delta \tilde{\fg}$ for the augmentation ideal of $\Z[\tilde{\fg}]$ and let $\Delta(\tilde{\fg}, M_{\Sigma_*})$ denote the kernel of the canonical projection
$\Z[\tilde{\fg}] \longrightarrow \Z[\fg]$.  There is a short exact sequence (see \cite{rw}*{Pg.\ 154})
\begin{equation} \label{e:rw} 
\begin{tikzcd}
0  \ar[r] & M_{\Sigma_*} \cong   \dfrac{\Delta(\tilde{\fg},M_{\Sigma_*})}{\Delta(\tilde{\fg}, M_{\Sigma_*}) \Delta\tilde{\fg}} 
 \ar[r] &  \dfrac{\Delta\tilde{\fg}}{\Delta(\tilde{\fg}, M_{\Sigma_*}) \Delta\tilde{\fg}}  \ar[r] & \Delta \fg \ar[r] & 0. 
 \end{tikzcd}
 \end{equation} 
 
Let $v \in S'$. The extension $\tilde{L}$ is unramifed (over $L$) at $w$, so there is a canonical restriction map $\WD(L_w^{\nr}/F_v) \longrightarrow \tilde{\fg}.$  In view of the definition of $W_w(L)$ given in (\ref{e:vandw}), this induces a canonical map
\[ \begin{tikzcd} \Ind_{\fg_v}^{\fg} W_w(L) = \Z[\fg] \otimes_{\Z[\fg_v]} W_w(L)  \ar[r,"f_v"] &  \dfrac{\Delta\tilde{\fg}}{\Delta(\tilde{\fg}, M_\Sigma) \Delta\tilde{\fg}}.
\end{tikzcd} \]

In \cite{rw}*{Theorem 5}, Ritter--Weiss show that the snake map is realized by \begin{equation} \label{e:rwform}
 \varphi_\bl((a_v)_{v \in S'}) = \sum_{v \in S'} f_v(a_v). \end{equation}

Let $a = (a_v)_{v \in S'} \in (W_\bl^\theta)^-$.   Choose $v' \in S'$ such that the decomposition group $\fg_{v'} \subset \fg$ contains complex conjugation $c$ (the existence of such a $v'$ is guaranteed by the assumptions on $S'$).
Let $\tilde{c}$ denote a lift of $c$  to $\WD(L_{w'}^{\nr}/F_v)$.

For each $v \in S'$, define $y_v \in W_\bl^-$ to have component at $v$ equal to $a_v$, component at $v'$ equal to  
\[ b_v = \theta_W(a_v) \otimes (1- \tilde{c})/2 \in (\Z[\fg] \otimes_{\Z[\fg_{v'}]} W_{w'}(L))^-, \] and all other components equal to 0.  Then $a = \sum_{v \in S'} y_v$ since $\theta_W(a) = 0$. Furthermore  each $y_v$ lies in $(W_\bl^\theta)^-$ by construction.
It therefore suffices to prove that $\varphi_\bl(y_v) =  \varphi_{\lambda_\bl}^v(a_v)$ for each $v \in S'$.  For this, we apply (\ref{e:rwform}) with the tuple $(a_v)_{v \in S'}$ replaced by $y_v$.

The module $W_v^-$ is generated over $\Z[\fg]^-$ by elements of the form \[ a_v= \tilde{\sigma} - 1 \in I_W \] for $\tilde{\sigma} \in \WD(L_w^{\nr}/F_v)$.  
Since we are working on the minus side, we have
\begin{align*}
f_v(a_v) & = f_v((1 -c)/2 \otimes a_v)  =(1 -\tilde{c})(\tilde{\sigma}- 1)/2, \\
f_{v'}(b_v) &=  (\tilde{\sigma}- 1)(\tilde{c} - 1)/2,
\end{align*}
and hence 
\[ \varphi_{\bl}(y_v) =  (\tilde{\sigma}\tilde{c} - \tilde{c}\tilde{\sigma})/2. \]
Under the isomorphism noted in the first nonzero term in (\ref{e:rw}), this correponds to the element 
\[ (\tilde{\sigma}\tilde{c} \tilde{\sigma}^{-1} \tilde{c}^{-1})^{1/2} \in \Gal(\tilde{L}/L) \cong M_{\Sigma_*}. \]
By (\ref{e:lambdadef}), this is precisely the value of $\varphi_{\lambda_{\bl}}^{v}(\tilde{\sigma} - 1)  = \lambda_{\bl}(\tilde{\sigma})$.
\end{proof}

As we now describe, the Galois theoretic description of $\nabla_{\Sigma}^{\Sigma'}(L)^-$ provided by Lemma~\ref{l:gtd} yields an explicit method of constructing homomorphisms \[ \nabla_{\Sigma}^{\Sigma'}(L)^- \longrightarrow B \] for $\Z[\fg]^-$-modules $B$.  Recall $R = \Z_p[\fg]^-$.
We work over an arbitrary $R$-algebra $A$ since this is the context we will require later.
Therefore let $B$ denote an $A$-module and let $[\kappa] \in H^1(G_F, B)$ denote a Galois cohomology class, where $B$ is endowed with a $G_F$-action via the composition \[ G_F \longrightarrow \fg \longrightarrow A^*. \]  Suppose that:
\begin{itemize}
\item The class $[\kappa]$ is unramified outside $\Sigma'$, locally trivial at  $\Sigma$, and tamely ramified at $\Sigma'$.
\item Let $B_0 \subset B$ denote the $A$-submodule generated by the image of the restriction 
\[ [\kappa]_{|G_L} \in H^1(G_L, B) = \Hom_{\cont}(G_L, B). \]
 The module
$B/B_0$ is generated over $A$ by the images of elements $x_v \in B$ for $v \in \Sigma$ and elements $x_{\fp}, x_{\fp}' \in B$.  The element $x_{\fp}'$ is fixed by the action of $G_{F,\fp}$ (i.e.\ by the action of $\fg_\fp$).
\item The class $[\kappa]$ is represented by a 1-cocycle $\kappa$ satisfying the following.
\begin{itemize}
\item For $v \in \Sigma$ and $\sigma \in G_{F, v}$, we have $\kappa(\sigma) = x_v(\sigma - 1).$
\item For $\sigma \in \WD(\overline{F}_\fp/F_\fp)$, we have \begin{equation} \label{e:xpdef}
 \kappa(\sigma) =     (\sigma - 1)    x_{\fp} + \ord_\fp(\sigma) x_{\fp}',\end{equation}
where $\ord_\fp$ is as in (\ref{e:ordp}). 
\item There exists $\tau \in G_F$, a lift of the complex conjugation in $\fg$, such that for all $\sigma \in G_F$ we have
\begin{equation} \label{e:nicecocycle}
\kappa(\sigma)=  [\kappa]|_{G_L}({\sigma}\tau {\sigma}^{-1} \tau^{-1})/2 \in B_0.
\end{equation}

\end{itemize}
\end{itemize}

\begin{theorem}  \label{t:nc}  Let $B$ be an $A$-module and $[\kappa] \in H^1(G_F, B)$  a Galois cohomology class satsifying the three bulleted points above. 
Write $A_\fp = A \otimes_{R} \Z_p[\fg/\fg_\fp]^-$.  There is a surjective $A$-module homomorphism 
\begin{equation} \label{e:alpha1} \begin{tikzcd}
 \alpha_1 \colon \tilde{\nabla}_{\Sigma}^{\Sigma'}(L)_A \ar[r,two heads] & B  \end{tikzcd} \end{equation}
induced by the map \[ \begin{tikzcd}
\alpha \colon (B_{\Sigma})_A \cong (A \oplus A_\fp) \oplus A^{\#S'-1} \ar[r] & B
\end{tikzcd} \] defined as follows:
\begin{itemize}
\item For $a_v \in A$ in the component at $v \in S' - \Sigma_\fp$, we have $\alpha(a_v) = a_v  \kappa(\sigma_v)$, where $\sigma_v$ denotes the Frobenius element at $v$.
\item For $a_v \in A$ in the component at $v \in \Sigma$, we have $\alpha(a_v) = a_v  x_v$.
\item For $(a_v, b_v) \in A \oplus A_\fp$ in the component at $\fp$, we have $\alpha(a_v, b_v) = a_v x_{\fp} + b_v x_{\fp}'.$
\end{itemize}
\end{theorem}

\begin{proof}  
The homomorphism $[\kappa]_{|G_L} \in H^1(G_L, B) = \Hom_{\cont}(G_L, B)$ is unramified outside $\Sigma'$, locally trivial at  $\Sigma$, and tamely ramified at $\Sigma'$.  It follows from class field theory that $[\kappa]_{|G_L}$ factors through \[ M_\Sigma \cong \Gal(\tilde{L}_\Sigma/L)\] and therefore induces a surjective map 
\[ \begin{tikzcd}
(M_\Sigma)_R \ar[r,two heads] & B_0.  
\end{tikzcd}
 \]

Equation (\ref{e:nicecocycle}) implies that $\kappa$ takes values in $B_0$.

To prove that $\alpha$ induces a map $\alpha_1$ as in (\ref{e:alpha1}), it suffices to prove that the composition
\[ \begin{tikzcd}
 (W_\Sigma^\theta)_A \ar[r] & (B_\Sigma^\theta)_A \ar[r] & (B_\Sigma)_A \ar[r,"\alpha"] & B 
 \end{tikzcd}
 \]
can be factored as
\[ \begin{tikzcd}    
(W_\Sigma^\theta)_A \arrow[r, "\varphi_{\Sigma}"] & (M_\Sigma)_A \arrow[r, "\kappa|_{G_L}"] & B_0 \arrow[r,hook] & B.
\end{tikzcd}\]
For then the image of $(V^\theta)_A$ in $(W_\Sigma^\theta)_R$ vanishes under $\alpha$.  Note that the composition $\kappa|_{G_L} \circ \varphi_{\Sigma}$ equals $\varphi_\kappa := \sum_{v \in S'} \varphi_\kappa^v$, by Lemma~\ref{l:gtd} together with equations (\ref{e:lambdadef}) and (\ref{e:nicecocycle}).

The fact that the restriction of $\alpha$  to $(W_\Sigma^\theta)_A$ equals $\varphi_\kappa$ follows from the assumptions on $\kappa$, as we now check on each component.

\begin{itemize}
\item Any $v \in S' - \Sigma_\fp$ in unramifed in $L$ and hence $(W_v)_A$ is generated over $A$ by
$\sigma_v - 1$, where $\sigma_v$ is the Frobenius element at $v$.   
By definition of the map $W_\Sigma \longrightarrow B_\Sigma$, the image of $\sigma_v - 1$ in the component of $B_\Sigma$ at $v$ is simply 1.  Therefore
\[ \alpha(\sigma_v - 1) = \kappa(\sigma_v) = \varphi_\kappa^v(\sigma_v - 1). \]
\item For $v \in \Sigma$, let $\sigma \in \fg_v$, and consider the element $\sigma - 1$ in the $v$-component of $(B_{\Sigma})_A$.  We find:
 \[ \alpha(\sigma - 1) = x_v(\sigma - 1) = \kappa(\sigma) = \varphi_\kappa^v(\sigma - 1). \]
\item  Let $y \in (W_\fp)_R$ be represented by $\tilde{\sigma} - 1$ for ${\sigma} \in \WD(L_w^{\nr}/F_\fp)$.
Consider $j_\fp(y)$ in the $\fp$-component of $(B_{\Sigma})_R$. We find:
\begin{align*}
 \alpha(j_\fp(y)) &= \alpha(  \sigma|_{L_w} - 1, \ord_\fp({\sigma})) \\
 & =  (\sigma - 1) x_{\fp}  + \ord_\fp({\sigma}) x_{\fp}'   \\
 & = \kappa(\sigma) \\ & = \varphi_\kappa^v(y).
  \end{align*}
\end{itemize}
This concludes the proof that $\alpha$ induces the desired map $\alpha_1$.  In fact, if we let $\alpha_0$ denote the composition
of $\alpha$ with the projection $B \longrightarrow B/B_0$, then we have demonstrated a commutative diagram
\begin{equation}  \label{e:bigd}
\begin{gathered}
\begin{tikzcd}
0  \ar[r] & (M_\Sigma)_A  \ar[r]  \ar[d,two heads,"\kappa_{|G_L}"] & \tilde{\nabla}_{\Sigma}^{\Sigma'}(L)_A  \ar[r] 
\ar[d,"\alpha_1"] & (Y^\theta_\Sigma)_A  \ar[r] \ar[d,two heads,"\alpha_0"] & 0 \\
0 \ar[r] & B_0 \ar[r] & B \ar[r] & B/B_0 \ar[r] & 0. 
\end{tikzcd}
\end{gathered}
\end{equation}
The surjectivity of $\alpha_0$ follows by the assumption that $B/B_0$ is generated over $R$ by the $x_v$ for $v \in \Sigma$ along with $x_{\fp}, x_{\fp}'$.
The surjectivity of $\alpha_1$ then follows from the Five Lemma.
\end{proof}

\subsection{The module $\nabla_{\!\sL}$} \label{s:nablal}

Recall $R = \Z_p[\fg]^-, \overline{R} = \Z_p[G]^-$.
As we did with $(B_{\Sigma_\fp}^\theta)_R$ in the proof of Theorem~\ref{t:gorder}, we can write down a generating set for the $R$-module $(B_{\Sigma}^\theta)_R$ as follows.  The module $(B_{\Sigma})_R$ is generated over $R$ by the vectors $e_0 = (1,0)$ and $e_1 = (0,1)$ in the component $R \oplus R_\fp$ at $\fp$ together with the standard basis vectors $e_v$ for $\prod_{v \in S' \setminus \fp} R \subset (B_{\Sigma})_R$.
The module $(B_{\Sigma}^\theta)_R$ is then generated by
\[ \{ b_0 = e_0 - e_\infty, \  b_1 = e_1, \ b_v = e_v - \theta_B(e_v)e_\infty \colon v \in S' \setminus \{\infty, \fp\} \}, \]
where $\infty \in S_\infty \subset S'$ is any fixed infinite place of $F$.

Denote the image of the vectors $b_0$ and $b_1$ in $\tilde{\nabla}_{\Sigma}^{\Sigma'}(L)$ by $\overline{b}_0$ and $\overline{b}_1$, respectively.  Recall  the $R$-algebra  $R_\sL$  defined in  (\ref{e:rldef}). We define
\begin{equation} \label{e:nablaldef}
 \nabla_{\!\sL} = \tilde{\nabla}_{\Sigma}^{\Sigma'}(L)_R \otimes_R R_\sL/(\overline{b}_1 + \sL \overline{b}_0).
\end{equation}

The following is the analogue of Theorem~\ref{t:nc} for the module $\nabla_{\!\sL}$.  Again we work over an arbitrary $R$-algebra $A$ and write $A_\sL = R_\sL \otimes_R A$, $\nabla_{\!\sL, A} = \nabla_{\!\sL} \otimes_R A$.

\begin{theorem}  \label{t:nl} Let $B$ and $[\kappa] \in H^1(G_F, B)$ be as in Theorem~\ref{t:nc}. Suppose that $\tilde{B}$ is an $A_{\sL}$-module endowed with an $A$-module map $B \longrightarrow \tilde{B}$ such that the image of $B$ generates $\tilde{B}$ over $A_{\sL}$. Suppose further that the images in $\tilde{B}$ of $x_{\fp}, x_{\fp}' \in B$ defined in (\ref{e:xpdef}) satisfy 
$x_{\fp}' + \sL \cdot x_{\fp} = 0$.  Then the map $\alpha_1$ of $(\ref{e:alpha1})$ induces a surjective $A_{\sL}$-module homomorphism 
\begin{equation} \label{e:nlsurj} \begin{tikzcd}
\alpha_{\sL} \colon \nabla_{\!\sL, A} \ar[r, two heads] & \tilde{B}. 
\end{tikzcd} \end{equation}
\end{theorem}

\begin{proof}  Since the image of $B$ generates $\tilde{B}$ as an $A_{\sL}$-module, the surjection $\tilde{\nabla}_{\Sigma}^{\Sigma'}(L)_A \longtwoheadrightarrow B$ induces a surjection
 \begin{equation} \label{e:tildeb}
 \tilde{\nabla}_{\Sigma}^{\Sigma'}(L)_A \otimes_A A_{\sL}  \longtwoheadrightarrow \tilde{B}.
 \end{equation}
Since $\alpha_1(\overline{b}_0) = x_{\fp}$ and $\alpha_1(\overline{b}_1) = x_{\fp}'$, the equality $x_{\fp}' + \sL \cdot x_{\fp} = 0$ implies that the surjection (\ref{e:tildeb}) factors through
\[ \nabla_{\!\sL, A} = \tilde{\nabla}_{\Sigma}^{\Sigma'}(L)_A \otimes_A A_{\sL} /(\overline{b}_1 + \sL \overline{b}_0)
\]
as desired.
\end{proof}

\subsection{Gross's Conjecture via Fitting Ideals} \label{s:gcfi}

In this section we prove the following interpretation of Gross's Conjecture.

\begin{theorem} \label{t:mfig}  The $R_\sL$-module $\nabla_{\!\sL}$ is quadratically presented and we have
\[ \Fitt_{R_\sL}(\nabla_{\!\sL}) = (\rec_G(u_\fp^{\Sigma, \Sigma'}) - \Theta_L). \]
Therefore, the equality
\[  \Fitt_{R_\sL}(\nabla_{\!\sL}) = 0 \]
implies the $p$-part of the modified Gross conjecture:
\[ \rec_G(u_\fp^{\Sigma, \Sigma'}) \equiv \Theta_L \pmod{I^2}. \]
\end{theorem}

We would like to point out that most of the computations of \S\ref{s:gcfi} are the same as, or slight variants of, the calculations of Burns, Kurihara, and Sano in \S5 of \cite{bks}.   Our approach to relating the Brumer--Stark unit $u_\fp^{\Sigma,\Sigma'}$ to Ritter--Weiss modules and studying its properties is modeled after theirs.  Our innovation is the definition of $R_\sL$ and $\nabla_{\!\sL}$ and the application of these techniques to the statement and proof of Theorem~\ref{t:mfig}.

\bigskip

First we note that the same construction used to define $\tilde{\nabla}_{\Sigma_\fp}^{\Sigma'}(L) = \nabla_{\Sigma_\fp}^{\Sigma'}(L)$ above, but working over $H$ rather than $L$, gives rise to modules 
$\nabla_{\Sigma}^{\Sigma'}(H)$ and $\nabla_{\Sigma_\fp}^{\Sigma'}(H)$.
  Since $H/F$ is unramified outside $\Sigma \cup \Sigma'$, these modules
  are the cokernels of maps
  \[ \begin{tikzcd} V^\theta(H) \ar[r, "f_{H, \Sigma}"] & B^\theta(H), \end{tikzcd} \qquad
  \begin{tikzcd} V^\theta(H) \ar[r, "f_{H, \Sigma_\fp}"] & B^\theta(H), \end{tikzcd} \]
respectively,  with the {\em same} domain and codomain.
The modules $\nabla_{\Sigma}^{\Sigma'}(H)_{\overline{R}}$ and $\nabla_{\Sigma_\fp}^{\Sigma'}(H)_{\overline{R}}$ satisfy properties analogous to those stated in Lemma~\ref{l:qp} and Theorem~\ref{t:bs}.  Specifically, $V^\theta(H)_{\overline{R}}$ is free over $\overline{R}$ with constant rank equal to the rank of the free module $B^\theta(H)_{\overline{R}}$, namely $\#S' - 1$.   Furthermore, as stated in Theorem~\ref{t:dkresult} above, 
we have by \cite{dk}*{Theorem 3.3} the equalities
\[ \Fitt_{\overline{R}} \nabla_{\Sigma}^{\Sigma'}(H)_{\overline{R}} = (\Theta_{\Sigma, \Sigma'}^{H}), \qquad
 \Fitt_{\overline{R}} \nabla_{\Sigma_\fp}^{\Sigma'}(H)_{\overline{R}} = (\Theta_{\Sigma_\fp, \Sigma'}^{H}) = 0. \]

In \cite{dk}*{Lemmas B.1 and B.2} we prove that there is a commutative diagram
\begin{equation} \begin{gathered}
 \begin{tikzcd}
V^\theta(L)_R \ar[d,"f_{\Sigma_\fp}"] \ar[r, "\N\Gamma", two heads] & V^\theta(L)^\Gamma_{\overline{R}} \ar[r,"\sim"] \ar[d] & V^\theta(H)_{\overline{R}} \ar[d,"f_{H, \Sigma_\fp}"] \\
B_{\Sigma_\fp}^\theta(L)_R \ar[r, "\N\Gamma", two heads] &  B_{\Sigma_\fp}^\theta(L)^{\Gamma}_{\overline{R}} \ar[r, "\sim"] & B^\theta(H)_{\overline{R}}.
\end{tikzcd}
\end{gathered} \label{e:bv}
\end{equation}
To give the analogous diagram for $\Sigma$, we need some additional notation.  
Define \[ \overline{B}_{\Sigma}(L) = \Z[\fg/\fg_\fp] \oplus \bigoplus_{v\in S' \setminus \fp} \Z[\fg], \]
and let $\pi \colon {B}_{\Sigma}(L) \longrightarrow \overline{B}_{\Sigma}(L)$ be the projection in which the first component at $\fp$, which is a factor $\Z[\fg]$, has been forgotten. 
Recall the map $f_\Sigma \colon V^\theta(L) \longrightarrow B_{\Sigma}^{\theta}(L)$ defined in (\ref{e:fsigmadef}).
 Let 
$f_{\Sigma, \pi} = \pi \circ f_{\Sigma}$.  Note that since $\fg_\fp \subset \Gamma$, multiplication by $\N\Gamma$ induces a well-defined map \[ \Z[\fg/\fg_\fp] \longrightarrow \N\Gamma \cdot  \Z[\fg] = \Z[\fg]^{\Gamma} \] and hence a well-defined map $ \overline{B}_{\Sigma}(L) \longrightarrow B_{\Sigma_\fp}(L)^\Gamma$.  By \cite{dk}*{Lemma B.2}, we then have a commutative diagram:
\begin{equation} \begin{gathered}
 \begin{tikzcd}
V^\theta(L)_R \ar[d,"f_{\Sigma, \pi}"] \ar[r, "\N\Gamma", two heads] & V^\theta(L)^\Gamma_{\overline{R}} \ar[r,"\sim"] \ar[d] & V^\theta(H)_{\overline{R}} \ar[d,"f_{H, \Sigma}"] \\
\overline{B}_{\Sigma}^\theta(L)_R \ar[r, "\N\Gamma", two heads] &  B_{\Sigma_\fp}^\theta(L)^{\Gamma}_{\overline{R}} \ar[r, "\sim"] & B^\theta(H)_{\overline{R}}.
\end{tikzcd}
\end{gathered} \label{e:bv2}
\end{equation}

Write $t = \#S' - 1$ and fix an $R$-basis $\{v_1, v_2, \dotsc, v_t \}$ of $V^\theta(L)_R$.
As in the proof of Theorem~\ref{t:gorder}, we choose the following basis for $B_{\Sigma_\fp}^\theta(L)_R$:
\begin{equation} \label{e:bbasis}
 \{b_v = e_v - \theta_B(e_v)e_\infty : v \in S' - \infty\}, \end{equation}
 where $\infty \in S_\infty \subset S'$ is any fixed infinite place of $F$ and $\{e_v\}$ is the standard basis of 
\[ B_{\Sigma_\fp}(L)_R = \prod_{v \in S'} R. \]
 Let $\{\overline{v}_1, \dotsc, \overline{v}_t\}$ 
and $\{\overline{b}_v\}$ 
denote the $\overline{R}$-bases of $V^\theta(H)_{\overline{R}}$ and 
$ B^\theta(H)_{\overline{R}}$, respectively, obtained by applying the horizontal maps in (\ref{e:bv}).

Having fixed these bases, we define:
\begin{itemize}
\item  $\eA_{\Sigma_\fp} \in M_{t \times t}(R)$ is the matrix for $f_{\Sigma_\fp}$.
\item $\eA_{\Sigma}$ is the $t \times (t+1)$ matrix representing the map $f_{\Sigma}$, with second column having entries in $R_\fp = \Z_p[\fg/\fg_\fp]^-$ and all other columns having entries in $R$.
\end{itemize}
By the commutative  diagrams (\ref{e:bv}) and (\ref{e:bv2}) we have:
\begin{itemize}
\item The reduction of $\eA_{\Sigma_\fp}$ modulo $I$,  denoted 
$\overline{\eA}_{\Sigma_\fp} \in M_{t \times t}(\overline{R})$, is the matrix for $f_{H, \Sigma_\fp}$.
\item Let $\overline{\eA}_{\Sigma}$ denote the matrix in $M_{t\times t}(\overline{R})$ obtained from $\eA_\Sigma$ by deleting the first column and reducing the other entries modulo $I$.
Then $\overline{\eA}_{\Sigma}$ is the matrix for $f_{H, \Sigma}$.
\end{itemize}
We furthermore note that:
\begin{itemize}
\item The matrices $\overline{\eA}_{\Sigma_\fp}$ and $\overline{\eA}_{\Sigma}$ agree other than their first columns, since the components away from $\fp$ of the maps $f_{H, \Sigma_\fp}, f_{H, \Sigma}$ are the same.
\item The first column of $\overline{\eA}_{\Sigma_\fp}$ consists of all zeroes, since $\fg_\fp \subset \Gamma$.
\end{itemize}

We now define a square $t \times t$ matrix $\eA_V$.  The last $t-1$ columns of $\eA_V$ have entries in 
$\overline{R}$ and are equal to the last $t-1$ columns of $\overline{\eA}_{\Sigma_\fp}$  (equivalently, 
$\overline{\eA}_{\Sigma}$). The first column of $\eA_V$ is the column vector $(\overline{v}_i)_{i = 1}^{t}$, with entries in $V^\theta(H)_{\overline{R}}$.  It makes sense to consider the determinant of $\eA_V$ as an element of $V^\theta(H)_{\overline{R}}$ by Leibniz formula for determinants.

\begin{lemma} \label{l:detav} We have $\det(\eA_V) \in \ker(f_{H, \Sigma_\fp}).$
\end{lemma}

\begin{proof} The value $f_{H, \Sigma_\fp}(\det(\eA_V))$ is the determinant of the matrix in which the first column of $\eA_V$ has been replaced by column whose elements are $f_{H, \Sigma_\fp}(\overline{v}_i) \in B^\theta(H)_{\overline{R}}.$
With respect to the decomposition of $B^\theta(H)_{\overline{R}}$ as a product over the places $v \in S'$,
each component of $(f_{H, \Sigma_\fp}(\overline{v}_i))_{i=1}^{t}$ corresponding to a place $v \in S' \setminus \fp$ is equal to another column of the matrix, namely the column corresponding to $v$.  Meanwhile the component at $v = \fp$ is the 0 vector, as noted in the bulleted point above, regarding $\overline{\eA}_{\Sigma_\fp}$.  It follows that the determinant is 0 in every component of $ B^\theta(H)_{\overline{R}}.$
\end{proof}

From the exact sequence
\begin{equation} \label{e:tateh}
\begin{tikzcd}
 0   \ar[r] & (\cO_{H, \Sigma_\fp, \Sigma'}^*)_{\overline{R}}  \ar[r] & V^\theta(H)_{\overline{R}}  \ar[r, "f_{H, \Sigma_\fp}"] & B^\theta(H)_{\overline{R}}  \ar[r] & {\nabla}_{\Sigma_\fp}^{\Sigma'}(H)_{\overline{R}}  \ar[r] & 0, 
 \end{tikzcd}
 \end{equation}
it follows from Lemma~\ref{l:detav} that $\det(\eA_V)  \in V^\theta(H)_{\overline{R}}$ is the image of a unit 
\begin{equation} \label{e:epsilondef}
 \epsilon \in  (\cO_{H, \Sigma_\fp, \Sigma'}^*)_{\overline{R}}. \end{equation}

\begin{lemma}  We have $\epsilon \in (\cO_{H, \fp, \Sigma'}^*)_{\overline{R}}$. 
\end{lemma}

\begin{proof} Let $v \in \Sigma$ and let $w$ be the chosen place of $H$ above $v$ in the definition of $V(H)$. 
Recall the short exact sequence
\begin{equation} \label{e:vwseq}
 \begin{tikzcd} \Ind_{G_w}^G \cO_w^* \ar[r, hook] & V_v(H) = \Ind_{G_w}^G V_w(H_w)  \ar[r, two heads, "\omega_v"] &
W_v(H) = \Ind_{G_w}^{G} W_w(H_w) .
\end{tikzcd}
\end{equation}
To show that $\epsilon \in (\cO_{H, \fp, \Sigma'}^*)_{\overline{R}}$, we must show that the component of $\det(\eA_V)$ in $V_v(H)$ has vanishing image  in $W_v(H)$ under $\omega_v$.
  Now $\omega_v(\det(\eA_V))$ is the determinant of the matrix $\eA_{\omega_v}$ in which the first column of $\eA_V$ has been replaced by $(\omega_v(\overline{v}_{i}))_{i=1}^t$.
We use the description of $W_w(H_w)$ given by Ritter--Weiss in \cite{rw}*{\S3}.  We have \begin{equation} \label{e:wvdef}
 W_w(H_w)  = \{(x,y) \in \Delta G_w \oplus \Z[G_w/I_w] \colon \overline{x} = (\sigma_w - 1)y \}. \end{equation}
 Here $\sigma_w \in G_w/I_w$ denotes  Frobenius.
   If we write \[ \omega_v(\overline{v}_i) = \sum_{\sigma \in G/G_w} \sigma \otimes (x_{\sigma, i}, y_{\sigma,i}), \] 
   then $\eA_{\omega_v}$ has first column equal to these values, 
   and another column (the column corresponding to $v$) equal to
    \[ \sum_{\sigma \in G/G_w} \sigma \otimes x_{\sigma,i} \in \Ind_{G_w}^{G} \Delta G_w. \] 
Indeed, this is the component of $f_{H, \Sigma}$ at the factor of $B^\theta(H)$ corresponding to $v$.  Because of the relationship between $\overline{x}$ and $y$ in (\ref{e:wvdef}), it is easy to see that the determinant of such a matrix is zero.  This is clear for the first component of the ordered pair, and for the second we note that 
\[
(\sigma_w -1) y_{\sigma,  i} = \overline{x}_{\sigma, i}.
\]
The vanishing of  $\det \eA_{\omega_v}$ yields the desired result $\epsilon \in (\cO_{H, \fp, \Sigma'}^*)_{\overline{R}}$.
\end{proof}

\begin{lemma} \label{l:eord}
Let $\fP$ denote the chosen place of $H$ above $\fp$ used in the definition of $V(H)$.  We have 
\[  \ord_G(\epsilon) := \sum_{\sigma \in G} \ord_\fP(\sigma(\epsilon)) \sigma^{-1} = \det(\overline{\eA}_{\Sigma})
= x \Theta_H \]
for some unit $x \in \overline{R}^*$.
\end{lemma}

\begin{proof}
The proof is similar to the previous calculations.  The key point here is that since $\fp$ splits completely in $H$, we have $H_\fP = F_\fp$, hence $V_\fP(H_\fP) \cong H_\fP^*$ and  $W_\fP(H_\fP) \cong \Z$. The sequence (\ref{e:vwseq})
becomes the canonical sequence
\[
 \begin{tikzcd} \Ind_{1}^G \cO_\fP^* \ar[r, hook] & V_\fp(H) = \Ind_{1}^G H_\fP^*  \ar[r, two heads, "\omega_\fp"] &
W_\fp(H) = \Ind_{1}^{G} \Z = \Z[G]
\end{tikzcd}
\]
with $\omega_\fp = 1 \otimes \ord_{\fP}$.
The composition 
\[ \begin{tikzcd} \cO_{H, \Sigma_\fp, \Sigma'}^* \ar[r] & V_\fp(H) \ar[r, "\omega_\fp"] & \Z[G]
\end{tikzcd} \]
therefore is precisely the map $\ord_G$.  It follows that the value of $ \ord_G(\epsilon)$ is the determinant of the matrix in which the first column of $\eA_V$ has been replaced by $\omega_\fp(v_i)$.  But this is by definition the matrix $\overline{\eA}_{\Sigma}$.

To conclude, we note that
\[ (\det(\overline{\eA}_{\Sigma})) = \Fitt_{\overline{R}} \nabla_{\Sigma}^{\Sigma'}(H) = 
(\Theta_H)
\]
by \cite{dk}*{Theorem 3.3 and Corollary 6.2}.
\end{proof}

If follows from Lemma~\ref{l:eord} that \begin{equation} \label{e:eup}
\epsilon = x \cdot u_\fp^{\Sigma, \Sigma'}. \end{equation}
For the next lemma, recall from \S\ref{s:igs} that we have a map
 \[ \rec_G \colon (\cO_{H, \Sigma, \Sigma'}^*)_{\overline{R}} \longrightarrow I/I^2, \qquad \epsilon \mapsto 
\sum_{\sigma \in G} (\rec_\fP(\sigma(\epsilon)) - 1)\tilde{\sigma}^{-1}, \]
where $\tilde{\sigma}$ is a lift of $\sigma$ in $\fg$. 

\begin{lemma}  \label{l:erec}
We have
\[  \rec_G(\epsilon) \equiv \det(\eA_{\Sigma_\fp}) \text{  in  } I/I^2. \]
\end{lemma}

\begin{proof} 
Recall from the proof of Lemma~\ref{l:eord} that $V_\fp(H) = \Ind_1^G H_\fP^*$.  The homomorphism 
\[ \rec_{\fP} \colon H_\fP^* \longrightarrow \Gamma \cong I_\Gamma/I_\Gamma^2 \]
therefore induces a map that we again denote $\rec_\fP \colon V_\fp(H) \longrightarrow I/I^2.$
The composition 
\[ \begin{tikzcd} \cO_{H, \Sigma_\fp, \Sigma'}^* \ar[r] & V_\fp(H) \ar[r, "\rec_\fP"] & I/I^2
\end{tikzcd} \]
 is precisely the map $\rec_G$.  It follows that $\rec_G(\epsilon)$ is the determinant of the matrix $\eA_{\rec}$ in which the first column of $\eA_V$ has been replaced by $(\rec_\fP(\overline{v}_i))_{i=1}^t$.
We must prove that \[ \det(\eA_{\Sigma_\fp}) \equiv \det(\eA_{\rec}) \pmod{I^2}. \]  

To prove this claim, first note that $\overline{\eA}_{\Sigma_\fp}$ and $\eA_{\rec}$ have all columns after the first equal in $\overline{R}$.  It therefore suffices to show that the first columns of $\eA_{\Sigma_\fp}$ and $\eA_{\rec}$ are equal in $I/I^2$.  

This follows from an unwinding of the definitions.  We revisit the definition of $V_\fP(H_\fP)$ given in (\ref{e:vandw}).  In this notation, the map \[ \rec_\fP \colon V_\fP(H_\fP) \longrightarrow I_\Gamma/I_\Gamma^2 \] giving the first column of $\eA_{\rec}$ is simply induced by the canonical restriction map \[ \sigma - 1 \mapsto \sigma|_L - 1, \qquad \sigma \in \WD(H_\fP^{\ab}/F_\fp). \]
Let $w$ denote the chosen place of $L$ above $\fP$ and $\fp$.
By (\ref{e:jpdef}), the map $V_\fP(L_\fP) \longrightarrow I$ giving the first column of $\eA_{\Sigma_\fp}$
 is also induced by the restriction $\sigma - 1 \mapsto \sigma|_L - 1$ for $\sigma \in \WD(L_w^{\ab}/F_\fp)$.

To conclude, we observe that by \cite{dk}*{Lemma B.1}, the composition of the maps
\[ \begin{tikzcd}
 \Ind_{\fg_v}^{\fg} V_w(L_w) \ar[r, "\N \Gamma"] & (\Ind_{\fg_v}^{\fg} V_w(L_w))^{\Gamma} \ar[r, "\sim"] & \Ind_{1}^G V_\fP(H_\fP) 
 \end{tikzcd}
 \]
 is induced by the map $V_w(L_w) \longrightarrow V_\fP(H_\fP)$ given by restriction: $\sigma - 1 \mapsto \sigma|_{H_\fP^{\ab}}-1.$
\end{proof}

We are finally ready for:

\begin{proof}[Proof of Theorem~\ref{t:mfig}]
Define 
\[
V_{\sL}^\theta = V^\theta(L) \otimes_{\Z[\fg]} R_{\sL}, \qquad B_{\sL}^\theta = B_{\Sigma}^\theta(L) \otimes_{\Z[\fg]} R_{\sL} / (b_1 + \sL b_0) \cong (R_{\sL})^{t}. 
\]
The module $B_{\sL}^\theta $ has free generators $\overline{b}_0, \overline{b}_2, \overline{b}_3, \dotsc, \overline{b}_t$ over $R_{\sL}$.  The module $V_{\sL}^\theta$ has free generators $v_1, \dotsc, v_t$ over $R_{\sL}$.
Therefore 
\[ 
\nabla_{\!\sL} = \tilde{\nabla}_{\Sigma}^{\Sigma'}(L) \otimes_{\Z[G]} R_{\sL}/(\overline{b}_1 + \sL \overline{b}_0) 
\]
 has a quadratic $R_{\sL}$-module presentation 
\[ \begin{tikzcd}
 V_{\sL}^\theta \ar[r, "f_{\sL}"] & B_{\sL}^\theta \ar[r] & \nabla_{\!\sL} \ar[r] & 0. 
\end{tikzcd}
\]
By definition, the matrix $\eA_{\sL}$ for $f_{\sL}$ with respect to our chosen bases is the matrix
$\eA_{\Sigma_\fp}$ with the first column replaced by the first column of 
$\eA_{\Sigma_\fp} - \sL \overline{\eA}_{\Sigma}.$  Note that this first column has entries in the ideal $(I, \sL)  \subset R_{\sL}$ that is annihilated by $I$.   Furthermore $\overline{\eA}_{\Sigma_\fp}$ and $\overline{\eA}_{\Sigma}$ have columns after the first that are equal.  It follows that
\[ 
\det(\eA_{\sL}) = \det(\eA_{\Sigma_\fp}) - \sL \det(\overline{\eA}_\Sigma). 
\]
By Lemma~\ref{l:eord}, we have $\det(\overline{\eA}_{\Sigma}) = x \Theta_H$ for some $x \in \overline{R}^*$. By (\ref{e:eup}) and Lemma~\ref{l:erec}, we have
$\det(\eA_{\Sigma_\fp}) = x \cdot \rec_{G}(u_\fp^{\Sigma, \Sigma'})$ in $I/I^2$, with the same $x$. Therefore,
 \begin{align*} \label{e:ta}
 \det(\eA_{\sL}) & =   \det(\eA_{\Sigma_\fp}) - \sL \det(\overline{\eA}_{\Sigma}) \\
 & = x \cdot \rec_{G}(u_\fp^{\Sigma, \Sigma'}) - x \cdot \sL \Theta_H \\
 & = x \cdot (\rec_G(u_{\fp}^{\Sigma, \Sigma'}) - \Theta_L).
  \end{align*}
Since $x$ is a unit, the equality
 \[ \Fitt_{R_{\sL}}(\nabla_{\!\sL}) = (\det(\eA_\sL)) = (\rec_G(u_{\fp}^{\Sigma, \Sigma'}) - \Theta_L) \] follows.
Since \[ \rec_G(u_{\fp}^{\Sigma, \Sigma'}) - \Theta_L \in I/I^2, \] the second statement of the theorem follows from the first by Theorem~\ref{t:ws}.
\end{proof}

\subsection{Working componentwise} \label{s:comp}

The rings $R$ and $\overline{R}$ are not in general connected. 
In working with modular forms, it will be convenient to replace these rings with individual components.
  We will also need to extend scalars when working with Galois representations, so we do so already at this point.  Therefore let $E$ denote a finite extension of $\Q_p$ and let $\cO$ denote the ring of integers of $E$.  We assume that $E$ contains the image of every character of $\fg$.

Write $\fg = \fg_p \times \fg'$, where $\fg_p$ is the $p$-Sylow subgroup of $\fg$ and $\fg'$ is the subgroup of $\fg$ containing the elements of prime-to-$p$ order.  For each odd character $\psi$ of $\fg'$, let $R_\psi$ denote the group ring $\cO[\fg_p]$ endowed with the $\fg$-action in which $g = g_p \cdot g'$ (with $g_p \in \fg_p, g' \in \fg'$) acts by multiplication by $g_p \psi(g')$. 
 We then have an isomorphism of $\cO[\fg]$-algebras
\[ {R} \otimes_{\Z_p} \cO = \cO[\fg]^- \cong \prod_{\psi} R_\psi,  \]
where the product ranges over the odd characters $\psi$ of $\fg'$. 
We let
\[ 
R_{\sL, \psi} = R_\sL \otimes_R R_\psi, \qquad \qquad \nabla_{\!\sL,\psi}  =  \nabla_{\!\sL} \otimes_{R_\sL} R_{\sL, \psi}. 
\]

 We consider the analogous decomposition 
$G = G_p \times G'$.  If $\chi$ is an odd character of $G'$,
we let  $\overline{R}_\chi$ denote the group ring $\cO[G_p]$ endowed with the $G$-action in which $g = g_p \cdot g'$ acts by multiplication by $g_p \chi(g')$.  Then $\overline{R} \otimes_{\Z_p} \cO \cong \prod_{\chi} \overline{R}_\chi$ with the product running over the odd characters $\chi$ of $G'$.  

If $\chi$ is an odd character of $G'$, then it may be viewed as a character of $\fg'$ via the canonical projection $\fg' \longrightarrow G'$, and we may consider both $R_\chi$ and $\overline{R}_\chi$.  Furthermore, in this case if we let $I_\chi$ denote the image of the relative augmentation ideal $I$ in $R_\chi$, then we have $R_\chi / I_\chi \cong \overline{R}_\chi$.

\begin{lemma}  \label{l:fittchi} The equality
\begin{equation} \label{e:desiredfittchi} \Fitt_{R_{\sL, \chi}}(\nabla_{\!\sL, \chi}) = 0 
\end{equation}
for each odd character $\chi$ of $G'$ implies the equality 
\begin{equation} \label{e:desiredfitt}
\Fitt_{R_{\sL}}(\nabla_{\!\sL}) = 0.
\end{equation}
\end{lemma}

\begin{proof}  As $\cO$ is free (and hence faithfully flat) over $\Z_p$, in order to prove (\ref{e:desiredfitt}) it suffices to show that
\[ 
\Fitt_{R_{\sL} \otimes_{\Z_p} \cO} (\nabla_{\!\sL} \otimes_{\Z_p} \cO) = 0. 
\]
As 
\[ 
R_{\sL} \otimes_{\Z_p} \cO \cong \prod_{\psi} R_{\sL, \psi}, 
\]
with the product running over all odd characters $\psi$ of $\fg'$, it suffices to prove that
\begin{equation} \label{e:rlpsi} 
\Fitt_{R_{\sL, \psi}}(\nabla_{\!\sL, \psi}) = 0  
\end{equation}
for all such $\psi$.  
The assumption (\ref{e:desiredfittchi}) is precisely this result if $\psi$ factors through $G'$.

It therefore remains to show that (\ref{e:rlpsi}) holds if $\psi$ is an odd character of $\fg'$ that does not factor through $G'$. For such $\psi$, there exists $\sigma$ in the kernel of $\fg' \longrightarrow G'$ such that $\psi(\sigma) \neq 1$.  Since $\psi$ has prime-to-$p$ order, the element $1 - \psi(\sigma)$ is a unit in $\cO$.  It follows that the image of $1 - \sigma \in I$ in $R_\psi$ is a unit.  Since $\sL I = 0$ in $R_{\sL}$, it follows that the image of $\sL$  in $R_{\sL, \psi}$ vanishes, and hence that
\begin{equation} \label{e:rlcong}
 R_{\sL, \psi} \cong R_\psi / (\Theta_L, I^2). \end{equation}
In view of the definition (\ref{e:nablaldef}), and again applying $\sL =0$ in $R_{\sL, \psi}$, we find
\begin{align*}
  \nabla_{\!\sL, \psi} & \cong (\tilde{\nabla}_{\Sigma}^{\Sigma'}(L)_{R_{\sL, \psi}})/(\overline{b}_1)  \\
  & \cong \nabla_{\Sigma_\fp}^{\Sigma'}(L)_{R_{\sL, \psi}},
\end{align*}
since $\tilde{\nabla}_{\Sigma}^{\Sigma'}(L) / \overline{b}_1 \cong \nabla_{\Sigma_\fp}^{\Sigma'}(L)$ by the definitions of these modules.
Now \[ \Fitt_R( \nabla_{\Sigma_\fp}^{\Sigma'}(L)_R) = (\Theta_L) \] by \cite{dk}*{Theorem 3.3} as stated in Theorem~\ref{t:dkresult} above.
The desired result (\ref{e:rlpsi}) follows since $\Theta_L = 0$ in $R_{\sL, \psi}$ by (\ref{e:rlcong}).
\end{proof}

\begin{lemma} \label{l:k} Let $\chi$ be an odd character of $G'$ and let $\fK \subset \fg', K \subset G'$ denote the kernels of $\chi$
when viewed as characters of $\fg'$ and $G'$, respectively.  Let $R_{\sL, \chi}'$ and $\nabla_{\!\sL, \chi}'$ denote the ring $R_{\sL, \chi}$ and the module $\nabla_{\!\sL, \chi}$, respectively, defined using the fields $L^{\fK}$ and $H^K$ in place of $L$ and $H$.  Then there exist canonical isomorphisms
\[ 
R_{\sL, \chi}' \cong R_{\sL, \chi}, \qquad  \nabla_{\!\sL, \chi}' \cong  \nabla_{\!\sL, \chi}. 
\]
\end{lemma}

\begin{proof}  There is clearly an $\cO[\fg]$-algebra isomorphism $R_\chi' \cong R_\chi$, as these are both the group ring $\cO[\fg_p]$ in which $g=g_p g'$ acts via $\chi(g')$.  The elements $\Theta_{L^{\fK}}$ and $\Theta_L$ correspond under this isomorphism.  Similarly we have an isomorphism $\overline{R}_\chi' \cong \overline{R}_\chi$ under which $\Theta_{H^K}$ and $\Theta_H$ correspond.
The $\cO[\fg]$-algebra isomorphism $R_{\sL,  \chi}' \cong R_{\sL,  \chi}$ follows immediately from these considerations.

For the second isomorphism of the lemma, we first show that 
\begin{equation} \label{e:nablachi}
\tilde{\nabla}_{\Sigma}^{\Sigma'}(L)_{R_\chi} \cong 
\tilde{\nabla}_{\Sigma}^{\Sigma'}(L^{\fK})_{R_\chi}. 
\end{equation}  
For this, we note that by \cite{dk}*{Lemma B.1}, there is an isomorphism
\[ 
V^\theta(L^{\fK}) \cong (\N {\fK}) V^\theta(L) \subset V^\theta(L), 
\] 
yielding a commutative diagram
\[ 
\begin{tikzcd}
V^\theta(L^{\fK}) \ar[r] \ar[d] & B^\theta_{\Sigma}(L^{\fK}) \ar[r] \ar[d, "\N \fK"] & \tilde{\nabla}_{\Sigma}^{\Sigma'}(L^{\fK}) \ar[r] \ar[d] & 0 \\
V^\theta(L) \ar[r] & B^\theta_{\Sigma}(L) \ar[r]  & \tilde{\nabla}_{\Sigma}^{\Sigma'}(L) \ar[r] & 0.
\end{tikzcd}
\]
Upon tensoring with $R_\chi$, the arrow labelled $\N \fK$ becomes an isomorphism.  Indeed, $\fK$ acts trivially on $R_\chi$, whence $\N \fK$ acts as $\#\fK$, which is prime-to-$p$ and hence invertible in $\Z_p$.  The isomorphism (\ref{e:nablachi}) follows.

The desired isomorphism $\nabla_{\!\sL, \chi}' \cong  \nabla_{\!\sL, \chi}$ now follows from the definition (\ref{e:nablaldef}).
\end{proof}

In view of Lemma~\ref{l:k}, we may (and do) hereafter replace $(L, H)$ by $(L^{\fK}, H^K)$ and therefore assume that $\fg' = G'$ and that $\chi$ is a faithful character of $G'$.  In particular, $G'$ is cyclic and \[ \Gamma = \ker(\fg \longrightarrow G)\] is a $p$-group.

\section{Group ring valued Hilbert modular forms} \label{s:cusp}

Let $m$ be a positive integer. Let $\chi$ be an odd character of $G'$.
In this section, we use the theory of group ring valued Hilbert modular forms to produce an $R_{\sL, \chi}$-module 
$\tilde{B}_p$  and a cohomology class $\kappa \in H^1(G_F, \tilde{B}_p)$ satisfying the conditions of Theorem~\ref{t:nl}.
At some point, we will have to assume that $\fp$ is not the only prime of $F$ above $p$ (the case of one prime above $p$ will be handled in 
\S\ref{s:onep}).
We will calculate that \[  \Fitt_{R_{\sL, \chi}}(\tilde{B}_p) \subset (p^m), \] which in conjunction with the $R_{\sL, \chi}$-module surjection 
  \[ \nabla_{\!\sL, \chi} \longrightarrow \tilde{B}_p \] of Theorem~\ref{t:nl} yields
   \[ \Fitt_{R_{\sL, \chi}}(\nabla_{\!\sL, \chi}) \subset (p^m). \]
   Since this holds for all $m$, we have $\Fitt_{R_{\sL, \chi}}(\nabla_{\!\sL, \chi}) = 0$.   Lemma~\ref{l:fittchi} implies $\Fitt_{R_{\sL}}(\nabla_{\!\sL}) = 0$, which by
    Theorem~\ref{t:mfig} completes the proof of the $p$-part of Gross's Conjecture.
   
   We refer the reader to \cite{dk}*{Section 7} for our definitions on group ring valued Hilbert modular forms, recalling only the essential notation here.  Let $\fn \subset \cO_F$ denote an integral ideal and $k \ge 1$ a positive integer. We let $M_k(\fn)$ denote the space of Hilbert modular forms of level $\fn$ and weight $k$.  The subgroup of forms whose $q$-expansion coefficients at all unramified cusps lie in $\Z$ is denoted $M_k(\fn, \Z)$.  If $A$ is any abelian group, we let $M_k(\fn, A) = M_k(\fn, \Z) \otimes A$.  The space $M_k(\fn, \Z)$ is endowed with an action of ``diamond operators" $S(\fm)$, indexed by the classes $\fm \in G_\fn^+$, the narrow ray class group of $F$ associated to the conductor $\fn$.  
   Suppose now that $R$ is a ring and $\bpsi \colon G_\fn^+ \longrightarrow R^*$ is a character.  Suppose that the abelian group $A$ has an $R$-module  structure.  Then  we define
   \[ M_k(\fn, A, \bpsi) = \{ f \in M_k(\fn, A) \colon f_{|S(\fm)} = \bpsi(\fm) f  \text{ for all } \fm \in G_\fn^+ \}. \]
   These are the forms of nebentypus $\bpsi$.  In our applications below, $R$ is a group ring (or a factor of a group ring), and $\bpsi$ is the tautological character.  For this reason, we call $M_k(\fn, A, \bpsi)$ the space of {\em group ring valued modular forms}.   We write $S_k(\fn, A, \bpsi) \subset M_k(\fn, A, \bpsi)$ for the subspace of cusp forms.
   
\subsection{The modified group ring Eisenstein series} 

We begin by recalling the reductions of previous sections.  By the results of \S\ref{s:remove}, we may assume that $T$ (and hence $\Sigma'$) contains no primes above $p$.  By the results of \S\ref{s:comp} we may assume that
 $\fg' = G'$ and that $\chi$ is a faithful odd character of $G'$.  In particular, $G'$ is cyclic and $\Gamma = \ker(\fg \longrightarrow G)$ is a $p$-group.

Let 
\begin{equation} \label{e:conds}
 \fn_0 = \cond(L/F),  \qquad \fn = \lcm(\fn_0, \ \prod_{\mathclap{\fq \in \Sigma_\fp \cup \Sigma'}} \fq \, ). 
\end{equation}
  Let $\fP$ be the $p$-part of $\fn$.  Note that $\fP \neq 1$ as $\fp \mid \fP$.
 We write
 \[ A = R_\chi = \cO[\fg_p]_{\chi}, \qquad \overline{A} = \overline{R}_\chi = \cO[G_p]_\chi. \] 
  There are canonical $\cO[\fg]$-algebra injections with finite cokernel:
  \[ A \hookrightarrow \prod_{\mathclap{\substack{\psi \in \hat{\fg} \\ \psi|_{G'} = \chi}}} \cO_\psi, \qquad \overline{A} \hookrightarrow \prod_{\mathclap{\substack{\psi \in \hat{G} \\ \psi|_{G'} = \chi}}} \cO_\psi, \qquad x \mapsto (\psi(x))_\psi. \]
 Here $\cO_\psi$ denotes the ring $\cO$ on which $\fg$ acts by the character $\psi$.  We call the characters indexing these products the {\em characters of $A$} and the {\em characters of $\overline{A}$,} respectively.  In particular, a character of $\overline{A}$ is simply a character of $A$ that is trivial on $\Gamma$.
  
 \begin{lemma} \label{l:psicond}
  Let $\psi$ be any character of $A$, and let $\fc_0 = \cond(\psi)$. 
    Put $\fc = \lcm(\fc_0, \fP)$, and $\fl = \fn/\fc$.  Then $\fl$ is a square-free product of primes not dividing $p$.
 \end{lemma}
  
  \begin{proof}  The fact that the primes dividing $\fl$ do not lie above $p$ is clear since $\fP$ is the $p$-part of $\fn$ and $\fP \mid \fc$.  To prove that $\fl$ is square-free, suppose that $\fq$ is a prime not lying above $p$ such that $\fq^m \mid\mid \fn$ with $m \ge 2$.  It suffices to show that $\fq^m \mid \fc$, whence $\fq \nmid \fl$.
  By the definition of $\fn$, we must have $\fq^m \mid \fn_0 = \cond(L/F)$.
  The proof now follows exactly as \cite{dk}*{Lemma 8.13}. 
  \end{proof}
  
  Let $\psi$ be a character of $A$.  Denote by $\psi_\fP$ the character $\psi$ viewed with modulus divisible by all primes dividing $\fP$. In \cite{dk}*{Definition 8.2}, we defined the following linear combination of level-raised Eisenstein series (where $\fl$ is as in Lemma~\ref{l:psicond}):
  \begin{equation} \label{e:wkpsi1def}
{W}_k(\psi_{\fP}, 1) = \sum_{\fa \mid \fl}  \mu(\fa)\psi(\fa) \N\fa^k E_k(\psi_{\fP},1)|_{\fa}  \in M_k(\fn, \psi).
\end{equation}
Here $\mu$ is the M\"obius function, which for squarefree ideals $\fa$ that are the product of $r$ distinct primes
satsifies $\mu(\fa)=(-1)^r$.

As we now recall, we showed in \cite{dk}*{\S8} that the forms $W_k(\psi_{\fP}, 1)$ interpolate into a group ring-valued family of modular forms, and we calculated the constant terms of this family at certain cusps.  Let 
 \[ \bpsi \colon \fg \longrightarrow A^*, \qquad g=g_p g' \mapsto g_p \chi(g') \] denote the canonical character.  We recall from \cite{dk}*{\S7.2.3} our notation on the set of cusps of level $\fn$, denoted $\cusps(\fn)$.
 A cusp $[\cA]$ is represented by a pair $\cA =(M, \lambda)$ where $M \in \GL_2^+(F)$, the set of all $2 \times 2$ matrices over $F$ with totally positive determinant, and $\lambda \in \Cl^+(F)$, the narrow  class group of $F$.  The definition of $M_k(\fn)$ involves the choice, for each $\lambda \in \Cl^+(F)$, of a representative ideal $\ft_\lambda$.  Writing $\fd$ for the different of $F$ and letting $M = \mat{a}{b}{c}{d}$, we define the ideals
 \[ \fb_\cA = a\cO_F + c(\ft_\lambda \fd)^{-1}, \quad \fc_{\cA} = c(\ft_\lambda \fd \fb_\cA)^{-1}. \]
 Following \cite{dk}*{\S7.2.5} we define for any integral ideal $\fb \mid \fn$:
 \begin{align*}
C_\infty(\fb, \fn) &= \{ [\cA] \in \cusps(\fn)\colon \fb \mid \fc_\cA\} \\
C_0(\fb, \fn) &= \{ [\cA] \in \cusps(\fn)\colon \gcd(\fb, \fc_\cA)=1\}. 
 \end{align*}
 If $\fb = \fn$, we simply write $C_\infty(\fn)$ and $C_0(\fn)$.  The notion of {\em normalized constant term} of a modular form at a cusp is defined in \cite{dk}*{\S7.2.3}.
In the remainder of the paper, we write for simplicity $\Theta_H = \Theta_{\Sigma, \Sigma'}^H$ and  $\Theta_L= \Theta_{\Sigma_\fp, \Sigma'}^L$   as in \S\ref{s:rw} above. 

 \begin{prop}  \label{p:wconstantterms}  There exists  a group ring valued form 
\begin{equation} \label{e:wbpsi}
W_1(\bpsi, 1) \in M_1(\fn,A, \bpsi) 
\end{equation}
such that the specialization of $W_1(\bpsi, 1)$ at character $\psi$ of $A$ is the form 
$W_1(\psi_{\fP}, 1)$.
The normalized constant term of $W_1(\bpsi, 1)$ at a cusp $\cA$ such that $[\cA] \in C_\infty(\fP, \fn)$ 
is 
 \begin{equation} \label{e:wkconst}
C_1^L(\cA) = \begin{cases}
\sgn(\N a) \bpsi^{-1}(a \fb_\cA^{-1})\Theta_L^\#/2^n  & \text{if } [\cA] \in C_\infty(\fn) \\
0 &  \text{if } [\cA] \in C_\infty(\fP,\fn)  \setminus C_\infty(\fn).
\end{cases}
\end{equation}
\end{prop}

Recall here that $\#$ denotes the involution on ${R}$ induced by $g \mapsto g^{-1}$ for $g \in \fg$.

\begin{proof}  For any odd $k \ge 1$, 
a group ring valued form \[ W_k(\bpsi, 1) \in M_k(\fn, \Frac(A), \bpsi) \] is defined in \cite{dk}*{Proposition 8.14, eqn.~(102)}, with notation from \emph{loc.\,cit.}, by
 \[
 W_k(\bpsi, 1) = \sum_{\fm \mid \fl} \N I_{\fm} \cdot \tilde{E}_k(\bpsi^{\fm}, 1)|_{\fm} \bpsi^{\fm}(\fm) \frac{1}{\# I_{\fm}} \prod_{v \mid \fm} (1- \N v^k).
 \]
Note that this definition and the proof that $W_k(\bpsi,1)$ specializes under a character $\psi$ of $A$ to $W_k(\psi_\fP, 1)$ uses the result of Lemma~\ref{l:psicond}.  All the normalized Fourier coefficients of $W_k(\bpsi,1)$ lie in $A$, except for possibly the constant terms, which lie in $\Frac(A)$.

 As we now explain,  \cite{dk}*{Proposition 8.7} implies that the constant term of $W_1(\psi_\fP, 1)$ at a cusp $[\cA] \in C_\infty(\fP, \fn) \setminus C_\infty(\fn)$ vanishes. First note that $\fP \neq 1$ and hence $C_0(\fP, \fn) \cap C_{\infty}(\fP, \fn) = \emptyset$.
  As in \cite{dk}*{\S8.1}, write \[ \fc_0 = \cond(\psi), \fc = \lcm(\fc_0, \fP), \text{ and } \ft = \fn/\fc. \]  
By Lemma~\ref{l:psicond}, $\ft$ is a squarefree product of primes not lying above $p$.
We therefore have \[ C_\infty(\fP, \fn) \cap C_\infty(\fc_0 \ft, \fn) = C_\infty(\fn). \]  Hence if $[\cA] \in C_\infty(\fP, \fn) \setminus C_\infty(\fn)$, then 
\cite{dk}*{Prop 8.7} yields that the constant term of $W_1(\psi_\fP, 1)$ at $[\cA]$ vanishes in both the cases $\fc_0 = 1$ and $\fc_0 \neq 1$.

 At a cusp $[\cA] \in C_\infty(\fn)$ the contant term of $W_1(\psi_\fP, 1)$ equals
\[ 
\sgn(\N a) \psi^{-1}(a \fb_\cA^{-1})\frac{L_{\Sigma_\fp, \Sigma'}(\psi, 0)}{2^n}.
\]
The element of $A$ interpolating these specializations is the element $C^L_1(\cA)$ defined in (\ref{e:wkconst}).  Since $C_1^L(\cA) \in A$, we obtain $W_1(\bpsi, 1) \in M_1(\fn,A, \bpsi)$. 
\end{proof}

We have an analogous construction of group ring forms over $G$.  We write \[ \overline{A} = \overline{R}_\chi = \cO[G_p]_\chi. \]  Let
\[ \overline{\bpsi} \colon \fg \longrightarrow \overline{A}^* \]
denote the reduction of $\bpsi$.  Note that $\overline{\bpsi}$ factors through $\fg \longrightarrow G$.
As in (\ref{e:conds}), we let
\[ \fm_0 = \cond(H/F), \qquad \fm = \lcm(\fm_0, \ \prod_{\mathclap{\fq \in \Sigma \cup \Sigma'}} \fq \, ). \]
Note that here, $\fp$ does not divide the level $\fm$.  Since we have $\fg' = G'$, however, note that the levels $\fn$ and $\fm$ agree away from $p$, i.e. if $\fP_0$ denotes the $p$-part of $\fm$, then $\fn/\fP = \fm/\fP_0$.
(This fact implies that the results of Propositions~\ref{p:wct1b} and~\ref{p:wct2}, which {\em a priori} would apply to the cusps in $C_\infty(\fP_0, \fm)$, also apply on $C_\infty(\fP, \fn)$.)

We use  \cite{dk}*{Proposition 8.14, eqn.~(102)} again to define, for odd $k \ge 1$,  a group ring valued modular form 
\begin{equation} \label{e:fraca}
 W_k(\overline{\bpsi}, 1) \in  M_k(\fm, \Frac(\overline{A}), \overline{\bpsi}) \subset M_k(\fn, \Frac(\overline{A}), \overline{\bpsi}).
 \end{equation}
Again, the $q$-expansion coefficients of $W_k(\overline{\bpsi}, 1)$ other than possibly the constant terms lie in $\overline{A}$. 
The specialization of $W_k(\overline{\bpsi}, 1)$ at a character $\psi$ of $\overline{A}$ is ${W}_k(\psi_{\fP_0}, 1)$, defined as in (\ref{e:wkpsi1def}) with $\fP$ replaced by $\fP_0$.

We now discuss the constant terms of $W_k(\overline{\bpsi}, 1)$ for odd $k \ge 1$.
The following result  follows directly from  \cite{dk}*{Proposition 8.7}. 
 \begin{prop} \label{p:wct1b}  
If $\Sigma \setminus S_\infty$ is nonempty, the normalized constant term of $W_1(\overline{\bpsi},1)$ at a cusp $\cA$ such that $[\cA] \in C_\infty(\fP, \fn)$ is
\[ 
C^H_1(\cA) = 
\begin{cases}
\sgn(\N a) {\overline{\bpsi}}^{-1}(a \fb_{\cA}^{-1})\Theta_H^\#/2^n & \text{if } [\cA] \in C_\infty(\fn) \\
0 &  \text{if } [\cA] \in C_\infty(\fP, \fn) \setminus C_\infty(\fn).
\end{cases}
\]
 In particular, if $\Sigma \setminus S_\infty$ is nonempty we have $W_1(\overline{\bpsi},1) \in M_1(\fn, \overline{A}, \overline{\bpsi})$. 
\end{prop}

If $\Sigma = S_\infty$ (in which case $\fP_0 = 1$) the situation is more complicated for $W_1(\overline{\bpsi},1)$. 
Furthermore in this case we  require the constant terms of the forms $W_k(\overline{\bpsi},1)$ for odd $k\ge 3$.
 
For each character $\psi$ of $\overline{A}$, let $\fb_0 = \cond(\psi)$, $\fb = \lcm(\fb_0, \fP_0)$, and $\fl = \fm/\fb$.
If $[\cA]$ is a cusp, we follow \cite{dk}*{Definition 8.3} and define two sets of primes:
\[ J_\fl = \{ \fq \mid \fl \colon [\cA] \in C_0(\fq, \fn) \}, \qquad  J_\fl^c =  \{ \fq \mid \fl \colon [\cA] \in C_\infty(\fq, \fn) \}. \]
Define $C_k'(\cA)$ to be the unique element of
 $\Frac(\overline{A})$ 
 such that for each character $\psi$ of $\overline{A}$, we have
\begin{equation} \label{e:cpdef}
 \psi(C_k'(\cA)) =  \frac{\tau(\psi)}{\N\fb_0^k} \psi(\fc_{\cA}) \sgn(\N(-c)) \frac{L_{S_{\infty}, \emptyset}(\psi^{-1}, 1-k)}{2^n} 
 \prod_{\fq \in J_{\fl}^c}(1 - \N \fq^k) \prod_{\fq \in J_{\fl}}(1- \psi(\fq))
 \end{equation}
if $[\cA] \in C_0(\fb_0, \fn)$, and \[  \psi(C_k'(\cA)) = 0 \] if $[\cA] \in C_\infty(\fP, \fn) \setminus C_0(\fb_0, \fn)$.  Here $\tau(\psi)$ denotes the Gauss sum defined in \cite{dka}*{Definition 4.1}.  The following proposition follows directly from \cite{dk}*{Propositions 8.6 and 8.7}.

\begin{prop} \label{p:wct2}  Suppose that $\Sigma = S_\infty$. The normalized constant term of $W_k(\overline{\bpsi}, 1)$ at a cusp $[\cA] \in C_\infty(\fP, \fn)$ is equal to $C^H_1(\cA) + C_1'(\cA)$ if $k = 1$, and is equal to $C_k'(\cA)$ if $k > 1$.
\end{prop}

\subsection{Construction of a cusp form} \label{s:ccf}

In this section we construct the cusp form required in our proof. The construction will need to be split into several cases. 
Recall that $\fP$ is the $p$-part of $\fn$.  The ideal $\fP$ is divisible precisely by the primes in $\Sigma_\fp =\Sigma \cup \{ \fp \}$.
Define $\fP'$ to be the product of all other primes above $p$, i.e.
\[ 
\fP' = \prod_{\mathclap{\fq \mid p, \ \fq \nmid \fP}} \fq. 
\]

\begin{itemize}
\item Case 1: the set $\Sigma \setminus S_\infty$ is non-empty, i.e. $\fp$ is not the only prime dividing $\fP$. 
\end{itemize}
In Case 2, we have $\Sigma = S_\infty$.
Note that the eigenvalues of the form $W_1(\bpsi, 1)$ for the operator $U_\fq$, $\fq \mid \fP'$, are $\bpsi(\fq)$ and $1$.  If $\chi(\fq) \neq 1$, then these are not congruent modulo the maximal ideal of $A$.  To ensure that the Hecke algebras we work with are local, it will be convenient to project to the eigenspace with eigenvalue $1$.
Let 
\[ 
\fP'_0  =\prod_{\mathclap{\substack{\fq \mid \fP' \\ \chi(\fq) \neq 1}}} \fq, \qquad 
\fP'_1  =\prod_{\mathclap{\substack{\fq \mid \fP' \\ \chi(\fq) = 1}}} \fq. 
\]
We subdivide Case 2 into three cases: 
\begin{itemize} \item Case 2(a): $\fP'_0 \neq 1$.
\item Case 2(b): $\fP'_0 = 1, \fP'_1 \neq 1$.
\item Case 2(c): $\fP'_0 = \fP'_1 = 1$, i.e.\ $\fp$ is the only prime above $p$ in $F$.
\end{itemize}

We consider the module \[ \Lambda = A \oplus \overline{A},\]
endowed with the canonical diagonal $A$-action.
  We view $A = A \oplus 0$ as a submodule of $\Lambda$ and denote the vector $(0,1) \in \Lambda$ by $\lambda$, so elements of $\Lambda$ will be written $a + \lambda b$ with $a \in A, b \in \overline{A}.$

 As in the previous section, we have the canonical characters  \[ 
\bpsi\colon \fg \longrightarrow A^*, \qquad \overline{\bpsi}  \colon \fg \longrightarrow G \longrightarrow \overline{A}^*. 
 \]
 We denote the image of $I$ in $A$ (i.e.\ the kernel of the canonical projection $A \longrightarrow \overline{A}$) by $I_A$.
 
Recall the following result of Silliman \cite{silliman}*{Theorem 8.10}, a  generalization of a result of Hida.
\begin{theorem} \label{t:hida}  Fix a positive integer $m$.  For positive integers $k \equiv 0 \pmod{(p-1)p^{N}}$ with $N$ sufficiently large depending on $m$, 
there is a modular form $V_{k} \in M_k(1, \Z_p, 1)$ such that 
$c(\fm, V_{k})  \equiv 0 \pmod{p^m}$ for all integral ideals $\fm$, and such that the normalized constant term  $c_{\cA}(0,V_k)$
 for each cusp $[\cA]$ is congruent to $1 \pmod{p^m}$.
\end{theorem}

To apply this result, we hereafter assume that $m$ is fixed and $k \equiv 1 \pmod{(p-1)p^{N})}$ with $N$ sufficiently large that the conclusion of Theorem~\ref{t:hida} holds.  In  subsequent arguments we will make $N$ larger still if necessary to obtain other properties of our modular forms.

In Case 1, we define
\[ 
  f = (W_1( \bpsi,1) - \lambda W_1(\overline{\bpsi},1)) V_{k-1} \in M_k(\fn, \Lambda, \bpsi).
\]
Note here that $W_1( \bpsi,1) \in M_k(\fn, A, \bpsi)$ and $W_1(\overline{\bpsi},1) \in  M_k(\fn, \overline{A}, \overline{\bpsi})$ by
Propositions~\ref{p:wconstantterms} and~\ref{p:wct1b}, respectively.

In Case 2 we
recall the non-zerodivisor 
 \begin{equation} \label{e:xdef}
 x = x(k) = \frac{\Theta^{H/F}_{S_{\infty}, \emptyset}(1-k)}{\Theta^{H/F}_{S_{\infty}, \emptyset}} \in \Frac(\overline{A}) \end{equation}
 considered in \cite{dk}*{Theorem 8.16}.
 The elements $\Theta^{H/F}_{S_{\infty}, \emptyset}(1-k)$ and $\Theta^{H/F}_{S_{\infty}, \emptyset}$ belong to $\Frac(\overline{A})$ and interpolate the nonzero algebraic numbers 
$L_{S_\infty, \emptyset}(\psi^{-1}, 1-k)$ and $L_{S_{\infty}, \emptyset}(\psi^{-1}, 0)$, respectively, as $\psi$ ranges over the odd characters of $G$ that restrict to $\chi$ on $G'$. For $k \equiv 1 \pmod{(p-1)p^N)}$ with $N$ sufficiently large, the ratio $x$ belongs to $\overline{A}$ by {\em loc.\,cit.}  We hereafter assume that $N$ is chosen so that this holds.
 Let \begin{equation} \label{e:tx}
 \tilde{x} = \text{ any non-zerodivisor lift of }x \text{ to }A.  
 \end{equation}
We define in Case 2:
\begin{equation} \label{e:fdef2ab}
 f = \tilde{x}  \left(W_1( \bpsi,1) - \lambda W_1(\overline{\bpsi},1)\right) V_{k-1} + \lambda W_k(\overline{\bpsi}, 1).
\end{equation}

{\em A priori}, in Case 2 the form $f$ lies in $M_k(\fn, \Lambda \otimes_{\Z_p} \Q_p, \bpsi)$ since we have only shown that the constant terms of the forms $W_1(\overline{\bpsi}, 1)$ and $W_k(\overline{\bpsi},1)$ lie in $\Frac(\overline{A}) = \overline{A} \otimes_{\Z_p} \Q_p$.  However, the proposition below shows that in fact $f \in M_k(\fn, \Lambda, \bpsi)$.  
In order to state results in all cases simultaneously, we set $\tilde{x} = 1$ in Case 1 for the remainder of the paper.

\begin{prop} \label{p:fcusp} With $N$ and $k$ as above, we have in all cases $f \in M_k(\fn, \Lambda, \bpsi)$.  Furthermore the form $f$ has constant terms at cusps $[\cA] \in C_\infty(\fP, \fn)$
lying in the submodule $ J \subset \Lambda$, where
\begin{equation}  \label{e:jdef}
J =   \tilde{x}J_0, \qquad J_0 = \left(I_A^2, \Theta_L^\# - \Theta_H^\# \lambda, \  p^m\Lambda\right). 
\end{equation}
\end{prop}

\begin{proof} We give the proof in Case 2, as Case 1 is similar and in fact easier.

We first note that the non-constant term $q$-expansion coefficients of $f$ lie in $\Lambda$.  
Indeed, all the forms appearing in the definition of $f$ have non-constant term $q$-expansion coefficients lying in $\Lambda$,
 so any failure of integrality of non-constant terms can  arise only from multiplying the constant terms of $W_1(1, \overline{\bpsi})$ by the non-constant terms of $V_{k-1}$. For $k \equiv 1 \pmod{(p-1)p^{N}}$ and $N$ sufficiently large, the non-constant terms of $V_{k-1}$ are divisible by any desired power of $p$, and this product will be integral.

We will prove the integrality of the constant terms of $f$ and the statement about cuspidality simultaneously.
By Propositions~\ref{p:wconstantterms} and \ref{p:wct2}, the constant term of $f$ at a cusp $[\cA] \in C_\infty(\fP, \fn)$ is equal to
\begin{align*}
 & \left( C^L_1(\cA) - \lambda C^H_1(\cA)- \lambda C_1'(\cA)\right) \tilde{x} y + \lambda  C_k'(\cA)  \\
 \equiv \ &  \left( \tilde{x}C_1^L(\cA) - \tilde{x} \lambda C_1^H(\cA)\right) - 
\lambda \left(C_1'(\cA) \tilde{x}y - C_k'(\cA)\right), 
\end{align*}
where $y = c_{\cA}(0, V_{k-1}) \equiv 1 \pmod{p^m}$ and the congruence is modulo $\tilde{x} p^m \Lambda \subset J$. For the first term, we have
\[
 \left( \tilde{x} C_1^L(\cA) - \tilde{x} \lambda C_1^H(\cA) \right) =  \frac{\sgn(\N a) \bpsi^{-1}(a \fb_{\cA}^{-1})}{2^n}\left(\tilde{x}\Theta_L^\# - \tilde{x}\Theta_H^\# \lambda \right) \in J. \]
It remains to show that $x y C_1'(\cA) -   C_k'(\cA)$ lies in $\overline{A}$ and  is divisible by $x p^m$. Now
$C_1'(\cA)$ is some fixed element of $\Frac(\overline{A})$ and $x \in \overline{A}$, so for $y$ sufficiently close to 1 $p$-adically, $(y - 1) C_1'(\cA)$ will lie in $p^m \overline{A}$.  Therefore it suffices to show that $x C_1'(\cA) -  C_k'(\cA)$ lies in $\overline{A}$ and  is divisible by $x p^m$, i.e.\ that $C_1'(\cA) -  x^{-1} C_k'(\cA)$ lies in $\overline{A}$ and  is divisible by $p^m$. 

For this, we note that by the definition (\ref{e:xdef}), multiplying by the factor $x^{-1}$ exactly replaces the $L$-value in the definition (\ref{e:cpdef}) of $C_k'(\cA)$ with that appearing in $C_1'(\cA)$, hence
\[
 C_1'(\cA) -  x^{-1}\Theta_H C_k'(\cA) = C_1'(\cA) \left( 1 - \N\fb_0^{1-k} \prod_{\fq \in J_{\fl}^c}\frac{1-\N\fl_i^k}{1 - \N\fl_i} \right).
\]
For  $k \equiv 1 \pmod{(p-1)p^{N}}$ and $N$ sufficiently large, the term in parentheses can be made divisible by arbitrarily large powers of $p$. (Note we are in Case 2, whence $\fb_0$ is prime to $p$.)
The result follows.
\end{proof}

\begin{corollary} \label{c:gdef}
There exists a $p$-ordinary cusp form $g \in S_k(\fn\fP', \Lambda, \bpsi)^{p\text{-ord}}$ such that in Cases 1 and 2(a) we have
\[ 
g \equiv W_1( \bpsi_{\fP_0'},1) - \lambda W_1(\overline{\bpsi}_{\fP_0'},1) \pmod{J_0},
\]
while in Cases 2(b) and 2(c) we have
\[
g \equiv \tilde{x} \big( W_1(\bpsi, 1) -\lambda W_1( \overline{\bpsi}, 1) \big) + \lambda W_k(1, \overline{\bpsi}_p) \pmod{J}.
\]
These congruences are understood to mean that the $q$-expansion coefficients $c(\fa, g)$ for nonzero ideals $\fa \subset \cO_F$ are congruent modulo $J_0$ or $J$ to the coefficients of the expressions on the right.
\end{corollary}

\begin{proof}
Silliman's result \cite{silliman}*{Theorem 8.4} 
 implies that there is an element $f' \in M_k(\fn, J, \bpsi)$ whose constant terms at $C_\infty(\fP, \fn)$ agree with those of $f$.  Therefore $f - f'$ has constant terms that vanish on $C_\infty(\fP, \fn)$.
 Let $e_\fP^{\ord}, e_p^{\ord}$ denote the ordinary operators of Hida (see \cite{dk}*{\S7.2.9}).
By
 \cite{dka}*{Theorem 5.1}, the form $g_0 = e_\fP^{\ord}(f - f')$ is cuspidal,
 whence \[ g = e_p^{\ord}(f -f') = e_{\fP'}^{\ord} g_0\] is cuspidal as well. Now the ordinary operator $e_p^{\ord}$ fixes the forms $W_1(\bpsi,1)$ and $W_1(\overline{\bpsi}, 1)$, whereas it sends $W_k(\overline{\bpsi}, 1)$ to its ordinary $p$-stabilization $W_k(\overline{\bpsi}, 1_p)$.  Furthermore, $V_{k-1} \equiv 1 \pmod{p^m}$ and $p^m \in J_0$. The result follows in Cases 2(b) and 2(c).
 
To complete the proof in Cases 1 and 2(a), we apply the operator $\prod_{\fq \mid \fP'_0} (U_\fq - \bpsi(\fq))$.  This operator sends 
\[ W_1(\bpsi, 1) \mapsto W_1(\bpsi_{\fP'_0}, 1), \qquad W_1(\overline{\bpsi}, 1) \mapsto W_1(\overline{\bpsi}_{\fP'_0}, 1), \] while it annihilates $W_k(\overline{\bpsi}, 1_p).$  This gives the result in Case 1.  In Case 2(a), we obtain a cusp form $g$ satisfying
 \[ g \equiv  \tilde{x}  \left(W_1( \bpsi_{\fP'_0},1) - \lambda W_1(\overline{\bpsi}_{\fP_0'},1)\right)
\pmod{\tilde{x}J_0}. \]
This congruence in particular implies that the Fourier coefficients of $g$ are divisible by $\tilde{x}$.  Since $\tilde{x}$ is a non-zerodivisor, we may divide by it and obtain a cusp form satsifying the same congruence as in Case 1.
\end{proof}

\subsection{Homomorphism from the Hecke algebra} \label{s:hha}

We now define certain Hecke algebras acting on $S_k(\fn\fP', A, \bpsi)$.  
We define $\T$ to be the $A$-subalgebra of $\End_A(S_k(\fn\fP', A, \bpsi))$ generated over $A$ by the following operators:
\begin{itemize}
\item $T_\fl$ for $\fl \nmid \fn\fP'$,
\item $S(\fa)$ for $(\fa, \fn\fP') = 1$, 
\item  $U_\fq$ for prime $\fq \in \Sigma \setminus S_\infty$ (i.e.\ for prime $\fq \mid \fP$, $\fq \neq \fp$),
\item $(U_\fp - 1)^2$, and
\item $(U_\fp - 1)t$ for $t \in I$. 
\end{itemize}

 Let $\tilde{\T}= \T[U_\fp]$.
Finally let $\T^\dagger = \tilde{\T}[U_\fq, \fq \mid \fP']$.\\

Due to the presence of the involution $\#$ in Proposition~\ref{p:fcusp}, we  consider an ``involuted" version of the ring $A_\sL$.  Define
\[
A_{\sL}^\# = A[\sL]/(\sL (\Theta_H)^\# - (\Theta_L)^\#, \sL I_A, \sL^2, I_A^2).
\]

At this point, we must momentarily abandon case 2(c), when $\fp$ is the only prime of $F$ above $p$.  In this case, the last bulleted point of the theorem below---which in some sense is the most important---does not hold.  Here $\fP' = 1$, and this last bulleted point then reads $\Ann_{A_\sL^\#}(W) \subset (p^m)$, which does not hold in case 2(c).  The algebra $W$ constructed in Theorem~\ref{t:heckehom} using the form $g$ is not large enough for our purposes.  In \S\ref{s:onep} we will therefore provide a different approach to proving the congruence (\ref{e:mgc}) that only applies when there is exactly one prime of $F$ above $p$.
 For the remainder of \S\ref{s:cusp} we assume that there is at least one prime above $p$ in $F$ other than $\fp$.

\begin{theorem}  \label{t:heckehom} Suppose we are in Cases 1, 2(a), or 2(b), i.e.\ there exists a prime of $F$ above $p$ other than $\fp$.  There exists an $A_{\sL}^\#$-algebra $W$ 
 and a surjective $A$-algebra homomorphism \[ \varphi \colon \T^\dagger \longrightarrow W \] such that:
\begin{itemize}
\item $T_\fl \mapsto \N\fl^{k-1} + \bpsi(\fl)$ for prime $\fl \nmid \fn\fP'$.  
\item $S(\fa) \mapsto \bpsi(\fa)$ for $(\fa, \fn\fP') = 1$.
\item $U_\fq \mapsto 1$ for prime $\fq \in \Sigma \setminus S_\infty$.
\item $U_\fp \mapsto 1 - \sL$.
\item $U_\fq \mapsto 1$ for prime $\fq \mid \fP_0'$.
\end{itemize}
For prime $\fq \mid \fP'$, write $\epsilon_\fq = \varphi(U_\fq - \bpsi(\fq)) \in W$.  Then:
\begin{itemize}
\item $\epsilon_\fq(\epsilon_\fq -1 + \bpsi(\fq)) = 0$.
\item $\Ann_{A_{\sL}^\#}\Big(\prod_{\fq \mid \fP'} \epsilon_\fq\Big) \subset (p^m)$.
\end{itemize}
\end{theorem}

\begin{proof} We recall from (\ref{e:jdef}) the definition
\[
J_0 = \left(I_A^2, \Theta_L^\# - \Theta_H^\# \lambda, \  p^m\Lambda\right). 
\]
To streamline the notation for all cases, write
\[ X = \begin{cases} 1 & \text{Cases 1/2(a)} \\
\tilde{x} & \text{Case 2(b),}
\end{cases} \qquad K = XJ_0, \]
and recall $\fP_0' = 1$ in Case 2(b).
Define
\[ 
\cC = \prod_{\fa \subset O_F} \Lambda/{K}, 
\] with the product indexed by the nonzero ideals $\fa \subset \cO_F$. There is an $A$-module map 
\[ 
c: S_k(\fn\fP', \Lambda, \bpsi) \longrightarrow \cC 
\] 
that associates to each cusp form $h$ its collection of normalized $q$-expansion coefficients $c(\fa, h)$ reduced modulo ${K}$. Note that if we let the operators $T_\fl$, $U_\fl$ act by the usual formulae on Fourier coefficients (see~\cite{dk}*{eqn. (97)}) and we let $S(\fa)$ act by $\bpsi(\fa)$, then the map $c$ is equivariant for these operators.

Let $\cF$ denote the image of the $\T^{\dagger}$-span of the cusp form $g$ defined in Corollary~\ref{c:gdef} under the map $c$. This is an $A$-module of finite type. Define $W$ be the image of the canonical $A$-algebra homomorphism 
\[ 
\T^{\dagger} \longrightarrow \End_{A}(\cF). 
\] 
This construction yields a canonical surjective $A$-algebra map 
\[ 
\varphi: \T^{\dagger} \longrightarrow W 
\] 
that sends a Hecke operator to its action on the Hecke span of $g$ under the map $c$.  

We will show in a moment that $W$ has the structure of an $A_{\sL}^\#$-algebra.  First we calculate the action of Hecke operators on $g$ using the congruence of Corollary~\ref{c:gdef}:
\begin{equation} \label{e:gcong}
 g \equiv  \begin{cases}
  W_1( \bpsi_{\fP_0'},1) - \lambda W_1(\overline{\bpsi}_{\fP_0'},1) & \text{Cases 1/2(a)} \\ 
 X \left(W_1( \bpsi,1) - \lambda W_1(\overline{\bpsi},1)\right) + \lambda W_k(\overline{\bpsi}, 1_p)
 & \text{Case 2(b)}
\end{cases} \pmod{K}.
\end{equation}

The fact that $\varphi(S(\fa)) = \bpsi(\fa)$ is clear since all our forms have nebentypus $\bpsi$.  
The operator $T_{\fl}$ acts by multiplication by $1 + \bpsi(\fl)$ on $W_1(\bpsi, 1)$ and $W_1(\overline{\bpsi}, 1)$, and it acts by multiplication by $\N\fl^{k-1} + \bpsi(\fl)$ on $W_{k-1}(\bpsi, 1_p)$.  Since $\N\fl^{k-1} \equiv 1 \pmod{p^m}$ and $Xp^m \in K$, we have
\[ T_\fl(g) \equiv (\N\fl^{k-1} + \bpsi(\fl)) g \pmod{K}. \]

 Similarly $\varphi(U_\fq) = 1$ for prime $\fq \in \Sigma \setminus S_\infty$, since all the forms $W_1(\bpsi, 1)$, $W_1(\overline{\bpsi}, 1), W_k(\overline{\bpsi}, 1_p)$ have $U_\fq$-eigenvalue 1.

The most interesting action is that of $U_\fp$.  This operator fixes $W_1(\bpsi, 1)$, since $\fp \mid \fP$.  It also acts as $\overline{\bpsi}(\fp) = 1$ on $W_k(\overline{\bpsi}, 1_p)$ because of the $p$-stabilization.  Yet $U_\fp$ does not act as a scalar on $W_1(\overline{\bpsi}, 1)$, as it is not stabilized.  Instead $U_\fp - 1$ sends $W_1(\overline{\bpsi}, 1)$ to its $\fp$-stabilization $W_1(\overline{\bpsi}, 1_\fp) = W_1(\overline{\bpsi}_\fp, 1)$, where the equality follows since $\overline{\bpsi}(\fp) = 1$ and we are in weight 1.  The conclusion of these considerations is that 
\begin{equation} \label{e:cc}
 (U_\fp - 1)g \equiv  
 - \lambda X W_1(\overline{\bpsi}_{\fp\fP_0'},1) 
 \pmod{K}. \end{equation}

We can now give $W$ the structure of an $A_\sL^\#$-algebra by letting $\sL$ act by $\varphi(1 - U_\fp)$.  To prove that this is well-defined, we must show that the relations in $A_{\sL}^\#$ are satisfied in $W$, namely $\sL(\Theta^H)^\#- (\Theta^L)^\# = \sL I_A = \sL^2 = I_A^2 = 0$.  The most interesting of these is the first.  From (\ref{e:cc}), we find the following congruences mod 
$K$:
\begin{align*}
 (\Theta^H)^\# (1 - U_\fp)g & \equiv  (\Theta^H)^\# X \lambda W_1(\overline{\bpsi}_{\fp\fP_0'},1)   \\ 
 & \equiv  (\Theta^L)^\# XW_1(\overline{\bpsi}_{\fp  \fP_0'},1)  \\
 & \equiv  (\Theta^L)^\# g.
 \end{align*}
The second congruence follows from the definition of $K$ and the third from (\ref{e:gcong}), in view of the fact that $(\Theta_L)^\# \in I_A$ 
annihilates $\lambda$ while $W_1(\bpsi_{\fP_0'}, 1) \equiv W_1(\overline{\bpsi}_{\fp\fP_0'}, 1) \pmod{I_A}$ and 
$X (\Theta^L)^\# I_A \in X I_A^2 \subset K$.  This proves that the first relation holds, while the others are similar but easier. 

 Next we study the action of $U_\fq - \bpsi(\fq)$ for $\fq \mid \fP'$.  This operator sends $W_1(\bpsi, 1)$ to $W_1(\bpsi_\fq, 1)$, sends $W_1(\overline{\bpsi}, 1)$ to $W_1(\overline{\bpsi}_\fq, 1)$, and annihilates
$W_k(\bpsi, 1_p)$.  For $\fq \mid \fP_1'$, we therefore have
\begin{equation}
\label{e:uq}
 (U_\fq - \bpsi(\fq))g \equiv X\left(W_1(\bpsi_{\fq \fP_0'}, 1) - \lambda W_1(\overline{\bpsi}_{\fq \fP_0'}, 1)\right)
\pmod{K}. \end{equation}
Note that $U_\fq$ acts as 1 on the form on the right.  This yields the equation 
\begin{equation} \label{e:eqeq}
\epsilon_\fq(\epsilon_\fq -1 + \bpsi(\fq)) = 0,
\end{equation}
where $\epsilon_\fq = \varphi(U_\fq - \bpsi(\fq))$.  Meanwhile if $\fq \mid \fP_0'$, then in particular we are in case 1/2(a) and $U_\fq$ acts as 1 on $g \pmod{K}$ by (\ref{e:gcong}).  For such $\fq$ we have 
\begin{equation} \label{e:eqeq2}
\epsilon_\fq = 1 - \bpsi(\fq)
\end{equation} and (\ref{e:eqeq}) holds trivially.

To prove the last item note that by (\ref{e:uq}) and (\ref{e:eqeq2}) we have
\begin{equation} \label{e:alleqs}
\prod_{\fq \mid \fP'} (U_{\fq} - \bpsi(\fq)) g \equiv  X \left( W_1(\bpsi_{\fP'}, 1) - \lambda W_1(\overline{\bpsi}_{\fP'},1) \right) \prod_{\fq \mid \fP_0'} (1 - \bpsi(\fq)) \pmod{K}.
\end{equation}
In case 2(c), this congruence does not hold as there is a contribution from the term $W_k(\overline{\bpsi}, 1_p)$.  This term does not appear in (\ref{e:gcong}) in cases 1/2(a), and is annihilated by the single application of an $\epsilon_\fq$ for $\fq \mid \fP'_1$ in case 2(b).  This is why case 2(c) must be removed from the present analysis.

Returning to (\ref{e:alleqs}), suppose this form is annihilated by an element $a + b \sL \in A_{\sL}^{\#}$, with $a, b \in A$. By definition, such an element acts by $a + b (1- U_{\fp})$ and hence, noting that $\prod_{\fq \mid \fP_0'} (1 - \bpsi(\fq))$ is a unit, we obtain
\[
a X \left( W_1(\bpsi_{\fP'},1) - \lambda W(\overline{\bpsi}_{\fP'}, 1)\right) + b X \lambda W_1(\overline{\bpsi}_{\fP'}, 1_\fp) \equiv 0 \pmod{K}.
\]

Analyzing the $q$-expansion coefficients $c(1, *)$ and $c(\fp, *)$ of this congruence, respectively, we obtain 
\begin{align*}
X(a - \lambda a + \lambda b)  & \equiv 0 \pmod{K} \\
X(a - 2 \lambda a + \lambda b) & \equiv 0 \pmod{K}.
\end{align*}
From these, we deduce $X(a+\lambda b) \in K$, and since $X$ is a non-zerodivisor it follows that $a + \lambda b \in J_0 = ( \Theta_H^{\#} \lambda - \Theta_L^{\#}, I_A^2 \Lambda, p^m \Lambda)$. It follows that in $A_{\sL}^{\#}$, we have
\[
a+ b \sL \in (\Theta_H^{\#} \sL - \Theta_L^{\#}, I_A^2, p^m) = (p^m)
\]
as desired.
\end{proof}

\subsection{Galois Representation}  \label{s:galoisrep}

In this section we recall the formalism of \cite{dk}*{\S9.1--9.2} regarding the Galois representation associated to the Hecke algebra $\T^\dagger$.
As explained in \cite{dk}*{\S8.5} the Hecke algebra ${\T}^{\dagger}$ is reduced. 
The kernel of $\varphi$ is contained in a unique maximal ideal $\fm \subset \T^\dagger$.  This maximal ideal is generated by the maximal ideal $(\fm_{\cO}, I)$ of $A$ together with the elements $T_\fl - (1+ \chi(\fl))$ for $\fl \nmid \fn\fP'$, $S(\fa) - \chi(\fa)$ for $(\fa, \fn\fP') = 1$, $U_\fq - 1$ for all $\fq \mid p$.

Let $\T_{\fm}^\dagger$ denote the completion of $\T^\dagger$ with respect to $\fm$ (and similarly $\T_{\fm}$, $\tilde{\T}_\fm$  the completions of $\T$ and $\tilde{\T}$ with respect to their maximal ideals $\fm \cap \T$, $\fm \cap \tilde{\T}$, respectively).  Set
\[
K = \Frac(\T_{\fm}^{\dagger}).
\]
As in \cite{dk}*{\S9.2}, there is a Galois representation
\[
\rho: G_F \longrightarrow \GL_2(K)
\]
such that 
\begin{itemize}
\item[(1)] $\rho$ is unramified outside $\fn p$. 
\item[(2)] For all primes $\fl \nmid \fn p$, the characteristic polynomial of $\rho(\Frob_{\fl})$ is given by 
\begin{equation} \label{e:chartl}
{\rm char}(\rho(\Frob_{\fl}))(x) = x^2 - T_{\fl} x + \bpsi(\fl) \N\fl^{k-1}.
\end{equation}
\item[(3)] For  $\fq \mid p$, let $G_{F,\fq} \subset G_F$ denote a decomposition group at $\fq$.  We have
\begin{equation} \label{e:repord}
\rho|_{G_{F,\fq}} \sim \begin{pmatrix} \bpsi \varepsilon_{\cyc}^{k-1}\eta_\fq^{-1} & * \\ 0 & \eta_{\fq} \end{pmatrix},
\end{equation}
where $\eta_{\fp}\colon G_{F,\fq} \longrightarrow (\T_\fm^\dagger)^*$ is the unramified character given by $\eta_{\fp}(\rec(\varpi^{-1})) = U_{\fp}$.  Here $\varpi$ denotes a uniformizer of $F_\fq^*$ and \[ \rec\colon F_\fq^* \longrightarrow G_{F,\fq}^{\ab} \] is the local Artin reciprocity map.  
\end{itemize}

For each $\fq \mid p$, let $V_{\fq}$ be the eigenspace of $\rho|_{G_{\fq}}$ i.e. the span of the vector $\binom{1}{0}$ in the basis for which (\ref{e:repord}) holds. 
We choose an element $\tau \in G_F$ as in \cite{dk}*{Proposition 9.3} so that its restriction to $\fg$ is complex conjugation and so that for all $\fq \mid p$, the subspace $V_{\fq}$ projected to each factor of $K$ is not stable under $\rho(\tau)$. As in \emph{loc.\,cit.,} fix a basis such that 
\begin{equation} \label{e:tauchoice}
\rho(\tau) = 
\begin{pmatrix} \lambda_1 & 0 \\ 0 & \lambda_2\end{pmatrix}
\end{equation}
where $\lambda_1 \equiv 1 \pmod{\fm}$ and $\lambda_2 \equiv -1 \pmod{\fm}$. For a general $\sigma$, we write 
\[
\rho(\sigma) = \begin{pmatrix} a(\sigma) & b(\sigma) \\ c(\sigma) & d(\sigma) \end{pmatrix}.
\]
For each $\fq \mid p$, there is a change of basis matrix $M_\fq = \begin{pmatrix} A_\fq & B_\fq \\ C_\fq & D_\fq \end{pmatrix} \in \GL_2(K)$ such that 
\begin{equation} \label{e:matrixeqn}
\begin{pmatrix} a(\sigma) & b(\sigma) \\ c(\sigma) & d(\sigma)\end{pmatrix} M_{\fq} = M_{\fq} \begin{pmatrix} \bpsi \varepsilon_{\cyc}^{k-1} \eta_{\fq}^{-1} & * \\ 0 & \eta_{\fq} \end{pmatrix}.
\end{equation}
The choice of $\tau$ ensures that $A_\fq$ and $C_\fq$ are invertible in $K$. Furthermore, equating the upper left hand entries in (\ref{e:matrixeqn}) yields: 
\begin{equation} \label{e:bequation}
b(\sigma) = x_\fq (a(\sigma) -  \bpsi \varepsilon_{\cyc}^{k-1} \eta_{\fq}^{-1}(\sigma))  \text{ for all }\sigma \in G_{F,\fq}, \text{ where } 
x_\fq = -\frac{A_\fq}{C_\fq}.
\end{equation}

\subsection{Cohomology Class} \label{s:cc}

Let $\varphi_{\fm} \colon \T_\fm^\dagger \longrightarrow W$ denote the extension of the homomorphism $\varphi$ to the completion of $\T^\dagger$.
Let \[ \I^\dagger = \ker(\varphi_{\fm}), \quad \tilde{\I} = \ker({\varphi}_{\fm}|_{\tilde{\T}_{\fm}}), \quad \I = \ker(\varphi_{\fm}|_{\T_\fm}). \]
 As in \cite{dk}*{\S9.3}, the choice of basis for $\rho$ implies that $a(\sigma), d(\sigma) \in \T$ with 
\begin{equation} \label{e:adcong}
a(\sigma) \equiv \varepsilon_{\cyc}^{k-1}(\sigma) \pmod{\I}, \qquad d(\sigma) \equiv \bpsi(\sigma) \pmod{\I}.
\end{equation}
Furthermore we have 
\begin{equation} \label{e:detcong}
\det \rho(\sigma) \equiv \varepsilon_{\cyc}^{k-1}\bpsi(\sigma).
\end{equation}

Recall that  $K=\Frac(\T_{\fm}^\dagger)$.
Recall the elements $x_\fq \in K$ for $\fq \mid p$ defined in (\ref{e:bequation}), including the distinguished prime $\fp$.  
Define $x_{\fp}' = x_{\fp}(U_\fp - 1)$.  Let:
\begin{align} 
\tilde{\B} & = \tilde{\T}_{\fm} \text{-submodule of } K \text{ generated by } b(\sigma)  \text{ for all } \sigma \in G_F \text{ along with } \nonumber \\
& \ \ \ \ \  x_{\fp}, x_{\fp}', \text{ and } x_\fq \text{ for all } \fq \in \Sigma. \nonumber \\
\tilde{\B}' & = \tilde{\T}_{\fm} \text{-submodule of } \tilde{\B} \text{ generated by } \tilde{\I} \tilde{\B}, p^m \tilde{\B}, \text{ and}  \nonumber \\
& \ \ \ \ \  b(\sigma) \text { for all } \sigma \in I_{\fq}, \fq \mid \fP'. \nonumber \\
\tilde{B} & = \tilde{\B}/\tilde{\B}'. \label{e:tbdef}
 \end{align}
 By construction, $\tilde{B}$ naturally has the structure of an  $A_{\sL}^\#$-algebra in which $\sL$ acts by multiplication by $(1 - U_\fp)$.
 
Denote by $\overline{b}(\sigma)$ the image of $b(\sigma)$ in $\tilde{B}$.
Since $\rho$ is a Galois representation, we have
\[ b(\sigma \sigma') = a(\sigma) b(\sigma') + b(\sigma)d(\sigma') \qquad \text{ for all } \sigma, \sigma' \in G_F. \]
The congruences (\ref{e:adcong}) therefore imply that the function 
\[ \kappa(\sigma) =  \overline{b}(\sigma)\bpsi^{-1}(\sigma) \]
is a 1-cocycle defining a cohomology class $[\kappa] \in H^1(G_F, \tilde{B}(\bpsi^{-1}))$.  Note here that since $p^m \tilde{\B} = 0$ in $\tilde{B}$ and $\varepsilon_{\cyc}^{k-1} \equiv 1 \pmod{p^m}$, the character $\varepsilon_{\cyc}^{k-1}$ acts trivially on $\tilde{B}$.
Let: \begin{align*}
{B}_0 &=  \T_\fm\text{-submodule of } \tilde{B} \text{ generated by the image of } b(\sigma) \text{ for all } \sigma \in G_F.  \\
B & =  \T_\fm\text{-submodule of } \tilde{B} \text{ generated by } B_0 \text{ along with }  \\
& \ \ \ \ \  x_{\fp}, x_{\fp}', \text{ and } x_\fq \text{ for all } \fq \in \Sigma. 
\end{align*}
Note that $\tilde{\I}\tilde{\B} = 0$ in $\tilde{B}$ and every element of $\T_\fm$ is equivalent to an element of $A$ modulo $\I$ (see the first three bullet points of Theorem~\ref{t:heckehom}). Therefore, in the definition of $B_0$ and $B$, it is equivalent to replace $\T_\fm$ by $A$.

The $A$-module structure of $\tilde{B}(\bpsi^{-1})$ is by definition the composition of the involution $\#$ with the natural $A$-module structure of $\tilde{B}$.  The canonical $A_{\sL}^\#$-module structure of $\tilde{B}$ can therefore be viewed as an $A_{\sL}$-module structure on $\tilde{B}(\bpsi^{-1})$.

We now verify that our construction verifies all the properties required to apply Theorems~\ref{t:nc} and~\ref{t:nl}.
First we describe the class $\kappa$ locally at primes in $\Sigma_\fp \setminus S_\infty$ using (\ref{e:bequation}).

\begin{lemma} \label{l:kappaloc} For $\fq \in \Sigma \setminus S_\infty$, we have
\[ \kappa(\sigma) = x_{\fq}( \bpsi^{-1}(\sigma) -1 ), \qquad \sigma \in G_{F, \fq}. \] 
Meanwhile if $x_{\fp}' = -\sL x_{\fp}$ then
\[ \kappa(\sigma) = x_{\fp}(\bpsi^{-1}(\sigma)-1) + \ord_{\fp}(\sigma)x_{\fp}', \qquad \sigma \in G_{F, \fp},\]
where $\ord_\fp$ is as in (\ref{e:ordp}). 
\end{lemma}

\begin{proof}
Note  that \begin{equation} \label{e:aeq}
a(\sigma) \equiv \varepsilon_{\cyc}^{k-1}(\sigma) \equiv 1 \pmod{(\I, p^m)}. \end{equation}
If $\fq \in \Sigma \setminus S_\infty,$ then $U_\fq \equiv 1 \pmod{\I}$, and hence by the definition of  $\eta_\fq$ following (\ref{e:repord}), 
$\eta_\fq$ acts trivially on $\tilde{B}(\bpsi^{-1})$.  Therefore (\ref{e:bequation}) becomes
\[ \kappa(\sigma) = \overline{b}(\sigma) \bpsi^{-1}(\sigma)= x_{\fq}( \bpsi^{-1}(\sigma) -1 ), \qquad \sigma \in G_{F, \fq}. \] 

For $\fq = \fp$, this applies except that $U_\fp \not\equiv 1 \pmod{\I}$.
Since $(U_\fp-1)^2 \in \I$, we have \[ U_\fp^n \equiv 1 + n(U_\fp - 1) \pmod{\I}. \]  It follows that
\[ \kappa(\sigma) = x_{\fp}(\bpsi^{-1}(\sigma)-1) + \ord_{\fp}(\sigma)x_{\fp}', \qquad \sigma \in G_{F, \fp},\]
where \[ x_{\fp}' = ( U_\fp-1) x_{\fp} = -\sL x_{\fp}.  \qedhere \] 
\end{proof}

\begin{theorem}  \label{t:properties}
The $G$-module $\tilde{B}(\bpsi^{-1})$ and the cohomology class $[\kappa] \in H^1(G_F, \tilde{B}(\bpsi^{-1}))$ satisfy the following properties.
\begin{itemize}
\item The class $[\kappa]$ is unramified outside $\Sigma'$, locally trivial at  $\Sigma$, and tamely ramified at $\Sigma'$.
\item The image of the restriction 
\[ [\kappa]_{|G_L} \in H^1(G_L, \tilde{B}) = \Hom_{\cont}(G_L, \tilde{B}) \]
 is equal to ${B}_0$.
\item The quotient 
${B}/{B}_0$ is generated over $A$ by $x_\fq$ for $\fq \in \Sigma$ and the elements $x_{\fp}, x_{\fp}'$.  The element $x_{\fp}'$ is fixed by the action of $G_{F,\fp}$.
\item The 1-cocycle $\kappa$ satisfies the following.
\begin{itemize}
\item For $\fq \in \Sigma$ finite and $\sigma \in G_{F, \fq}$, we have  
$\kappa(\sigma) = (\bpsi^{-1}(\sigma)-1)x_\fq.$
\item For $\sigma \in \WD(\overline{F}_\fp/F_\fp)$, we have \begin{equation} \label{e:ksp}
 \kappa(\sigma) =     (\bpsi^{-1}(\sigma)-1)    x_{\fp} + \ord_\fp(\sigma) x_{\fp}'. \end{equation}
 \item  Let $\tau$ be the special element used in \S\ref{s:galoisrep} to fix the chosen basis of $\rho$.  For all $\sigma \in G_F$, we have
\begin{equation} \label{e:nicecocycleverify}
\kappa(\sigma)=  [\kappa]|_{G_L}({\sigma}\tau {\sigma}^{-1} \tau^{-1})/2 \in B_0.
\end{equation}
\end{itemize}
\item With respect to the $A_{\sL}$-structure on $\tilde{B}(\bpsi^{-1})$, we have  \[ x_{\fp}' + \sL x_\fp = 0. \]
\end{itemize}
\end{theorem}

\begin{proof}
The Galois representation $\rho$, and hence the cohomology class $[\kappa]$, is unramified outside $\Sigma \cup \Sigma'$ and the primes dividing $p$.  There are three types of primes above $p$: those in $\Sigma$, the distinguished prime $\fp$, and the primes dividing $\fP'$.  

Since in the definition of $\tilde{B}$ we have taken the quotient by the $\tilde{\T}_{\fm}$-module spanned by $b(\sigma)$  for $\sigma \in I_\fq, \fq \mid \fP'$, it follows that $[\kappa]$ is unramified at the primes dividing $\fP'$. 
 For $\fq \in \Sigma \setminus S_\infty$, Lemma~\ref{l:kappaloc}
  expresses $\kappa_{| G_{F, \fq}}$ as a coboundary and demonstrates that $[\kappa]$ is locally trivial at $\fq \in \Sigma$.  At $\fp$, note that since $ \ord_{\fp}(\sigma) = 0 $  for $\sigma \in I_{\fp}$, Lemma~\ref{l:kappaloc} shows that $[\kappa]$ is unramified at $\fp$.  To conclude the proof of the first item, note that $\Sigma'$ contains no primes above $p$, and all our modules are pro-$p$.  Therefore $[\kappa]$ is tamely ramified at all primes in $\Sigma'$.

For the second bullet point, let $B_L \subset B_0$ denote the image of $[\kappa]_{| G_L}$.  Then of course the image of $[\kappa]_{| G_L}$ in $H^1(G_L, (B_0/B_L)(\bpsi^{-1}))$ vanishes, and hence by \cite{dk}*{Lemma 6.3}, the image of $[\kappa]$ in $H^1(G_F, (B_0/B_L)(\bpsi^{-1}))$ vanishes.  Writing $\overline{\kappa}$ for the image of $\kappa$ in $B_0/B_L$, we may then write \[ \overline{\kappa}(\sigma) = z(\bpsi^{-1}(\sigma) - 1) \] for some $z \in B_0/B_L$.  Yet by construction $\kappa(\tau) = 0$ and $\bpsi^{-1}(\tau) = -1$. We therefore obtain $z=0$.  Hence $\overline{\kappa} = 0$ as a cocycle.  But by definition of $\kappa$ and $B_0$, the image of the cocycle $\kappa$ generates $B_0$ and hence the image of $\overline{\kappa}$ generates $B_0/B_L$.  It follows that $B_0/B_L = 0$, i.e.\ $B_0 = B_L$ as desired.

 Next we show that $\kappa$ satisfies equation (\ref{e:nicecocycleverify}). As $B_0$ is pro-$p$ with $p$ odd, it is enough to show that 
\[
2 \bpsi^{-1}(\sigma) \overline{b}(\sigma) = \overline{b}(\sigma \tau \sigma^{-1} \tau^{-1}),
\]
where $\tau$ is  as in  (\ref{e:tauchoice}). The upper right entry (``$b$-entry'') of $\rho(\sigma \tau \sigma^{-1} \tau^{-1})$ is 
\begin{equation} \label{e:bexp}
b(\sigma \tau \sigma^{-1} \tau^{-1}) = \det(\rho(\sigma))^{-1} a(\sigma) b(\sigma) \left(1 + \frac{\lambda_2}{\lambda_1} \right).
\end{equation}
 The congruence (\ref{e:detcong}) yields
\[
\delta \equiv \bpsi(\sigma) \pmod{p^m, \I}. 
\]
The first congruence in (\ref{e:adcong}) yields
\[
a(\sigma) \equiv 1 \pmod{p^m, \I}.
\] 
Furthermore, $\lambda_1 \equiv 1 \equiv -\lambda_2 \pmod{p^m, \I}$. Hence the expression on the right  side of  (\ref{e:bexp}) is congruent to $2 \bpsi^{-1}(\sigma) \overline{b}(\sigma)$. This finishes the proof that $\kappa$ satisfies equation (\ref{e:nicecocycleverify}).

The remaining bullet points follow directly from the definitions or have already been established.
\end{proof}

In view of Theorems~\ref{t:nc} and~\ref{t:nl}, we deduce from Theorem~\ref{t:properties}:
\begin{corollary}  \label{c:fit}
We have an $A_{\sL}$-module surjection
 \[ \nabla_{\!\sL, A} \longtwoheadrightarrow \tilde{B}(\bpsi^{-1}) \] and hence an inclusion
 $\Fitt_{A_{\sL}}(\nabla_{\!\sL, A}) \subset \Fitt_{A_{\sL}}(\tilde{B}(\bpsi^{-1}))$.
\end{corollary}

\subsection{Calculation of the Fitting ideal}

It remains to prove that \[  \Fitt_{A_{\sL}}(\tilde{B}(\bpsi^{-1})) \subset (p^m). \]
Applying the involution $\#$, this is equivalent to 
\[  
\Fitt_{A_{\sL}^\#}(\tilde{B}) \subset (p^m) = p^m A_{\sL}^\#. 
\]  
This removal of the twist by $\bpsi^{-1}$ will be convenient so that the usual $\fg$-module structure on $\tilde{B}$ via $\bpsi$ is compatible with the $\tilde{\T}_\fm$-module structure on $\tilde{B}$ and the homomorphism $\varphi_\fm$, which satsifies $\varphi_\fm(S(\fa)) = \bpsi(\fa)$.

  We first recall the following lemma from  \cite{dk}*{Lemma 9.9}.
\begin{lemma} \label{l:bgen} Recall that $\B_0 \subset \tilde{\B}$ denotes the $\tilde{\T}_{\fm}$-submodule generated by the elements $b(\sigma)$ for $\sigma \in G_F$.
 There are finitely many elements $b_1, \ldots, b_n \in {\B}_0$ that are non-zerodivisors in $K$, which  generate $\B_0$ as an $A$-module.
 \end{lemma}
 
 \begin{theorem} \label{t:mainfit2}  With notation as in Theorem \ref{t:heckehom}, we have \[ \Fitt_{A_{\sL}^\#}(\tilde{B}) \subset
  \bigg(p^m, \Ann_{A_{\sL}^\#}\Big( \prod_{\fq \mid \fP'} \epsilon_\fq \cdot W\Big) \bigg), \]
 and hence by the last item of that theorem we have
 \[  \Fitt_{A_{\sL}^\#}(\tilde{B}) \subset (p^m). 
 \]
\end{theorem}

\begin{proof} 
The proof proceeds closely along the lines of that in \cite{dk}*{Theorem 9.10}. Let  $\fq_1, \dotsc, \fq_r$ denote the primes dividing $\fP'$. For each $\fq_i$, choose an element $\sigma_i \in G_{\fq_i} \subset G_F$ that lifts $\rec(\varpi_i) \in G_{\fq_i}^{\ab}$, where
$\varpi_i$ is a uniformizer for $F_{\fq_i}$.  Define
\[ c_i := -b(\sigma_i)\bpsi(\fq_i) \varepsilon_{\cyc}^{1-k}(\sigma_i) = x_{\fq_i}(U_{\fq_i} - \bpsi(\fq_i) + \I) \in \tilde{\B}. \]
Here and throughout this proof, we use the notation $a = b + \I$ to mean $a = b + z$ for some $z \in \I$ to avoid needing to add distinct variable names for each such $z$ that appears.  

By choosing the elements $b_1, \dots, b_m$ from Lemma~\ref{l:bgen} together with $x_\fq$ for all $\fq \in \Sigma_\fp$, we get elements $b_1, \dotsc, b_n$ ($n = m + \#\Sigma_\fp$) of $\tilde{\B}$ that are non-zerodivisors in $K=\Frac(\T_{\fm})$ and generate this module over $\tilde{\T}_\fm$.  The images of these elements in $\tilde{B}$ are therefore $A_{\sL}^\#$-module generators.

To calculate $\Fitt_{A_{\sL}^\#}(\tilde{B}) $ we use the generating set $c_1, \dotsc, c_r, b_1, \dotsc, b_n$ for $\tilde{B}$ over $A_{\sL}^\#$.  Of course, these first $r$ generators are not necessary, but including them will aid us in proving the theorem.  Suppose we have a matrix 
\[ 
M \in M_{(n+r) \times (n+r)}(A_{\sL}^\#) 
\]
such that each row of $M$ represents a relation amongst our generators, i.e.\ such that 
\[ 
M(c_1, \dotsc, c_r, b_1, \dotsc, b_n)^T \equiv 0 \text{ in } \tilde{B}^{n+r}.
\]
We need to show that $\det(M) \in A_{\sL}^\#$ satisfies $(\det(M) + p^m z) \prod_{\fq \mid \fP'} \epsilon_\fq = 0$ in $W$ for some $z \in A_{\sL}^\#$.

Write $M   = (Y | Z)$ in block matrix form, where 
\[ 
Y = (y_{ij}) \in M_{(n+r)\times r}(A_{\sL}^\#), \qquad 
Z = (z_{ij}) \in M_{(n+r)\times n}(A_{\sL}^\#). 
\]
Since $\bpsi$ and $\eta_{\fq_i}$ are unramified at $\fq_i$ and $a(\sigma) \equiv \varepsilon_{\cyc}^{k-1} \pmod{\I}$, we have by (\ref{e:bequation}):
\[ b(I_{\fq_i}) \subset x_{\fq_i} \I. \]
Also, since the $b_i$ generate $\tilde{\B}$, every element of $\tilde{\I} \tilde{\B}$ can be written as a sum of elements of the form   $ t_i b_i$ with $t_i \in \tilde{\I}$.
Therefore each relation \[ \sum_{j=1}^r y_{ij} c_j + \sum_{j=1}^n z_{ij} b_j \equiv  0 \text{ in } \tilde{B} \] can be expressed as in equality in $\tilde{\B}$ as
\begin{equation} \label{e:xyzsum}
  \sum_{j=1}^r  x_{\fq_j}(\tilde{y}_{ij}(U_{\fq_j} - \bpsi(\fq_j)) + \tilde{\I}) +\sum_{j=1}^n (\tilde{z}_{ij} + \tilde{\I} + p^m \Tilde{\T}_\fm) b_j  = 0. \end{equation}
  We reiterate that here and in what follows, the symbols $\tilde{\I}$ (twice) and $\tilde{\T}_\fm$ represent elements of those sets for which we do not, for notational reasons, introduce separate variable names.  Here we have denoted by $\tilde{y}_{ij}$ and $\tilde{z}_{ij}$ elements of $\tilde{\T}_{\fm}$ such that $\varphi(\tilde{y}_{ij}) = y_{ij}$, $\varphi(\tilde{z}_{ij}) = z_{ij}$.
It  follows from (\ref{e:xyzsum}) that if we define a matrix $M' \in M_{(n+r) \times (n+r)}(\Frac(\tilde{\T}_{\fm}))$  in block form by
\[ 
M' = \left( x_{\fq_j}(\tilde{y}_{ij}( U_{\fq_j} - \bpsi(\fq_j) ) + \tilde{\I}) \ \ \ \ |  \ \ \ \ (\tilde{z}_{ij} + \tilde{\I} + p^m \tilde{\T}_\fm) b_j \right), 
\]
then $\det(M') = 0$ in $K$ since it has rows that sum to 0.  We can cancel the factors $x_{\fq_i}$ and $b_j$ scaling the columns of $M'$, since these are non-zerodivisors in $K$.  We obtain that $\det(M'') = 0$ where 
\[  
M'' = \left( (\tilde{y}_{ij}(U_{\fq_j} - \bpsi(\fq_j)) + \tilde{\I}) \ \ |  \ \ (\tilde{z}_{ij} +\tilde{\I} + p^m \tilde{\T}_\fm)  \right) \in M_{(n+r)\times (n+r)}(\T^\dagger).
\]
Taking the determinant of $M''$ and applying $\varphi$, we obtain
\[ 
0 = \varphi( \det(M''))= (\det(M) + p^m A_{\sL}^\#) \prod_{\fq \mid \fP'} \epsilon_\fq \text{ in } W. 
\]  
Therefore,  \[ \det(M) \in \bigg(p^m, \Ann_{A_{\sL}^\#}\Big( \prod_{\fq \mid \fP'} \epsilon_\fq \cdot W\Big) \bigg), \]
yielding the first statement of the theorem.
The second statement then follows immediately from  the last bulleted statement in Theorem~\ref{t:heckehom}.
\end{proof}

We immediately find:

\begin{theorem}  \label{t:fittal}  Suppose we are in cases 1, 2(a), or 2(b).
We have 
\[ \Fitt_{A_{\sL}}(\nabla_{\!\sL, A}) = 0.\]
\end{theorem}
\begin{proof}

Corollary~\ref{c:fit} and Theorem~\ref{t:mainfit2} combine to yield
\[  
\Fitt_{A_{\sL}}(\nabla_{\!\sL, A}) \subset \Fitt_{A_{\sL}}(\tilde{B}(\bpsi^{-1})) \subset (p^{m}). 
\]
Since this holds for all $m$, we have $\Fitt_{A_{\sL}}(\nabla_{\!\sL, A}) = 0$ as desired.
\end{proof}

Combining our results yields the $p$-part of Gross's conjecture.
\begin{proof}[Proof of Theorem~\ref{t:maingross}]  For now we assume there is more than one prime above $p$ in $F$, leaving the case of one prime for the next section.
Theorem~\ref{t:fittal} and Lemma~\ref{l:fittchi} imply that
$\Fitt_{R_{\sL}}(\nabla_{\!\sL}) = 0.$  Theorem~\ref{t:mfig} then yields
the $p$-part of the modified Gross conjecture:
\[ \rec_G(u_\fp^{\Sigma, \Sigma'}) \equiv \Theta_L \pmod{I^2}. \]
Lemma~\ref{l:equiv} now gives the $p$-part of Gross's conjecture for the Brumer--Stark unit $u_\fp$:
\[ \rec_G(u_\fp) \equiv  \Theta_{S_\fp,T}^{L/F} \pmod{I^2}. \qedhere \]
\end{proof}

\section{The case of one prime above $p$ in $F$} \label{s:onep}

In this section we handle Case 2(c), where $\fp$ is the only prime of $F$ lying above $p$.  We impose this condition for the remainder of the paper.  
Rather than calculating the Fitting ideal of $\nabla_{\!\sL}$, we prove the $p$-part of the modified Gross Conjecture (the congruence (\ref{e:mgc})) by taking advantage of two features that present themselves when there is only one prime above $\fp$: (1) the cyclotomic tower is ramified only at $\fp$, hence we may deform up this tower without altering the depletion set $\Sigma_\fp$; (2) The rank one rational Gross--Stark conjecture, proven in \cite{ddp} and \cite{v}, is known.  In essence, our argument  is to show that the rational conjecture (for the cyclotomic tower) implies the integral conjecture (for arbitrary $L/F$).  The key input in this reduction is the strong version of the Brumer--Stark conjecture giving the Fitting ideal of $\nabla_{\Sigma_\fp}^{\Sigma'}$, conjectured by Burns and Kurihara and proven in \cite{dk}.  This result was stated in Theorem~\ref{t:dkresult} above, and we apply it in this section to the compositum of $L$ with cyclotomic extensions of $F$.  Another important result applied is the nonvanishing of the derivative of $p$-adic $L$-functions at $s=0$, which follows by combining the result of the rational Gross--Stark conjecture with the spectacular transcendence theorem of Brumer--Baker on the linear independence of $p$-adic logarithms of algebraic numbers.  This again takes advantage of the fact that we are in a rank 1 situation since there is only one prime above $p$ in $F$.

We begin by recalling the necessary results and notation concerning $p$-adic $L$-functions.

\subsection{$p$-adic $L$-functions}

In our current setting we have 
$\Sigma = S_\infty$, $\Sigma_\fp = S_\infty \cup \{\fp\}$, and $\fp$ is the only prime of $F$ above $p$.
For each odd character $\chi$ of $G$, Deligne--Ribet and Cassou-Nogues construct a $p$-adic meromorphic function
\[ L_p(\chi\omega, s) \colon \Z_p \longrightarrow \Q_p(\chi) \]
satisfying the interpolation property
\[ L_p(\chi\omega, n) = L_{\Sigma_\fp, \emptyset}(\chi \omega^n, n) \]
for all integers $n \le 0$. Here $\omega \colon G_F \longrightarrow \mu_{p-1}$ denotes the  Teichm\"uller character and $\Q_p(\chi)$ denotes the extension of $\Q_p$ obtained by adjoining the values of $\chi$.  Analyticity and integrality are achieved if we incorporate the smoothing set $\Sigma'$, i.e. we have a $p$-adic analytic function
\[ L_{p, \Sigma'}(\chi\omega, s) \colon \Z_p \longrightarrow \Z_p(\chi) \] satisfying
\[ L_{p, \Sigma'}(\chi\omega, n) = L_{\Sigma_\fp, \Sigma'}(\chi \omega^n, n) \]
for all integers $n \le 0$. Moreover, these $p$-adic $L$-functions interpolate to a group ring valued Stickelberger function.
There exists a $p$-adic analytic function 
\[ \Theta_{p, \Sigma'}^H(s) \colon \Z_p \longrightarrow \Z_p[G]^- \] satisfying the interpolation property
\[ \Theta_{p,\Sigma'}^H(1-k) = \Theta_{\Sigma_\fp, \Sigma'}(1-k) 
 \]
for all positive integers $k \equiv 1\pmod{p-1}$.  As we now describe, $\Theta_{p, \Sigma'}^H(s)$ can be constructed as a certain $p$-adic integral.
Let $F_\infty/F$ denote the cyclotomic $\Z_p$-extension of $F$ and let $\fh_\infty = \Gal(HF_\infty/F)$.  
For each integer $m \ge 0$ we let $F_m \subset F_\infty$ denote the $m$th layer of the tower and let $\fh_m = \Gal(HF_\infty/HF_m)$.
Then $\fh_m$ is an open subgroup of $\fh_\infty$ and its cosets provide a cover of $\fh_\infty$ by disjoint opens.
For $\sigma \in G_m = \fh_\infty/\fh_m \cong \Gal(HF_m/F)$ we define
\[ \mu(\sigma + \fh_m) = \zeta_{\Sigma_\fp, \Sigma'}(\sigma, 0), \quad \text{where } 
 \Theta_{\Sigma_\fp, \Sigma'}^{HF_m/F}(0) = \sum_{\tau \in G_m} \zeta_{\Sigma_\fp, \Sigma'}(\tau, 0)\tau^{-1}. \]
For $\sigma \in G$ define
\begin{equation} \label{e:zetaint}
 \zeta_{p, \Sigma'}(\sigma, s) = \int_{\sigma + \fh_0} \langle \varepsilon_{\cyc}(\tau) \rangle^{-s} d \mu(\tau). 
 \end{equation}
Here $\fh_0 = \Gal(HF_\infty/H)$.  We then have 
\[ \Theta_{p, \Sigma'}^H(s) = \sum_{\sigma \in G} \zeta_{p, \Sigma'}(\sigma, s) \sigma^{-1}. \] 
 Taking the derivative of (\ref{e:zetaint}) with respect to $s$ and evaluating at 0, we obtain
\begin{equation} 
 \zeta_{p, \Sigma'}'(\sigma, 0) = -\int_{\sigma + \fh_0} \log_p \varepsilon_{\cyc}(\tau) d \mu(\tau). 
 \end{equation}
To evaluate this modulo $p^m$, we may take the Riemann sum over the cosets of $\fh_m$.  We obtain:
\begin{lemma}  \label{l:derivative}
For every integer $m \ge 0$, we have \[
(\Theta_{p, \Sigma'}^H)'(0) \equiv - \sum_{\sigma \in G_m}  \zeta_{\Sigma_\fp, \Sigma'}(\sigma, 0) \log_p\varepsilon_{\cyc}(\sigma)\overline{\sigma}^{-1} \pmod{p^m}. \]
Here $\overline{\sigma}$ denotes the image of $\sigma$ in $G$.
\end{lemma}

For notational simplicity, we will simply write $\Theta_H'$ for $(\Theta_{p, \Sigma'}^H)'(0)$ in the sequel.

\subsection{The rational Gross--Stark conjecture}

The following result, which we refer to as the rank 1 rational Gross--Stark conjecture, was proven in \cite{ddp} and \cite{v}.  The latter paper removed two assumptions from the former, making the result unconditional.

\begin{theorem}[\cite{gross}*{Conjecture 2.12}] \label{t:ddp}
Let $u \in U_\fp^-$ such that 
\begin{equation} \label{e:ordnz}
 \chi(\ord_G(u)) = \sum_{\sigma \in G}  \chi(\sigma)^{-1} \ord_{\fP}(\sigma(u_\fp)) \neq 0.
\end{equation}
  Then
\begin{equation} \label{e:lddp}
 \frac{\sum_{\sigma \in G} \chi(\sigma)^{-1} \log_p(\Norm_{F_{\fp}/\Q_p}(\sigma(u)))}{ \chi(\ord_G(u)) } = - \frac{L_p'(\chi^{-1}\omega,0)}{L(\chi^{-1},0)}. \end{equation}
\end{theorem}

When applied to the Brumer--Stark unit $u_{\fp}^{\Sigma,\Sigma'}$, this result can be interpreted as follows:

\begin{corollary} \label{c:ddp}
We have 
\begin{equation} \label{e:ddp}
\Theta_H' = - \sum_{\sigma \in G} \overline{\bpsi}(\sigma)^{-1}  \log_p(\Norm_{F_{\fp}/\Q_p}(\sigma(u_{\fp}^{\Sigma, \Sigma'}))).
\end{equation}
\end{corollary}
\begin{proof} It is enough to show that the two sides agree after the application of $\chi$ for every odd character $\chi$ of $G$, i.e.\ that
\begin{equation} \label{e:lddp2}
 L_{p, \Sigma'}'(\chi^{-1} \omega, 0) = 
-\sum_{\sigma \in G} \chi(\sigma)^{-1} \log_p(\Norm_{F_{\fp}/\Q_p}(\sigma(u_{\fp}^{\Sigma, \Sigma'}))). 
\end{equation}
Noting that
\[ \frac{L_p'(\chi^{-1}\omega,0)}{L(\chi^{-1},0)} = \frac{L_{p, \Sigma'}'(\chi^{-1}\omega,0)}{L_{\Sigma,\Sigma'}(\chi^{-1},0)} 
\]
since the smoothing factors $\prod_{\fq \mid \Sigma'} (1 - \chi^{-1}(\sigma_\fq) \N\fq)$ cancel, and recalling that 
\[ \chi(\ord_G(u_{\fp}^{\Sigma,\Sigma'})) = L_{\Sigma,\Sigma'}(\chi^{-1},0) \] by the definition of $u_{\fp}^{\Sigma,\Sigma'},$ the desired result (\ref{e:lddp2}) follows directly from (\ref{e:lddp}).
\end{proof}

The following result is known to the experts and has at its heart two deep facts---Theorem~\ref{t:ddp} above and the celebrated transcendence result of Brumer--Baker on the linear independence of logarithms of algebraic numbers over $\overline{\Q}$.
\begin{theorem} \label{t:nv}
Let $\chi$ be an odd character of $G$.  We have $L_p'(\chi\omega, 0) \neq 0$.
\end{theorem}

\begin{proof}
In view of Theorem~\ref{t:ddp}, it suffices to prove that 
\[ \sum_{\sigma \in G} \chi(\sigma)^{-1} \log_p(\Norm_{F_{\fp}/\Q_p}(\sigma(u))) \neq 0 \] for any $u \in U_\fp^-$ such that $\chi(\ord_G(u)) \neq 0$.  This follows from the theorem of Brumer--Baker \cite{bb} as explained by Gross in \cite{gross}*{Proposition 2.13}.
\end{proof}

We interpret this result in terms of the group-ring element $\Theta_H'$ as follows.

\begin{corollary}  The element $\Theta_H' \in \overline{R} = \Z_p[G]^-$ is a non-zerodivisor.
\end{corollary}

\subsection{The ring $R_{X,m}$ and module $\nabla_{X,m}$}

Recall $R = \Z_p[\fg]^-$. For any nonnegative integer $m$, define the ring 
\begin{equation} R_{X,m} = R[X]/(\Theta_L - \Theta_H' X, X^2, I X, I^2, p^mX). \label{e:rxdef} \end{equation}

\begin{lemma} For $m$ sufficiently large, the canonical map $R/I^2 \longrightarrow R_{X,m}$ is injective. \label{l:nxi}
\end{lemma}

\begin{proof}
The proof is easier than Theorem~\ref{t:ws} since $\Theta_H'$ is a non-zerodivisor in $\overline{R}$.
 Let $a \in R$ have image in $R_{X,m}$ that vanishes.  Writing down the fact that $a$ lies in the ideal defining $R_{X,m}$ shows that $a \equiv r\Theta_L \pmod{I^2}$ for some $r \in R$ such that $\overline{r} \Theta_H' \equiv 0 \pmod{p^m}$ in $\overline{R}$.  Yet since $\Theta_H'$ is a non-zerodivisor, 
there exists a nonnegative integer $h$ such that $\Theta_H'$ divides $p^h$ in $\overline{R}$. We therefore find $p^h \overline{r} \equiv 0 \pmod{p^m}$, whence $\overline{r} \equiv 0 \pmod{p^{m-h}}$ for $m \ge h$.  Therefore $r \in (I, p^{m-h})$ so $a \in (I^2, p^{m-h}I)$.  But $\#\Gamma$ annihilates $I/I^2$ (recall $\Gamma = \Gal(L/H)$) so for $m$ large enough $p^{m-h}I \subset I^2$, and we have $a \in I^2$ as desired.
\end{proof}

Recall that for an integer $m \ge 0$, $F_m$ denotes the $m$th layer of the cyclotomic $\Z_p$-extension of $F$.
Let $\fg_m = \Gal(LF_m/F)$ and $\Gamma_m = \Gal(LF_m/H)$.  Let $R_m = \Z_p[\fg_m]^-$ and
let \[ \bpsi_m \colon \fg_m \longrightarrow R_m^* \] denote the canonical character.  We define a ring homomorphism
\begin{equation} \label{e:betadef} \beta_m \colon R_m \longrightarrow R_{X,m}, \qquad \sigma \mapsto \sigma|_L  + \sigma|_H \cdot \log_p(\varepsilon_{\cyc}(\sigma)) X. 
\end{equation}
This is well-defined since the image of $\sigma \in \fg_m$ determines the value of $\log_p(\varepsilon_{\cyc}(\sigma))$ modulo $p^m$.

Define
\begin{equation} \label{e:nablaxdef}
 \nabla_{X, m} = {\nabla}_{\Sigma_\fp}^{\Sigma'}(L F_m)_{R_m} \otimes_{R_m} R_{X,m},
 \end{equation}
where the $R_m$-action on the right factor is given by $\beta_m$.

\begin{lemma} \label{l:fittnx}
The $R_{X,m}$-module $ \nabla_{X, m}$ is quadratically presented and 
\[ \Fitt_{R_{X,m}}(\nabla_{X,m})= (\rec_G(u_\fp^{\Sigma,\Sigma'}) - \Theta_L). \]
\end{lemma}

To prepare for the proof of Lemma~\ref{l:fittnx}, 
define $I_m = \ker(\Z_p[\fg_m]^- \longrightarrow \Z_p[G]^-)$. We have $I_m/I_m^2 \cong \Z_p[G]^- \otimes \Gamma_m$. We lift the map $\rec_G$ defined in (\ref{e:recg}) by defining
\[ 
\rec_{G,m} \colon (\cO_{H, \Sigma, \Sigma'}^*)_{\overline{R}} \longrightarrow I_m/I_m^2, \qquad \epsilon \mapsto \sum_{\sigma \in G} (\rec_{\fP,m}(\sigma(\epsilon)) - 1)\tilde{\sigma}_m^{-1},
\]
where $\rec_{\fP, m}\colon H_{\fP}^* \longrightarrow \Gamma_m$ is the reciprocity map and $\tilde{\sigma}_m$ is a lift of $\sigma$ in $\fg_m$. Under the projection $\fg_m \longrightarrow \fg$, the element $\rec_{\fP, m}(\epsilon)$ is mapped to $\rec_{\fP}(\epsilon)$. Furthermore, \[ \varepsilon_{\cyc}(\rec_{\fP,m}(\epsilon)) \equiv  \Norm_{H_{\fP}/\Q_p}(\epsilon) \text{ in } (\Z/p^m\Z)^*. \] Therefore, under the map $\beta_m$ defined in (\ref{e:betadef}), we have
\begin{equation} \label{e:recn}
\beta_m(\rec_{G,m}(\epsilon)) = \rec_G(\epsilon) + X \sum_{\sigma \in G}\overline{\bpsi}(\sigma)^{-1} \log_p  \Norm_{H_{\fP}/\Q_p}(\sigma(\epsilon)).
\end{equation}

\begin{proof}[Proof of Lemma~\ref{l:fittnx}] Lemmas~\ref{l:qp} and \ref{l:erec} (applied to $LF_m$ in place of $L$) imply that
${\nabla}_{\Sigma_\fp}^{\Sigma'}(L F_m)_{R_m}$ is  quadratically presented over $R_m$ and that
\[ \Fitt_{R_m} ({\nabla}_{\Sigma_\fp}^{\Sigma'}(L F_m)_{R_m}) = ( \rec_{G,m}(u_\fp^{\Sigma, \Sigma'})). \]
Applying (\ref{e:recn}), we then calculate
\begin{align} 
\beta_m(\rec_{G,m}(u_\fp^{\Sigma, \Sigma'})) &= \rec_G(u_\fp^{\Sigma, \Sigma'}) + X \sum_{\sigma \in G}\overline{\bpsi}(\sigma)^{-1} \log_p  \Norm_{H_{\fP}/\Q_p}(\sigma(u_\fp^{\Sigma, \Sigma'}))  \nonumber \\
& =  \rec_G(u_\fp^{\Sigma, \Sigma'}) - \Theta_H' X  \label{e:applygs} \\
& =  \rec_G(u_\fp^{\Sigma, \Sigma'}) - \Theta_L. \label{e:applyref}
 \end{align}
Equation (\ref{e:applygs}) follows from Corollary~\ref{c:ddp} and equation (\ref{e:applyref}) is one of the defining relations of $R_{X, m}$ in (\ref{e:rxdef}). The lemma follows.
\end{proof}

On the other hand, in our previous work \cite{dk} we calculated the Fitting ideal of ${\nabla}_{\Sigma_\fp}^{\Sigma'}(L F_m)$ exactly, yielding the following result:

\begin{theorem}  We have $\Fitt_{R_{X,m}}(\nabla_{X,m}) = 0$. \label{t:fittxz}
\end{theorem}

\begin{proof} Theorem 3.3 of \cite{dk}, recalled already in Theorem~\ref{t:dkresult} above, 
implies that 
\[    \Fitt_{R_m}(\nabla_{\Sigma_\fp}^{\Sigma'}(LF_m)_{R_m}) = (\Theta^{LF_m}_{\Sigma_\fp, \Sigma'}). \]
Passing to $R_{X,m}$ by applying $\beta_m$, we find
\begin{align}
\beta_m(\Theta^{LF_m}_{\Sigma_\fp, \Sigma'}) &= \Theta_L + X \sum_{\sigma \in \fg_m} \zeta_{\Sigma_\fp, \Sigma'}(\sigma, 0) \log_p\varepsilon_{\cyc}(\sigma)\overline{\sigma}^{-1} \nonumber \\
&= \Theta_L + X \sum_{\sigma \in G_m} \zeta_{\Sigma_\fp, \Sigma'}(\sigma, 0) \log_p\varepsilon_{\cyc}(\sigma)\overline{\sigma}^{-1}  \label{e:gtog} \\
&= \Theta_L - X \Theta_H'  \label{e:applyth} \\
&= 0.  \label{e:final} \end{align}
Here $\overline{\sigma}$ denotes the image of $\sigma$ in  $G$.  Equation (\ref{e:gtog}) follows from the distribution property of partial zeta functions since $\log_p \varepsilon_{\cyc}(\sigma) \overline{\sigma}^{-1}$ depends only on the image of $\sigma$ in $G_m$. Equation (\ref{e:applyth}) follows from Lemma~\ref{l:derivative}.  Equation (\ref{e:final}) is a defining relation of $R_{X,m}$.  The result follows.
\end{proof}

Combining Lemma~\ref{l:fittnx} with Theorem~\ref{t:fittxz}, we obtain \begin{equation} \label{e:recutrx}
 \rec_G(u_\fp^{\Sigma, \Sigma'}) - \Theta_L =0 \text{ in } R_{X,m}. \end{equation}
   By Lemma~\ref{l:nxi},  equation (\ref{e:recutrx}) for $m$ large enough yields 
\[ \rec_G(u_\fp^{\Sigma, \Sigma'}) \equiv \Theta_L \pmod{I^2} \]
in $R$.  This  completes the proof of the $p$-part of the modified Gross conjecture (\ref{e:mgc}) in Case 2(c).
Lemma~\ref{l:equiv} then yields the $p$-part of Gross's conjecture for the Brumer--Stark unit $u_\fp$:
\[ \rec_G(u_\fp) \equiv  \Theta_{S_\fp,T}^{L/F} \pmod{I^2}. \]
This completes the proof of Theorem~\ref{t:maingross}.

\end{document}